\documentclass[a4paper,oneside,11pt]{amsart}

\input{preamble.sty}

\begin{document}

\title[Conformal groups of Lorentzian manifolds]{The conformal group of a compact simply connected Lorentzian manifold}
\author{Karin Melnick and Vincent Pecastaing}
\date{\today}

\begin{abstract}
We prove that the conformal group of a closed, simply connected, real analytic Lorentzian manifold is compact.  D'Ambra proved in 1988 that the isometry group of such a manifold is compact \cite{dambra.lorisom}.  Our result implies the Lorentzian Lichnerowicz Conjecture for real analytic Lorentzian manifolds with finite fundamental group.
\end{abstract}

\thanks{MSC 53C50, 57S20 \\
This project was initiated during an extended visit by Pecastaing to the University of Maryland supported by GEAR Network grant DMS-1107452, 1107263, 1107367 from the NSF.  Melnick gratefully acknowledges support from NSF grant DMS-1255462 and from the Max-Planck-Institut f\"ur Mathematik in Bonn, where she was a Visiting Scientist during much of the research for this paper.  Pecastaing was also partially supported by FNR grants INTER/ANR/15/11211745 and OPEN/16/11405402 while a Postdoctoral Researcher at the University of Luxembourg during much of the research and the writing of this paper.  We thank A. Zeghib for suggesting this problem many years ago.  We thank A. Chatzistamatiou and C. Frances for helpful conversations.}

\maketitle

\tableofcontents

\setlength{\parskip}{1ex}

\section{Introduction}
Transformation groups of manifolds equipped with geometric structures depend on the topology of the manifold and the nature of the geometric structure.  An ambitious program initiated by Zimmer and Gromov in the 1980s aims to roughly classify compact manifolds with rigid geometric structures admitting a noncompact transformation group \cite{gromov.rgs, dag.rgs}.  Lorentzian manifolds are intriguing in this context, as they exhibit some unique phenomena among semi-Riemannian signatures.  In this paper, we prove:

\begin{theorem}
\label{thm:main}
If $(M,g)$ is a closed, simply connected, real analytic Lorentzian manifold, then its conformal group is compact.
\end{theorem}



An immediate corollary is that for $(M,g)$ closed, real-analytic Lorentzian, with finite fundamental group, $\Conf(M,g)$ is compact, because the conformal group of the universal cover of $M$ is an extension of $\Conf(M,g)$ by $\pi_1(M)$.
There exist a multitude of closed, analytic Lorentzian manifolds with infinite fundamental group having noncompact conformal group.  
Two important examples appearing below are the Lorentzian Einstein space $\Ein^{1,n-1}$ (see example \ref{ex:einstein_universe} and section \ref{sec:cartan_connxn_holonomy}) and the Hopf manifolds of example \ref{ex:hopf_manifold}.
See \cite{frances.lorentz.klein} for more examples, of infinitely-many distinct topological types, in all dimensions at least 3.

\begin{example}
Real analytic Lorentzian metrics on closed, simply connected manifolds can be constructed as follows. Choose a nonvanishing analytic vector field $X$ on an odd-dimensional sphere ${\bf S}^{2n+1}$, equipped with any Riemannian metric $g_R$. Let $H$ be the orthogonal distribution to $X$ with respect to $g_R$.  Then, define an analytic Lorentzian metric $g_L$ by declaring $X$ and $H$ orthogonal with respect to $g_L$; $g_L(X,X)=-1$; and  $g_L = g_R$ on $H \otimes H$.  We can moreover form the product of $({\bf S}^{2n+1},g_L)$ with any closed, simply connected, analytic  Riemannian manifold.
\end{example}

\subsection{D'Ambra's Theorem on Lorentzian isometries}

An important result in the Zimmer-Gromov program, as well as an important motivation for our work, is the following theorem of D'Ambra from 1988:

\begin{theoremA*}[D'Ambra 1988 \cite{dambra.lorisom}]
If $(M,g)$ is a closed, simply connected, real-analytic Lorentzian manifold, then its isometry group is compact.
\end{theoremA*}

The isometry group of a compact Riemannian manifold is always compact, by the classical theorems of Myers and Steenrod, but compact pseudo-Riemannian manifolds can have noncompact isometry groups.  For example, the bi-invariant Cartan-Killing metric on $\SL(2,\BR)$ is Lorentzian, and descends to the quotient by any cocompact lattice, on which $\PSL(2,\BR)$ survives as isometry group.  Noncompact, connected isometry groups of closed Lorentzian manifolds have been classified by Adams and Stuck \cite{as.lorisom1, as.lorisom2} and Zeghib \cite{zeghib.lorisom1, zeghib.lorisom2}. 

D'Ambra's theorem is not true for higher signature: also in \cite{dambra.lorisom} she constructed an analytic metric of signature $(2,7)$ on ${\bf S}^3 \times {\bf S}^3 \times {\bf S}^3$ with noncompact isometry group.  As $\Conf(M,g)$ always contains $\Isom(M,g)$, her example shows that our theorem does not hold in higher signature, either.  Another such example appears in the next section.

D'Ambra's proof is a \emph{tour de force} of Gromov's theory of rigid geometric structures, which has maximum strength for \emph{compact, simply connected} manifolds with \emph{real analytic}, \emph{unimodular}, rigid geometric structure (of algebraic type) \cite{gromov.rgs}.  The finite volume left invariant by $\Isom(M,g)$ plays a fundamental role in D'Ambra's proof.  A group of conformal transformations of a closed manifold preserves a finite volume if and only if it is isometric for a metric in the conformal class, in which case it is called \emph{inessential}.  D'Ambra's theorem thus says we may assume $\Conf(M,g)$ is essential, but it provides only limited assistance beyond that.

We do not know whether our theorem or D'Ambra's theorem holds for smooth metrics.



\begin{example} 
\label{ex:hopf_manifold}
The Lorentzian \emph{Hopf manifold} is the quotient of Minkowski space minus the origin by the similarities $\{ 2^k \cdot \mbox{Id} \}_{k \in \BZ}$; it is diffeomorphic to $S^1 \times S^{n-1}$.  The similarity group of $\Min^{1,n-1} \backslash \{ 0 \}$ is $\BR^* \times \SO(1,n-1)$ (modulo $- \mbox{Id}$ if $n$ is even), which descends to $S^1 \times \mbox{O}(1,n-1)$ on the quotient.  All point stabilizers in the Hopf torus are linearizable with unimodular isotropy, but no noncompact subgroup of $S^1 \times \mbox{O}(1,n-1)$ preserves a global volume.
\end{example}

\subsection{Conformal transformations of compact Lorentzian manifolds}
\label{sec:intro_llc}

Another motivation for our work is the:

\begin{conjecture*}[Lorentzian Lichnerowicz Conjecture]
If $(M,g)$ is a closed Lorentzian manifold of dimension at least 3 admitting an essential conformal group, then $(M,g)$ is conformally flat.
\end{conjecture*}

In Riemannian signature, a stronger result, the Ferrand-Obata Theorem \cite{lf.lich, obata.lich} holds.  In higher signature $(p,q)$, with $\mbox{min} \{p,q \} \geq 2$, the above statement is false, by examples of Frances \cite{frances.pq.counterexs}.  The following conformally homogeneous spaces are the basic examples of compact semi-Riemannian manifolds with essential conformal group (The Ferrand-Obata Theorem says it is the only example in Riemannian signature.)

\begin{example}
\label{ex:einstein_universe}
For $p,q$ between $0$ and $n$, and $p+q=n$, the conformal semi-Riemannian spaces $({\bf S}^p \times {\bf S}^q, -g_{{\bf S}^p} \oplus g_{{\bf S}^q} ) $ admit $O(p+1,q+1)$ as conformal group and are conformally flat.  The quotients by $(-\mbox{Id} \times -\mbox{Id})$ are projective varieties called $\Ein^{p,q}$, the \textit{semi-Riemannian Einstein spaces}.  Only the Lorentzian Einstein space has infinite fundamental group, and our Theorem \ref{thm:main} does not hold in any other signature.
\end{example}


A consequence of the special topology of $\Ein^{1,n-1}$ is the following:
\begin{propositionB*}
\label{prop:non_flat}
A compact, simply connected Lorentzian manifold cannot be conformally flat.
\end{propositionB*}

A simply connected, conformally flat Lorentzian manifold $M$ has, by a classical monodromy argument, a developing map $f: M \rightarrow \Ein^{1,n-1}$, which is in particular a local diffeomorphism.  If $M$ is moreover closed, then $f$ is a covering map---a contradiction.

The $2$-sphere admits no Lorentzian metric.
 Our main result partially confirms the Lorentzian Lichnerowicz Conjecture, thanks to D'Ambra's Theorem A and this proposition B.
Previous global results supporting the Lorentzian Lichnerowicz Conjecture, and related global rigidity results on conformal transformations in other signatures, are for algebraically large groups,  such as semisimple groups \cite{bn.simpleconf, fz.simpleconf, pecastaing.smooth.sl2r, pecastaing.ssconf} or nilpotent groups of a maximal nilpotence degree \cite{fm.nilpconf}.

A previous local rigidity result will play an important role below.

\begin{theorem*}[Frances--Melnick 2013 \cite{fm.champsconfs}]
Let $(M,g)$ be a real-analytic Lorentzian manifold of dimension at least $3$.  Let $X$ be an analytic conformal vector field on $M$, vanishing at a point $p$.  Then at least one of the following holds:
\begin{itemize}
\item $X$ is analytically linearizable in a neighborhood of $p$
\item $(M,g)$ is conformally flat.
\end{itemize}
\end{theorem*}
This result is used repeatedly in combination with proposition B in our proof.

The Riemannian Lichnerowicz Conjecture even holds for $M$ noncompact \cite{lf.noncompact}.  The following examples illustrate why the Lorentzian conjecture does not apply to noncompact manifolds.

\begin{example}
\label{ex:alekseevsky}
In \cite{alekseevski.selfsim}, Alekseevsky exhibits a family of non-conformally flat Lorentzian metrics on $\BR^n$ for which the diagonal flow $\diag(e^{2t},e^t,\ldots,e^t,1)$ acts conformally and essentially. For all these metrics, the fixed line $\BR e_n$ is isotropic and the invariant hyperplanes $\{ \mathcal{F}_c =  e_1^\perp + c e_n \}$ are degenerate and totally geodesic.  
 \end{example}
 In the course of our proof, we will show that these examples do not conformally embed in a closed, simply connected Lorentzian manifold.
 
\subsection{Dynamical foliations in semi-Riemannian geometry}

In Alekseevsky's examples \ref{ex:alekseevsky}, the family of totally geodesic hypersurfaces $\{ \mathcal{F}_t \}$ are the strong-unstable leaves of the flow.  Neither isometric nor conformal transformations of closed Lorentzian manifolds need be hyperbolic in any sense.  Nonetheless Zeghib in \cite{zeghib.tgl1, zeghib.tgl2}, following a suggestion of D'Ambra and Gromov \cite[Sec 7.4]{dag.rgs}, proved that any unbounded sequence of \emph{isometries} of a compact Lorentzian manifold gives rise to an \emph{approximately stable foliation} by totally geodesic, degenerate hypersurfaces.  When the metric is real analytic, then his foliation is real analytic.  Zeghib thus gave an alternate proof of D'Ambra's Theorem A, thanks to the following theorem:

\begin{theoremC*}[Haefliger \cite{haefliger.codim1}]
A closed, simply connected, real analytic manifold $M$ does not admit a real analytic, codimension-one foliation.
\end{theoremC*}
(See also \cite[Cor IV.1.9]{godbillon.book}.)
Zeghib's construction relies essentially on the Levi-Civita connection of the metric; in particular, the fact that a totally geodesic lamination is Lipschitz is needed to extend the foliation over the whole manifold.  In the conformal setting, no such global foliations exist in general.  For sequences of  conformal transformations of semi-Riemannian manifolds which are \emph{stable} near a point, Frances has constructed dynamical foliations on a neighborhood of the point in \cite{frances.degenerescence}, which have proved quite useful.

 \subsection{About our proof}

 D'Ambra's proof of Theorem A and the results of Proposition B and Haefliger's Theorem C are three anchors of our proof.  Between them lay several difficulties.  D'Ambra's original proof breaks down in the absence of an invariant volume.  Existing approaches to global rigidity results for conformal Lorentzian transformation groups require some nontrivial algebraic structure from the groups, which is absent in our problem.  Global dynamical foliations typically require hyperbolicity or an invariant affine connection.

Our proof proceeds as follows.  In section \ref{sec:dambra_consequences} we recall Gromov's stratification results and reduce to a maximal connected, abelian subgroup $H \leq \Conf(M,g)$.  Thanks to the normal forms theorem of \cite{fm.champsconfs} stated above, we can apply D'Ambra's arguments to $H$ whenever it preserves a volume on a nonempty open subset of $M$.  Then we can assume a certain ubiquity of sequences with conformal dilation tending to zero.

A crucial tool for classifying the possible local dynamics of conformal transformations on our Lorentzian manifold and for thoroughly analyzing the geometric implications of those dynamics is the Cartan connection canonically associated to a conformal Lorentzian structure.  In section \ref{sec:cartan_connxn_holonomy}, we recall the definition of this Cartan geometry and of the \emph{holonomy sequences} associated to sequences of conformal transformations.  
We present an analysis of the Weyl curvature module and prove an algebraic result which later allows us to define distributions of degenerate hyperplanes via the Weyl curvature, provided it takes certain special values.
Then we illustrate the application of holonomy and stability by proving theorem \ref{thm:main} given the existence of a \emph{contracting} holonomy sequence.

In section \ref{sec:balanced_proof}, we prove the theorem given the existence of a \emph{balanced, linear} holonomy sequence (see definition \ref{def:holonomy_trichotomy}; these are like the flow in example \ref{ex:alekseevsky}).  Such sequences force the Weyl curvature to have values in a very special subbundle, which in turn implies local integrability of certain distributions on the Cartan bundle, the leaves of which are stable submanifolds for the initial sequence of conformal transformations.  These leaves have a geometric interpretation which allows us to construct a foliation by degenerate hypersurfaces on an open, dense set.  Nonextendability of this foliation corresponds to vanishing of an isotropic conformal vector field, which implies conformal flatness by a theorem of Frances \cite{frances.ccvf}.  Otherwise, the foliation extends, which contradicts Haefliger's Theorem C.  The construction of the foliation in this section is a miniature of our construction of a codimension-one foliation on an open, dense set in the final section of the proof.

In section \ref{sec:fixed_points} we show that there are no $H$-fixed points.  Of course, the local result of \cite{fm.champsconfs} is key, as is Gromov's Stratification.  We show that the full isotropy is linearizable.
As suggested by example \ref{ex:hopf_manifold}, it is not possible to conclude only from studying the action in the neighborhood of one fixed point.  However, we can choose a particular orbit nearby and follow it to another fixed point in the orbit closure.  The interplay of the $H$-action around these two fixed points leads to a contradiction.

We begin
section \ref{sec:end_of_proof} by carefully analyzing the remaining possibilities for holonomy sequences at points with nonclosed orbits.  They 
are not stable, nor are they necessarily linear.  There is nonetheless a strong interplay of the dynamics, the curvature, the algebra governing the Cartan connection, and the conformal geometry, by which we construct a foliation by invariant, degenerate hypersurfaces on an open, dense set.
We show that the resulting foliation extends over closed, isotropic orbits with linear, unipotent isotropy. By the conclusion of section \ref{sec:dambra_consequences}, a dense set of points has such an orbit in its orbit closure.  A contradiction is reached by showing the set of closed, isotropic orbits is covered by finitely-many such extended leaves, which leads to a lightlike conformal vector field in $M$ and then to conformal flatness by \cite{frances.ccvf}.  This last section is the most difficult part of the proof, in which we must work with multiple particularly challenging features of conformal Lorentzian dynamics: instability, nonlinearity, and unipotent dynamics.

\section{Consequences of D'Ambra's proof and the case of an invariant volume}
\label{sec:dambra_consequences}

Throughout this section, let $(M,g)$ be a compact, simply-connected, real-analytic Lorentzian manifold of dimension at least $3$.  Here we recall relevant results for such manifolds from Gromov's theory of rigid geometric structures.  These results drive D'Ambra's proof of Theorem A above.  Existence of an invariant volume for Lorentzian isometry groups also plays a crucial role in her paper.  Here we adapt her arguments to the conformal setting to the extent possible, which leaves us with condition (\ref{eqn:negation_finite_volume}), a statement about nonexistence of certain volumes, for the sequel.

\subsection{Gromov's stratification and reduction to a connected, abelian group}
\label{sec:stratification_abelian}

The conformal class $[g]$ is a rigid geometric structure in the sense of \cite{gromov.rgs}, which yields very strong conclusions on $\Conf(M,g)$ and the structure of its orbits.  These conclusions also hold for a maximal connected, abelian subgroup of $\Conf(M,g)$; D'Ambra observed that showing compactness of any such subgroup suffices to prove compactness of the full automorphism group.  Let $\{ X_1, \ldots, X_r \}$ be a basis of a maximal abelian subalgebra of $\Kill^{conf}(M,g)$, the conformal vector fields of $(M,g)$.  Let 
$$ H' = \{h \in \Conf(M,g) \ : \  \ h_* X_i = X_i \ \forall i \}$$
Now $H'$ is the automorphism group of the augmented rigid geometric structure $[g] \cup \{X_1, \ldots, X_r\}$ (see \cite[Sec 3.6]{gromov.rgs}).  

\begin{theorem}[Gromov \cite{gromov.rgs} 3.5.B, 3.5.C]
\label{thm:finite_components}
The groups $\Conf(M,g)$ and $H'$ each have finitely many connected components.  For any $x \in M$, the stabilizers of $x$ in $\Conf(M,g)$ and $H'$ each have finitely many connected components.
\end{theorem}
(Gromov in fact proved that the stabilizers above are algebraic, in a sense explained in section \ref{sec:isotropy_jordan_decomp} below.)
To show compactness of $\Conf(M,g)$, it thus suffices to prove compactness of the connected component of the identity $\Conf^0(M,g)$.   In order to prove 
the latter property, it suffices to show compactness of the identity component $(H')^0 = H$ (see \cite[Sec 7]{dambra.lorisom}).  Note that $H$ is abelian.

Here are several consequences of Gromov's Stratification Theorem which will be used in the sequel.  A proof of point (1) for $G = \Conf^0(M,g)$ using the Cartan connection associated to the conformal structure is in \cite[Thm 4.1]{me.frobenius}; the proof for $H$ can be obtained by combination with the proof of \cite[Thm 4.19]{pecastaing.frobenius}.

\begin{theorem}[Gromov \cite{gromov.rgs} 3.5.A, 3.2, 3.5.C]
\label{thm:gromov_stratification}
Let $G = \Conf^0(M,g)$ or $H$.  
\begin{enumerate}
\item There is an open, dense subset $\Omega \subseteq M$ and a smooth map of locally constant corank from $\Omega$ to a smooth manifold for which the connected components of the fibers are $G$-orbits.  The complement of $\Omega$ is an analytic subset of $M$.
\item For all $x \in M$, the orbit $G.x$ is locally closed and there is a closed $G$-orbit in the closure $\overline{G.x}$.  
\item For all $x \in M$, the closure $\overline{G.x}$ is semianalytic and locally connected. 
\end{enumerate}
\end{theorem}

Recall that a subset $S$ of a topological space is \emph{locally closed} if $S$ is open in the closure $\overline{S}$.
A semianalytic subset of an analytic manifold is one locally cut out by finitely many analytic equalities and inequalities; see \cite{bierstone.milman} for properties of these sets.  The closure of a semianalytic set is again semianalytic.  Semianalytic sets are locally closed; by corollary 2.7 in \cite{bierstone.milman}, they are locally connected.

\subsection{Free action on an open, dense subset}

D'Ambra's local result \cite[Cor 3.3]{dambra.lorisom} can be proved in our conformal setting, thanks to the normal forms for conformal vector fields of \cite{fm.champsconfs}.

Let $H$ be the connected, abelian subgroup of $\Conf(M,g)$ defined in the previous section.  Let $L$ be a maximal compact subgroup of $H$, so $H \cong  L \times \BR^k$ for some $k \in \BN$.  

\begin{proposition}
\label{prop:freeness} 
There exists an open, dense subset $\Omega_f$, with analytic complement in $M$, on which $H$ acts freely.
\end{proposition}

\begin{proof}
Let $\Omega_f$ be the open, dense subset given by theorem \ref{thm:gromov_stratification} (1), and let $p$ be the smooth map on $\Omega_f$ with fibers equal the $H$-orbits in $\Omega_f$.  
  
Let $\Stab_H(x)$ be the stabilizer of $x \in \Omega_f$.  As in D'Ambra's proof, we observe that 
for any $h \in \Stab_H(x)$, the differential $D_xh$ acts identically on $T_x(H.x)$ since $H$ is abelian.
Using that the $H$-orbits exactly equal the fibers of the smooth map $p$ on an $H$-invariant neighborhood of $x$, we conclude that $D_xh$ is trivial on $T_xM / T_x(H.x)$ for all $h \in \Stab_H(x)$.
The following lemma is also valid in the conformal setting, and implies $D_x h = \mbox{Id}$ for any $h \in \Stab_H(x)$:

\begin{lemma}[\cite{dambra.lorisom} 3.2]
Let $f$ be a linear conformal transformation of $\BR^{1,n-1}$ and $V \subseteq \BR^{1,n-1}$ a vector subspace such that $f$ acts identically on both $V$ and $\BR^{1,n-1} / V$. Then, $f = \mbox{Id}$.
\end{lemma}

Triviality of $D_xh$ for all $h \in \Stab_H^0(x)$ implies $M$ is conformally flat by \cite{fm.champsconfs}, contradicting Proposition B, or $\Stab_H^0(x) = 1$.  
Theorem \ref{thm:finite_components} implies $\Stab_H(x)$ is finite.  Any torsion in $H$ is contained in $L$, and preserves a Lorentzian metric in $[g]$.  Now $\Stab_H(x)$ is linearizable, and therefore trivial, in a neighborhood of $x$.  We conclude $\Stab_H(x) = 1$.
\end{proof}

\subsection{Proof \emph{\`a la D'Ambra} given a finite, invariant volume on an open subset}

\begin{proposition}
\label{prop:proof_with_finite_volume}
If there is a finite, $H$-invariant volume on an $H$-invariant open subset $S \subseteq M$, then $H$ is compact.
\end{proposition}

\begin{proof}
Let $\Omega = S \cap \Omega_f$, where $\Omega_f$ is as in proposition \ref{prop:freeness}.  Then the conclusion of \cite[Prop 3.7.B]{gromov.rgs} holds for the $H$-action on $\Omega$:  for all  points $x$ in this set, $H.x = L.x$.  On the other hand, the stabilizer of $x$ is trivial.  Then $H = L$.   
\end{proof}

\subsection{Sufficient condition for existence of such a volume}
\label{subsec:volume_condition}

Fix a finite, smooth, $L$-invariant volume $\nu$ on $M$.  Given $h \in H$, we define $\Delta_h(x) = (\det D_x h)^{1/n}$, where the determinant is with respect to $\nu$.  The cocycle relation $\Delta_{gh}(x) = \Delta_g(h.x) \cdot \Delta_h(x)$ is satisfied by $\Delta$.  Note that $\Delta_g \equiv 1$ for all $g \in L$.

Let $\Omega_f$ be as in proposition \ref{prop:freeness}.  Let $p : \Omega_f \rightarrow X$ be the smooth map given by theorem \ref{thm:gromov_stratification}.  For any $x_0 \in \Omega_f$, there is a compact local transversal $\overline{\Sigma}$ to $p$ containing $x_0$, mapping diffeomorphically to its image under $p$. Let $\Sigma$ be the interior of $\overline{\Sigma}$. There is a well-defined projection $\rho: H.\Sigma \rightarrow \Sigma$.

A coboundary equal to $\Delta$ on $H.\Sigma \subset \Omega_f$ is given by 
$$b(y) = \Delta_h(\rho(y)), \  \mbox{where} \ y = h.\rho(y)$$
The function $b$ is well-defined because $H$ acts freely on $H.\Sigma$ by proposition \ref{prop:freeness}.
Now $b^{-n} \cdot \nu$ is an $H$-invariant volume on the open set $H.\Sigma$; it is finite if $b^{-n} \in \mathcal{L}^1(H.\Sigma, \nu)$.
 
Of course, $b^{-n} \in \mathcal{L}^1(H.\Sigma, \nu)$ if it is bounded, because $\nu(M) < \infty$.  
If there is $C \in \BR$ bounding $\Delta_h(x)^{-n}$ from above for all $x \in \Sigma$ and $h \in H$
then $b^{-n} \in \mathcal{L}^1(H.\Sigma, \nu)$, and there is a finite, $H$-invariant volume on $H.\Sigma \subseteq M$.
Thus \emph{if there is no finite, $H$-invariant volume on any open subset of $M$, then there is a dense set of points $S_{mix} \subset \Omega_f$, such that for all $x \in S_{mix}$}, 
\begin{equation}
\label{eqn:negation_finite_volume}
\exists \  x_k \rightarrow x, h_k \in H \ \mbox{such that} \ \Delta_{h_k}(x_k) \rightarrow 0
  \end{equation}
  By proposition \ref{prop:proof_with_finite_volume}, we may henceforth assume that this is the case.
 
\section{Holonomy sequences and proof in the contracting case}
\label{sec:cartan_connxn_holonomy}

In this section we realize our conformal Lorentzian structure as a Cartan geometry infinitesimally modeled on $\Ein^{1,n-1}$.  We recall the notion of holonomy sequences and how it captures the local dynamics of a sequence of conformal transformations.  Section \ref{sec:cartan_geometry} also includes a study of the curvature module for conformal Lorentzian structures and establishes some algebraic results that will be used in the sequel.  In section \ref{sec:contracting_proof} we illustrate the application of these tools by proving conformal flatness in the presence of a uniformly contracting sequence of conformal transformations.

\subsection{The Cartan geometry associated to a conformal Lorentzian structure}
\label{sec:cartan_geometry}

\subsubsection{The equivalence principle for conformal structures}
\label{sec:equivalence_principle}

Let $\BR^{2,n}$ be $\BR^{n+2}$ with standard basis $\{ e_0, \ldots, e_{n+1} \}$, equipped with the quadratic form
$$Q_{2,n} := 2x_0x_{n+1} + 2x_1x_{n}+ x_2^2 + \cdots + x_{n-1}^2.$$
The Lorentzian Einstein Universe, denoted $\Ein^{1,n-1}$, is the projectivization of the nullcone 
$$\mathcal{N}^{2,n} \setminus \{0\} = \{ {\bf x} \in \BR^{n+2} \setminus \{0\} \ | \ Q_{2,n}({\bf x}) = 0\}$$

 It is a smooth quadric hypersurface of ${\bf R P}^{n+1}$, that naturally inherits a conformal class $[g_{1,n-1}]$ of Lorentzian signature from the ambient quadratic form of $\BR^{2,n}$. It is diffeomorphic to ${\bf S}^1 \times_\iota {\bf S}^{n-1}$, where $\iota$ is the inversion on both factors (in particular, it has infinite fundamental group). By construction, there is a transitive conformal action of $G = \PO(2,n)$ on $\Ein^{1,n-1}$, and in fact $\Conf(\Ein^{1,n-1},g_{1,n-1}) \cong G$. Thus, $\Ein^{1,n-1}$ is a compact, conformally homogeneous space, identified with $G/P$, where $P$ is the parabolic subgroup stabilizing an isotropic line in $\BR^{2,n}$ (modulo $\pm \Id$). It is the \textit{model space} of conformal Lorentzian geometry in the following sense.

\begin{theorem}[\'E. Cartan, see  \cite{sharpe} Ch V, \cite{cap.slovak.book.vol1} Sec 1.6]
Let $M^n$ be a connected manifold of dimension $n \geq 3$.  A conformal Lorentzian structure on $M$ canonically determines 
\begin{itemize}
\item a principal $P$-bundle $\pi: \widehat{M} \rightarrow M$; and 
\item a regular, normal Cartan connection $\omega \in \Omega^1(\widehat{M},\lieg)$ satisfying, for all $\hat{x} \in \widehat{M}$,
\begin{enumerate}
 \item $\omega_{\hat{x}} : T_{\hat{x}} \widehat{M} \stackrel{\sim}{\rightarrow} \lieg$
 \item $\omega_{\hat{x}.g} \circ R_{g*} = \mbox{Ad}(g^{-1}) \circ \omega_{\hat{x}} \ \ \forall g \in P$
 \item $\omega \left( \DDt (\hat{x}.e^{tY}) \right) \equiv Y \ \  \forall \ Y \in \liep$
\end{enumerate}
\end{itemize}
\end{theorem}
See the references above for details on the regularity and normality conditions on $\omega$; they do not play a prominent role in our work below.

The bundle $\widehat{M}$ admits an interpretation as a subbundle of the 2-frame bundle of $M$, comprising 2-jets of infinitesimal conformal equivalences to Minkowski space.  The projection of $\widehat{M}$ to the 1-frame bundle of $M$ is the bundle of frames compatible with a fixed normalization with respect to the conformal metric (see \cite[p 139] {kobayashi.transf}).

Conformal transformations of $M$ lift naturally and uniquely to bundle automorphisms of $\widehat{M}$ leaving $\omega$ invariant; conformal vector fields similarly lift to $P$-invariant vector fields on $\widehat{M}$ with Lie derivative annihilating $\omega$.  The uniqueness of lifts comes from $\Ein^{1,n-1}$ being an \emph{effective} $\PO(2,n)$-homogeneous space (see \cite[Def IV.3.2]{sharpe}, \cite[Prop 3.6]{me.frobenius}).  We will use the same notation for conformal transformations or vector fields and their respective lifts to $\widehat{M}$.  The lifted action of $\Conf(M,g)$ preserves the parallelization of $\widehat{M}$ determined by $\omega$, so it is free; similarly, lifts of conformal vector fields to $\widehat{M}$ are nonvanishing.  

The following formula follows from the axioms for $\omega$ (see, for example, \cite[Lem 5.4.12]{sharpe}): let $\hat{\gamma}(s)$ be a curve in $\widehat{M}$ and $p(s)$ a curve in $P$.  Then 
\begin{equation}
\label{eqn:omega_along_curves}
\omega \left( \DDs \hat{\gamma}(s) .p(s) \right) = \Ad p(s)^{-1} \left( \omega (\hat{\gamma}'(s)) \right) + \omega_P \left( p'(s) \right)
\end{equation}
where $\omega_P$ is the left-invariant Maurer-Cartan form of $P$.

\subsubsection{Explicit root-space decomposition of $\so(2,n)$}
\label{sec:cartan_decomposition}

Let $G$ and $P$ be as above.  The details of the structure of $\lieg$ are used extensively in our proof below.


Let $\mathbb{I}_{2,n}$ be the inner product determined by $Q_{2,n}$ from the previous section, and $\mathbb{I}_{1,n-1}$, or simply $\mathbb{I}$, its restriction to the Minkowski subspace $e_0^\perp \cap e_{n+1}^\perp$.
The Lie algebra $\lieg= \so(2,n)$ can be parametrized
\begin{equation*}
\lieg= 
\left \{
\begin{pmatrix}
a & \xi & 0 \\
v & X & - \mathbb{I} ^t \xi \\
0 & - ^t v \mathbb{I}  & -a
\end{pmatrix}
, \ a \in \BR, \ v \in \BR^n, \ \xi \in \BR^{n*}, \ X \in \so(\BR^n,\mathbb{I}) \cong \so(1,n-1)
\right \}
\end{equation*}
%

This decomposition yields the grading $\lieg= \lieg_{-1} \oplus \lieg_0 \oplus \lieg_1$, where the summands are parametrized by $v$, $(a,X)$, and $\xi$, respectively, (see \cite[p 118]{cap.slovak.book.vol1}) and  $\liep = \lieg_0 \ltimes \lieg_1$.  The ideal $\lieg_1 \lhd \liep$ is often denoted $\liep^+$.  The corresponding subgroups of $P$ are denoted $P^+$ and $G_0$.  We have $P \cong G_0 \ltimes P^+$, and $G_0 \cong \CO(1,n-1)$.

We can similarly decompose the $\so(1,n-1)$ factor:
\begin{equation*}
\so(1,n-1)= 
\left \{
\begin{pmatrix}
b & U_+ & 0 \\
U_- & R & - ^tU_+   \\
0 & - ^{t}U_- & -b
\end{pmatrix}
, \ b \in \BR, \ U_- \in \BR^{n-2}, \ U_+ \in (\BR^{n-2})^*, \ R \in \so(n-2)
\right \}.
\end{equation*}
An $\BR$-split Cartan subalgebra in $\lieg$, with respect to the Cartan involution $\theta(Y) = - ^{t} Y$ is
\begin{equation}
\label{eqn:cartan_subalgebra}
\liea =
\left \{
\begin{pmatrix}
a & & & & \\
  & b & & & \\
  & & {\bf 0} & & \\
  & & & -b & \\
  & & & & -a
\end{pmatrix}
,\ a,b \in \BR 
\right \}, 
\qquad \mbox{where} \ {\bf 0} = (0, \ldots, 0) \ n-2 \ \mbox{times}
\end{equation}

The real simple roots $\gamma$ and $\beta$ are given by $\gamma(a,b) = a-b$ and $\beta(a,b) = b$, where $(a,b)$ refers to the corresponding matrix of $\liea$. The root $-\alpha = -(\gamma + \beta)$ gives the infinitesimal conformal dilation of an element of $\liea$ in the representation by similarities on $\lieg_{-1} \cong \BR^{1,n-1}$.  The following diagram represents the full restricted root-space decomposition with respect to $\liea$ of $\so(2,n)$:
\begin{equation}
\label{eqn:root_decomposition}
\begin{pmatrix}
\liea & \lieg_{\alpha - \beta} & \lieg_{\alpha} & \lieg_{\alpha+\beta} & 0 \\
\lieg_{\beta - \alpha}   & \liea & \lieg_{\beta} & 0 & \lieg_{\alpha+\beta} \\
  \lieg_{-\alpha} & \lieg_{- \beta} & \liem & \lieg_{\beta} & \lieg_{\alpha }  \\
 \lieg_{- \alpha - \beta}  & 0 & \lieg_{- \beta} & \liea & \lieg_{\alpha - \beta} \\
  0  & \lieg_{- \alpha - \beta} & \lieg_{- \alpha} & \lieg_{\beta - \alpha} & \liea
\end{pmatrix}
\qquad 
\begin{array}{l}
\liem \cong \so(n-2) \\
\dim \lieg_{\beta} = n-2 = \dim \lieg_{\alpha} \\
\dim \lieg_{ \alpha- \beta} = 1 = \dim \lieg_{\alpha + \beta}
\end{array}
\end{equation}
 The factor $\liem$ is the centralizer of $\liea$ in a maximal compact subalgebra of $\so(2,n)$ and is parametrized by $R$ in the decomposition of $\so(1,n-1)$.  We have $\liep^+ = \lieg_{\alpha - \beta} \oplus \lieg_{\alpha} \oplus \lieg_{\alpha + \beta}$, which is equivalent as a $\CO(1,n-1)$-representation to $\BR^{1,n-1*}$.  For later use, we fix a basis $E_1, \ldots, E_n$ of $\lieg_{-1}$ which is aligned with the root-space decomposition above and normalized with respect to $\mathbb{I}_{1,n-1}$ under the $G_0$-equivariant isomorphism with $\BR^{1,n-1}$.

The following lemma establishes nondegeneracy of brackets between opposite root spaces (it holds in any real semisimple Lie algebra, provided $2 \lambda$ is not a restricted root).
\begin{lemma}
\label{lem:bracket_nondegeneracy}
  Let 
$\theta$ be a Cartan involution of $\lieg$ compatible with the Cartan subalgebra $\liea$ above, and denote
$\Sigma \subset \liea^*$ the restricted roots. 
For all $\lambda \in \Sigma$ and for all nonzero $X,Y \in \lieg_{\lambda}$, the bracket $[X,\theta Y] \neq 0$.
\end{lemma}

\begin{proof}
We use the following formula from the proof of Lemma 2.4 of \cite{pecastaing.rank1}:
\begin{equation*}
 [[X,\theta Y],X] = 2|\lambda|^2 B_{\theta}(X,Y)X - |\lambda|^2 B_{\theta}(X,X) Y \qquad \forall \ X,Y \in \lieg_{\lambda} 
\end{equation*}
 where $B_\theta(X,Y) = - B(X,\theta Y)$ for $B$ the Killing form.  Assume $[X,\theta Y] = 0$. If $X$ and $Y$ were linearly independent, then the above formula would imply $B_{\theta}(X,X)=0$, contradicting the fact that $B_{\theta}$ is positive definite. If $X$ and $Y$ were collinear and both nonzero, the formula would then give $B_{\theta}(X,X)X = 0$, again a contradiction.
\end{proof}

\subsubsection{Cartan curvature and harmonic curvature}
\label{sec:cartan_curvature}

The \emph{Cartan curvature} is the obstruction to the local equivalence of $(\widehat{M} \stackrel{\pi}{\rightarrow} M,\omega)$ with the model $(G \rightarrow \Ein^{1,n-1},\omega_G)$; local equivalence is simply conformal flatness---local conformal equivalence of $(M,g)$ with the flat Minkowski space $\Min^{1,n-1}$.  The curvature is defined on $\widehat{M}$ by 
$$ \Omega = d \omega + \frac{1}{2} [\omega, \omega]$$
It is a semi-basic 2-form, which we will often view as a $P$-equivariant function
$$ \kappa : \widehat{M} \rightarrow \wedge^2 (\lieg/\liep)^* \otimes \lieg$$
where $P$ acts on the target vector space by the natural representation built from the adjoint $\mbox{Ad}_\lieg P$.  The curvature is invariant by all automorphisms, as well as local automorphisms, of the geometry.

For conformal structures, the regularity condition on $\omega$ implies that $\kappa$ in fact has values in the $P$-submodule $\wedge^2 (\lieg/\liep)^* \otimes \liep$.  The $\emph{Weyl curvature}$ can be viewed in this context as a component of $\kappa$, called the \emph{harmonic curvature}, having values in an irreducible $P$-submodule of $\wedge^2 (\lieg/\liep)^* \otimes \liep$ (see \cite[Secs 3.1.12, 4.1.2]{cap.slovak.book.vol1}); dimension $n=4$ is an exception, in which the harmonic curvature module has two irreducible components, corresponding to self-dual and anti-self-dual forms.  Any irreducible $P$-submodule factors through $P \rightarrow P / P^+ \cong G_0$, which means that the Weyl curvature descends to a tensor on $M$, a section of $\wedge^2 T^*M \otimes \End(TM)$.  The Weyl submodule of $\wedge^2 (\lieg/\liep)^* \otimes \liep$ projects isomorphically to its image under the quotient
$$ \wedge^2 (\lieg/\liep)^* \otimes \liep \rightarrow  \wedge^2 (\lieg/\liep)^* \otimes \lieg_0 \cong \wedge^2 (\lieg/\liep)^* \otimes \co(1,n-1). $$
and in fact has values in $\wedge^2 (\lieg/\liep)^* \otimes \lieg_0'$, where $\lieg_0' = [\lieg_0,\lieg_0] \cong \so(1,n-1)$---in other words, all values $\kappa(u,v)$ project to dilation-free elements of $\lieg_0 \cong \co(1,n-1)$ (see \cite[VII.3.1]{sharpe} or \cite[Sec 1.6.8]{cap.slovak.book.vol1}).   
We will denote $\bar{\kappa}$ the $P$-equivariant function $\widehat{M} \rightarrow \wedge^2 (\lieg/\liep)^* \otimes \lieg_0'$ resulting from composition of $\kappa$ with the quotient. The value $\bar{\kappa}_{\hat{x}}$ is the Weyl curvature tensor at $x = \pi(\hat{x})$ in the 1-frame determined by $\hat{x}$.


\subsubsection{Weyl curvature tensor and corresponding $G_0$-modules}
\label{sec:weyl_representation}

The Weyl curvature tensor $W$ is a classical object in conformal geometry.  With respect to any metric $g$ in the conformal class $[g]$, it is a component of the Riemannian curvature tensor $R_g$, obtained by subtracting all traces (see, eg, \cite[Sec 1.6.8]{cap.slovak.book.vol1}).  In particular, $W$ shares all symmetries of $R_g$, such as
$$ g(W(X,Y)U,V) = - g(W(X,Y)V,U) \qquad \mbox{and} \qquad g(W(X,Y)U,V) = g(W(U,V)X,Y)$$
for all vector fields $X,Y,U,$ and $V$.
The Weyl tensor is moreover conformally invariant: if $f \in \Conf(M,g)$,
then
$$ f_* \left( W(f^{-1}_* X, f^{-1}_* Y) f^{-1}_* Z  \right) = W(X,Y)Z \qquad \forall X,Y,Z \in \mathcal{X}(M) $$
When $\dim M \geq 4$, vanishing of $W$ on an open subset $U \subseteq M$ is equivalent to conformal flatness of $(U, \left. g\right|_U)$.

The function $\bar{\kappa}$ factors through a $G_0$-equivariant map on the bundle of conformal frames $\widehat{M}/P^+$.  A point of this bundle over $x \in M$ determines in particular an inner product $g_x$ in the conformal class at $x$.  The skew-symmetry with respect to $g_x$ of $W_x(u,v) \ \forall u,v \in T_x M$ corresponds to $\bar{\kappa}_{\hat{x}}(u',v')$ belonging to $\lieg_0' \cong \so(1,n-1)$ for all $u', v' \in \lieg_{-1} \cong \lieg/\liep$.  The classical decomposition of $(R_g)_x$ into $W_x$ and a complement involving the Ricci tensor corresponds to $\bar{\kappa}_{\hat{x}}$ lying in the kernel of the homomorphism (\cite[Def 7.1.21, Exer 7.1.22]{sharpe}).  
\begin{eqnarray*}
\rho_{Ric} & : & \wedge^2 \lieg_{-1}^* \otimes \lieg_0' \rightarrow \lieg_{-1}^* \otimes \lieg_{-1}^* \\
 \rho_{Ric} & : &  \lambda \wedge \mu \otimes U \mapsto  \lambda \circ U \otimes \mu - \mu \circ U \otimes \lambda
\end{eqnarray*}
Elements satisfying the Bianchi identity, as $\bar{\kappa}_{\hat{x}}$ does, have image under $\rho_{Ric}$ in the symmetric submodule of $\lieg_{-1}^* \otimes \lieg_{-1}^*$.

Next, consider the map $r : \wedge^2 \lieg_{-1}^* \otimes \lieg_0' \rightarrow \lieg_0' \otimes \lieg_0'$ which raises an index as follows:
$$ r: \lambda \wedge \mu \otimes U \mapsto (\lambda \otimes \mu^\sharp) \otimes U - (\mu \otimes \lambda^\sharp) \otimes U$$
Here $\mu^\sharp$ denotes the dual in $\lieg_{-1} \cong \BR^{1,n-1}$ with respect to the inner product $\mathbb{I}_{1,n-1}$ (see section \ref{sec:cartan_decomposition}).  The symmetry of pairs in $W_x$ with respect to $g_x$ corresponds to $r \circ \bar{\kappa}_{\hat{x}}$ lying in $\Sym^2(\lieg_0')$.

Denote $G_0' < G_0$ the subgroup, isomorphic to $\mbox{O}(1,n-1)$, of elements with determinant $\pm 1$ with respect to the basis $\{ E_1, \ldots, E_n \}$ of $\lieg_{-1}$.
Note that $r$ is $G_0'$-equivariant, but not $G_0$-equivariant.  Thus $\bar{\kappa}^\sharp = r \circ \bar{\kappa}$ is $G_0' \ltimes P^+$-equivariant, but obeys for $h$ a dilation by $e^s$
\begin{equation}
\label{eqn.dilation.kappabar}
\bar{\kappa}^\sharp_{\hat{x} h^{-1}} = e^{-2s} \bar{\kappa}_{\hat{x}}^\sharp 
\end{equation}
although $h$ naturally acts trivially on $\mathbb{V} = \Sym^2(\lieg_0')$.  In order to make $\bar{\kappa}^\sharp$ a $P$-equivariant map, we extend the representation of $G_0'$  on $\mathbb{V}$ to $G_0$ by assigning to the dilation by $d \neq 0$ of $\lieg_{-1}$ a scalar contraction, multiplication by $d^{-2}$, on $\mathbb{V}$.

Raising an index conjugates $\rho_{Ric}$ (restricted to the submodule satisfying the Bianchi identity) to $\rho_{Ric}^\sharp$, from  $\Sym^2(\lieg_0')$ to the $\mathbb{I}_{1,n-1}$-symmetric endomorphisms of $\lieg_{-1}$:
$$\rho_{Ric}^\sharp  :  U \cdot V \mapsto \frac{1}{2} \left( U \circ V+ V \circ U \right)$$

With respect to the basis $\{ \alpha, \beta \}$ of $\liea^*$, the element of the dual basis corresponding to $\beta$ generates the 1-dimensional subspace $\liea \cap \lieg_0'$.  The $\beta$-weights of the $G_0'$-module $\mathbb{V}$ are $\{ -2, -1, 0, 1, 2\}$.
Denote by $\lieu_+$ the subalgebra parametrized by $U_+$ in the decomposition of $\so(1,n-1)$ in section \ref{sec:cartan_decomposition}, equal to $\lieg_\beta$; by $\lies$ the subalgebra parametrized by $R$ and $b$; and by $\lieu_-$ the subalgebra parameterized by $U_-$, equal to $\lieg_{-\beta}$.  Let $\{ U_i^+ \ : \ i = 1, \ldots, n-2 \}$ be a basis of $\lieu_+$ and similarly for $\lieu_-$; $\{ R_{ij} \ : i < j \}$ a basis of $\liem$; and $B$ the element corresponding to $b=1$.  
In terms of these bases, the positive weight modules in $\mathbb{V}$ are 

\begin{eqnarray*}
\mathbb{V}_{+2}   & =  & \left\{ \sum_{i,j} \alpha_{ij} U_i^+ \cdot U_j^+ \ : \ \alpha_{ij} = \alpha_{ji} \right\} \\
\mathbb{V}_{+1}  & =  & \left\{ \sum_i \beta_i U_i^+ \cdot B + \sum_{i,j,k} \gamma_{ijk} U_i^+ \cdot R_{jk} \ : \ \gamma_{ijk} = - \gamma_{ikj} \right\} 
\end{eqnarray*}
where $i,j,k$ range from $1$ to $n-2$.  

The conditions to be in the kernel of $\rho_{Ric}^\sharp$ for the positive weight modules are
\begin{eqnarray*}
\mathbb{V}_{+2} & : & \sum \alpha_{ii} = 0  \\
\mathbb{V}_{+1} & : & \beta_i = 2 \sum_{j \neq i} \gamma_{jji}, i = 1, \ldots, n-2 
\end{eqnarray*}

When $\dim M = 3$, then $W \equiv 0$ and the Cotton-York tensor is the obstruction to local flatness.  On $\widehat{M}$, it corresponds to the component of $\kappa$ in $\wedge^2 (\lieg / \liep)^* \otimes \liep^+ \cong \wedge^2 (\lieg / \liep)^* \otimes (\lieg / \liep)^*$ (see \cite[Cor 1.6.8]{cap.slovak.book.vol1}).  First we analyze this module in arbitrary dimension and then focus on the case $n=3$.  Again raising an index gives the $G_0$-module $\mathbb{U} = \lieg_0' \otimes \lieg_{-1}^*$.  Assuming $W \equiv 0$, the Cotton tensor corresponds to a $P$-equivariant map $\gamma^\sharp: \widehat{M} \rightarrow \mathbb{U}$, provided we set, for $h$ a dilation by $d \neq 0$, 
$$ \gamma^\sharp_{bh^{-1}} = d^{-3}  \gamma_b^\sharp$$

Let $\xi_i =  \ ^tE_i \mathbb{I}$, for $i=1, \ldots, n$.
The $\beta$-weights on $\lieg_1 \cong \lieg_{-1}^*$ are $\{ -1,0,1 \}$.  We have
\begin{eqnarray*}
\mathbb{U}_{+2}   & =  & \left\{ \sum_{i} \alpha_i U_i^+ \cdot \xi_1 \right\} \\
\mathbb{U}_{+1}  & =  & \left\{ \sum_{i,j} \beta_{ij} U_i^+ \cdot \xi_{j+1} + \sum_{i,j} \gamma_{ij}  R_{ij} \cdot \xi_1 + \delta B \otimes \xi_1 \ : \ \gamma_{ij} = - \gamma_{ji} \right\} 
\end{eqnarray*}
where $i,j$ range from $1$ to $n-2$; if $n=3$, then there are no terms with $R_{ij}$.  

For $n=3$, the Cotton-York tensor has values in the bundle corresponding to the 5-dimensional, irreducible representation of $G_0' \cong O(1,2)$.  The isomorphism
$$t : \so(1,2) \mapsto \lieg_{-1} \qquad  t: U^+_1 \mapsto E_1, \ B \mapsto E_2, \ U^-_1 \mapsto E_3$$
is $\mbox{SO}(1,2)$-equivariant.
With $t$, we convert the Cotton-Yorke tensor to an $\mbox{SO}(1,2)$-equivariant map $\gamma^{\sharp \sim} : \widehat{M} \rightarrow \End(\lieg_{-1})$.  The image is in the $\mathbb{I}$-symmetric, trace-free endomorphisms, which we will denote $\mathbb{U}^\sim$ (see \cite[p 373]{cap.slovak.book.vol1}).  Letting scalar matrices $d I_3$ act on $\mathbb{U}^\sim$ by $d^{-3}$ and $P^+$ act trivially makes $\gamma^{\sharp \sim}$ into a $P$-equivariant map.

The weight spaces in $\mathbb{U}^\sim$ are each one-dimensional, with 
\begin{eqnarray*}
\mathbb{U}_{+2}^\sim   & =  & \BR  E_1 \cdot \xi_1  \\
\mathbb{U}_{+1}^\sim  & =  & \BR  (E_1 \cdot \xi_2 +  E_2 \cdot \xi_1) 
\end{eqnarray*}

We will denote $\mathbb{U}^{\sim+} = \mathbb{U}_{+1}^\sim \oplus \mathbb{U}_{+2}^\sim$ below.

\subsubsection{Determination of isotropic line by positive weight spaces}
\label{sec:weyl_isotropic_lines}

Now let $\mathbb{V}$ be the Weyl curvature module in $\Sym^2(\lieg_0')$, decomposed according to $\beta$-weights, as in section \ref{sec:weyl_representation} above.  Let $\varpi \in \mathbb{V}_{+2} \oplus \mathbb{V}_{+1}$, and write
$$\varpi = \sum_{i,j} \alpha_{ij} U_i^+ \cdot U_j^+ + \sum_i \beta_i U_i^+ \cdot B + \sum_{i,j,k} \gamma_{ijk} U_i^+ \cdot R_{jk} $$
with $\alpha_{ij} = \alpha_{ji}$ and $\gamma_{ijk} = - \gamma_{ikj}$.  We define
\begin{equation}
\label{def:curv_rep_subspaces}
\mathbb{V}^+_{Ric} = (\mathbb{V}_{+2} \oplus \mathbb{V}_{+1}) \cap \ker \rho^\sharp_{Ric} \ \mbox{and} \ \mathbb{V}^+_B = \{ \varpi \in \mathbb{V}^+_{Ric} :  \beta_i = 0 \ \forall \ i = 1, \ldots, n-2\} 
\end{equation}

\begin{proposition}
\label{prop:zariski_closed_subspaces}
The subsets $G_0.\mathbb{V}^+_{Ric}$ and $G_0.\mathbb{V}^+_B$ are Zariski closed in $\mathbb{V}$ for $n \geq 4$; when $n=3$, then $G_0.\mathbb{U}^{\sim+}$ and $G_0.\mathbb{U}_{+2}^\sim$ are Zariski closed in $\mathbb{U}^\sim$. 
\end{proposition}

\begin{proof}
  Assume $n \geq 4$.
Let $N = \dim \mathbb{V}$, $N_{Ric} = \dim \mathbb{V}^+_{Ric}$ and $N_B = \dim \mathbb{V}^+_B$.  We first show that $G_0.[\mathbb{V}^+_{Ric}] \subseteq \Gr(N_{Ric},N)$ and $G_0.[\mathbb{V}^+_B] \subseteq \Gr(N_B,N)$ are Zariski closed.

On the Grassmannian varieties, the action of $G_0$ factors through $G_0 \rightarrow \PO(1,n-1)$; we will denote this simple group by $J$ for the remainder of the proof.  Let $Q < J$ be the stabilizer of the line $\BR E_1$ (modulo $\{ \pm \Id \}$).  It is the maximal parabolic subgroup, determined by $\beta$, which has been defined as a restricted root of $\lieg = \so(2,n)$, but is in fact also the unique nontrivial restricted root of $\mathfrak{j} \cong \so(1,n-1)$.  
Note that 
$$ \mathbb{V}^+ = \mathbb{V}_{+1} \oplus \mathbb{V}_{+2}  = \lieu_+ \odot \lies \oplus \Sym^2 \lieu_+$$
(with notation as in section \ref{sec:weyl_representation})
and that $\lieu_+ = \lieg_\beta = \mathfrak{j}_\beta$, while $\lies = \BR B \oplus \liem$ is the 0-weight space in $\so(1,n-1)$ with respect to $\beta$. 
The Lie algebra of $\lieq$ is precisely $ \lies \ltimes \lieu_+$.
Now $Q$ leaves $\mathbb{V}^+$ invariant.  As $\rho^\sharp_{Ric}$ is $J$-equivariant, it follows that $Q$ leaves $\mathbb{V}^+_{Ric} \subset \mathbb{V}^+$ invariant.  Moreover, $Q$ leaves 
$$\mathbb{V}^+_B = \lieu_+ \odot \liem \oplus \Sym^2 \lieu_+$$ 
invariant, because $Q.\lieu_+ = \lieu_+$ and $Q.\liem = \liem + \lieu_+$.

Now, the stabilizers of $[\mathbb{V}_{Ric}^+]$ and $[\mathbb{V}_B^+]$ are algebraic, proper subgroups of $J$.  But a proper, algebraic subgroup containing the maximal parabolic $Q$ necessarily equals $Q$ (see \cite[pp 411--415]{knapp.lie.groups}; note that $Q$ is the maximal closed subgroup of $J$ with Lie algebra $\lieq$, because the normalizer of $Q^0$ stabilizes $\BR E_1$).  We conclude that $\Stab [\mathbb{V}_{Ric}^+] = \Stab [\mathbb{V}_B^+] = Q$.  Then the orbit maps $J \rightarrow \Gr(N_{Ric},N), \Gr(N_B,N)$ factor through embeddings of the closed projective variety $J/Q$ into the respective Grassmannian varieties.  

\begin{lemma}
\label{lem:complex_stabilizers}
The stabilizers in $J_\BC \cong PO(n,\BC)$ of $[\mathbb{V}^+_{Ric}] \in \Gr_\BC(N_{Ric},N)$ and $[\mathbb{V}^+_B] \in \Gr_\BC(N_B,N)$  equal the maximal parabolic subgroup determined by the root $\beta$, which is $Q_\BC$.  
\end{lemma}

\begin{proof}
The representation of $J_\BC$ on $\mathbb{V}_\BC$ is isomorphic to $\Sym^2 \mathfrak{j}_\BC \cong \Sym^2 \so(n,\BC)$.  Now $(\lieu_+)_\BC$ is the $+1$-weight space for the complex root $\beta$ (strictly containing the root space for $\beta$ unless $n=3$), and $\lies_\BC$ is the $0$-weight space with respect to $\beta$.  Now $Q_\BC$ acts by automorphisms of $\lieq_\BC = \lies_\BC \ltimes (\lieu_+)_\BC$, so it preserves 
$\mathbb{V}^+_\BC $.  The linear map $(\rho^\sharp_{Ric})_\BC$ is evidently $J_\BC$-equivariant, so $Q_\BC$ preserves $(\mathbb{V}^+_{Ric})_\BC$.  The subalgebra $\liem_\BC \ltimes (\lieu_+)_\BC$ is in fact an ideal in $\lieq_\BC$, corresponding to the $(\liem + \lien)$ portion of the Langlands decomposition; therefore, it is normalized by $Q_\BC$.  It follows that $Q_\BC$ preserves $(\mathbb{V}^+_B)_\BC$.  As $Q_\BC$ is the stabilizer in $J_\BC$ of $\BC E_1$, it follows as in the real case that the full stabilizers of these subspaces are both equal to $Q_\BC$.
\end{proof}

The lemma above implies that the orbits $J_\BC.[\mathbb{V}^+_{Ric}]$ and $J_\BC.[\mathbb{V}^+_B]$ are the embedded images of closed projective varieties, isomorphic to $J_\BC/Q_\BC$.   Now the properness of closed projective varieties \cite[Thm I.5.3]{shafarevich} implies that both images in the Grassmannian variety are Zariski closed. 

In general $J$ need not act transitively on the real points of the homogeneous variety $J_\BC/Q_\BC$; we note that, in our case, the real points comprise the quadric in ${\bf RP}^{n-1}$, diffeomorphic to $S^{n-2}$, given in homogeneous coordinates by $2x_1x_n + x_2^2 + \cdots + x_{n-1}^2 = 0$, which is precisely $J/Q$.  Thus the real points of the orbits $J_\BC.[\mathbb{V}^+_{Ric}]$ and $J_\BC.[\mathbb{V}^+_B]$ are precisely the orbits $J.[\mathbb{V}^+_{Ric}]$ and $J.[\mathbb{V}^+_B]$, which are therefore also Zariski closed.

Next, we restrict the canonical bundles of the Grassmannian varieties to these closed orbits to obtain Zariski closed subbundles in $\Gr \times \mathbb{V}$.   The subsets $G_0.\mathbb{V}^+_{Ric}$ and $G_0.\mathbb{V}^+_B$ are the projections of these Zariski closed subbundles to $\mathbb{V}$.  These projections are Zariski closed by properness of Grassmannian varieties.

Via the $\SO(1,2)$-equivariant isomorphisms $\BR^{1,2*} \cong \BR^{1,2} \cong \so(1,2)$, the representation $\mathbb{U}^\sim$ is also isomorphic to $\Sym^2 \so(1,2)$.  The subspaces corresponding to $\mathbb{U}^{\sim+}$ and $\mathbb{U}_{+2}^\sim$ have analogous parametrizations to $\mathbb{V}^+_{Ric}$ and $\mathbb{V}_B^+$, respectively, in terms of the basis $\{ U^+_1, B \}$ of the Lie algebra of the maximal parabolic $\lieq < \so(1,2)$.  Then the appropriate analogue of lemma \ref{lem:complex_stabilizers} and the properness argument also hold in this case.
 \end{proof}

\begin{proposition}
\label{prop:line_determination}
For $\varpi \in \mathbb{V}^+_{Ric}$ (or $\nu \in \mathbb{U}^{\sim+}$ if $n=3$), the evaluation $\varpi(u \cdot v) = 0$ for all $u,v \in \ell = \BR E_1$ ($\nu(u) = 0$ for all $u \in \BR E_1$).  An isotropic line with this property is unique, unless $\varpi \in \mathbb{V}^+_B$ and $n \geq 4$ (or $\nu=0$ if $n=3$).  In this case, $\varpi \lrcorner u = 0$ for all $u \in \ell = \BR E_1$, and $\ell$ is the unique isotropic line with this property unless $\varpi = 0$.   
\end{proposition}

\begin{proof}
It is straightforward to verify that $\varpi(E_1 \cdot E_1) = 0$.  Let $v$ be a different nonzero vector in the null cone $\mathcal{N}$; up to rescaling, it can be written 
$$ v = E_n + \sum_i a_i E_i - \left( \frac{1 }{2} \sum_i a^2_i \right) E_1 \qquad a \in \BR^{n-2}$$
The coefficient of $E_i \cdot E_n$ in $\varpi(v \cdot v)$ is $\beta_i$.  Thus $\varpi(v \cdot v) = 0$ implies $\beta_i = 0$ for all $i = 1, \ldots, n-2$.  

Assuming $\beta_i = 0$ for all $i = 1, \ldots, n-2$, it is clear that $\varpi \lrcorner E_1 = 0$.  The coefficient of $E_i \cdot R_{jk}$ in $\varpi \lrcorner v$ is $\gamma_{ikj}$.  Thus $\varpi \lrcorner v = 0$ implies $\gamma_{ijk} = 0$ for all $i,j,k$.  In this case, the coefficient of $E_i \cdot U_j^+$ in $\varpi \lrcorner v$ is $- \alpha_{ij}$.  We conclude that $\varpi \lrcorner v$ can only be 0 if $\varpi = 0$.
\end{proof}

\begin{corollary}
\label{cor:weyl_determines_lines}
There are $G_0$-equivariant, algebraic maps from $G_0. \mathbb{V}^+_{Ric} \backslash G_0.\mathbb{V}^+_B$ and from $G_0.\mathbb{V}^+_B \backslash \{ 0 \}$ to ${\bf P} (\mathcal{N} \backslash \{ 0 \} )$.  When $n=3$, there is a $G_0$-equivariant, algebraic map from $G_0.\mathbb{U}^{\sim+} \backslash \{ 0 \}$ to ${\bf P} (\mathcal{N} \backslash \{ 0 \} )$.
\end{corollary}

\subsection{Holonomy sequences}

Let $\{f_k\}$ be a sequence of conformal transformations of $M$, and suppose $f_k .x_k \rightarrow y$ for a convergent sequence $x_k \rightarrow x$.  
The following definition captures the divergence of $f_k$ in the fiber direction along $\{ x_k \}$.

\begin{definition}
A \emph{holonomy sequence} for $\{ f_k \}$ at $x$ is $p_k \in P$ for which there exist a convergent sequence $\hat{x}_k \rightarrow \hat{x} \in \pi^{-1}(x)$ and $\hat{y} \in \widehat{M}$ such that
$$ f_k. \hat{x}_k .p_k^{-1} \rightarrow \hat{y}$$
A \emph{pointwise holonomy sequence} is $\{ p_k \}$ as above for which also $\{ \hat{x}_k \} \subset \pi^{-1}(x)$.
\end{definition}

In the situation $f_k .x_k \rightarrow y$ with $x_k \rightarrow x$ given above, any choice of convergent sequence $\hat{x}_k \in \pi^{-1}(x_k)$ and $\hat{y} \in \pi^{-1}(y)$ gives a holonomy sequence.  When $x_k = x$, we can obtain a pointwise holonomy sequence.

Any $X \in \lieg$ defines a vector field $\hat{X}$ on $\widehat{M}$ by $\omega(\hat{X}) \equiv X$.  The exponential map of a Cartan geometry is given 
by the time-one flow along these $\omega$-constant vector fields.
\begin{definition}
\label{def.exp.map}
The \emph{exponential map} at $\hat{x} \in \widehat{M}$ is 
$$ \exp_{\hat{x}}(X) = \varphi^1_{\hat{X}}(\hat{x}) \in \widehat{M}$$
for $X$ in a sufficiently small neighborhood of $0$ in $\lieg$.
\end{definition}
The restriction of $\exp_{\hat{x}}$ to a sufficiently small neighborhood of $0$ in $\lieg$ is a 
diffeomorphism onto a neighborhood of $\hat{x}$ in $\widehat{M}$. The map $\pi \circ \exp_{\hat{x}}$ induces a diffeomorphism from a neighborhood of $0$ in $\lieg_{-1}$ to a neighborhood of $x=\pi(\hat{x})$ in $M$. Projections to $M$ of exponential curves 
$s\mapsto\exp(\hat{x},sX)$ for $X\in\lieg_{-1}$ are \emph{conformal geodesics}. 

Suppose that for $h \in \Conf(M,g)$, there is $\hat{x} \in \widehat{M}$ with $h.\hat{x} = \hat{x}.g$ for some $g \in P$ (in particular, $h$ fixes $x = \pi(\hat{x}))$.  Suppose that for $X \in \lieg$, the following equation holds in $G$:
\begin{equation}
ge^{sX}=e^{c(s)X}p(s) \qquad \forall s \in I
\end{equation} 
Here $I$ is an interval containing $0$; $c: I {\rightarrow} I'$ is a diffeomorphism fixing $0$; and $p: I\rightarrow P$ is a smooth path with $p(0)=g$.  Hence,  $g$ acts on the curve $[e^{sX}]$ in $\Ein^{1,n-1} \cong G/P$ by a reparametrization. 
Then it follows from \cite[ Prop 4.3]{fm.nilpconf} or \cite[Prop 2.1]{cap.me.parabolictrans} that the analogous equation holds in $\widehat{M}$:
\begin{equation}
\label{eqn:isotropy_exp_curves}
h.\exp(\hat{x}, sX)=\exp(\hat{x}, c(s)X).p(s)\quad \forall s\in I.
\end{equation}

A holonomy sequence for $\{ f_k \}$ at $x$ is also valid along certain exponential curves from $x$:

\begin{proposition}(Propagation of holonomy)
\label{prop:propagation_holonomy}
 Let $\{p_k \}$ be a holonomy sequence for $f_k$ at $x$, with respect to $\hat{x}_k \in \pi^{-1}(x_k)$.  Suppose given $Y_k  \rightarrow Y \in \lieg \backslash \liep$ for which $\Ad p_k(Y_k)$ converges.  Then, provided $Y$ is in the domain of $\exp_{\hat{x}}$, $\{ p_k \}$ is also a holonomy sequence for $\{ f_k \}$ at $x' = \pi \circ \exp(\hat{x},Y)$ with respect to $\hat{x}_k' = \exp(\hat{x}_k,Y_k)$.  
\end{proposition}

This is a variation on several existing propositions, such as \cite[Prop 6.3]{frances.locdyn} or \cite[Prop 2.8]{cap.me.parabolictrans}.  It is quickly deduced from equivariance properties of the exponential map, as follows:

\begin{proof}
Let $\hat{y} = \lim f_k.\hat{x}_k .p_k^{-1}$ and $Y_* = \lim \Ad p_k(Y_k)$.  
\begin{eqnarray*}
f_k.\hat{x}_k'.p_k^{-1} & = & f_k.\exp(\hat{x}_k,Y_k).p_k^{-1} \\
& = & \exp(f_k.\hat{x}_k.p_k^{-1}, \Ad p_k (Y_k)) \\
& \rightarrow & \exp(\hat{y},Y_*)
\end{eqnarray*}
\end{proof}

On the other hand, at a given $x$ with a holonomy sequence $\{ p_k \}$, there are many other valid holonomy sequences.  Choosing $\hat{x}_k' = \hat{x}_k l_k'$ and $\hat{y}' = \hat{y} l^{-1}$, where $l_k \rightarrow l$ and $l_k' \rightarrow l'$ in $P$, gives a new holonomy sequence $p_k' = l_k p_k l_k'$, which we will call \emph{vertically equivalent} to $\{ p_k \}$.

Recall that $P$ is isomorphic to the semidirect product $G_0 \ltimes P^+$, where $P^+ \cong \BR^{1,n-1*}$.  As $G_0$ is reductive, it has a KAK decompositon with $A = \exp(\liea)$ (see (\ref{eqn:cartan_subalgebra}) in section \ref{sec:cartan_decomposition}).
Any holonomy sequence $\{ p_k \}$ is vertically equivalent to one of the form $\{ d_k \tau_k \}$ with $d_k \in A$ and $\tau_k \in P^+$.  In view of the $\Ad K$-action on $A$, such a sequence is in fact equivalent to one with $d_k \in A'$, the semigroup of $A$ comprising elements $d$ satisfying $\beta(\ln d) \leq 0$.  Then we will say the sequence is in \emph{$A'P^+$-form}.

The following terminology for sequences in $A'P^+$ is based on \cite[Sec 3.2]{frances.lorentz.klein}.

\begin{definition}
\label{def:holonomy_trichotomy}
  Let $\{ d_k \tau_k \}$ be a sequence in $A'P^+$-form, and let $D_k = \ln d_k$.  The sequence is said to be
\begin{itemize}
\item of \emph{bounded distortion} if $\alpha(D_k)$ is bounded.
\item \emph{balanced} if $\alpha(D_k) + \beta(D_k)$ is bounded, but each term is unbounded.
\item \emph{mixed} if $\alpha(D_k), \beta(D_k)$, and $(\alpha + \beta)(D_k)$ are unbounded, with $\alpha(D_k) \rightarrow \infty$.
\end{itemize}
It is called \emph{linear} if $\tau_k \equiv 1$.
\end{definition}

\subsection{Approximately stable spaces}
\label{sec:approximately_stable}

The following definition is inspired by \cite{zeghib.tgl1}, \cite[Sec 7.4]{dag.rgs} (see also \cite[Def 2.10]{cap.me.parabolictrans} for a non-approximate version, and \cite[Sec 4.4]{frances.degenerescence} for a related notion of stability and stable foliations).

\begin{definition}
Let $\mathbb{V}$ be a $P$-module, and let $\{ p_k \}$ be a sequence in $P$.  The \emph{approximately stable set} for $\{ p_k \}$ in $\mathbb{V}$ is
$$ \mathbb{V}^{AS}(p_k) = \{ v = \lim v_k \in \mathbb{V} \ : \ p_k.v_k \ \mbox{is bounded} \} $$
\end{definition}

Let $\mathbb{V} = \lieg / \liep$, so $TM \cong \widehat{M} \times_P \mathbb{V}$ (see, eg,  \cite[Thm 5.3.15]{sharpe}).  Denote $q$ the quotient $\widehat{M} \times \mathbb{V} \rightarrow TM$.  If $\{ p_k \}$ is a holonomy sequence for $\{ f_k \}$ at $x$ with respect to $\hat{x}_k$, and $Y \in \mathbb{V}^{AS}(p_k)$, then $q(\hat{x},Y) \in T_xM$ is approximately stable for $\{ f_k \}$ in the sense of Zeghib.  For this case, we will call $q(\hat{x},\mathbb{V}^{AS}(p_k))$ the \emph{approximately stable subset of $T_xM$} for $\{ p_k \}$.

Approximate stability relates to propagation of holonomy:

\begin{remark}
Let $\{p_k \}$ be a holonomy sequence with respect to $\hat{x}_k \rightarrow \hat{x}$.
Let $\mathbb{V} = \lieg$ and let $Y \in \mathbb{V}^{AS}(p_k) \backslash (\mathbb{V}^{AS}(p_k) \cap \liep)$.  
 Then by proposition \ref{prop:propagation_holonomy}, $\{p_k \}$ is also a holonomy sequence at $\exp(\hat{x},Y)$, provided $Y$ is in the domain of $\exp_{\hat{x}}$. 
\end{remark}

The above remark will be applied below to holonomy sequences in $A'P^+$-form and $Y_k \rightarrow Y$ in $\lieg\backslash \liep$.

The following proposition is a version for sequences of \cite[Prop 2.9]{cap.me.parabolictrans}:

\begin{proposition}(Approximate stability for invariant sections)
\label{prop:as_for_inv_sections}
Given a $P$-module $\mathbb{V}$, represent a continuous, $\{ f_k \}$-invariant section of the associated bundle $\widehat{M} \times_P \mathbb{V}$ by a continuous, $P$-equivariant, $\{ f_k \}$-invariant map $\sigma : \widehat{M} \rightarrow \mathbb{V}$.  Given any holonomy sequence $\{ p_k \}$ for $\{ f_k \}$ with respect to $\hat{x}_k \rightarrow \hat{x}$, the value $\sigma(\hat{x}) \in \mathbb{V}^{AS}(p_k)$.
\end{proposition}

\begin{proof}
$$ p_k . \sigma(\hat{x}_k) = \sigma(f_k. \hat{x}_k .p_k^{-1}) \rightarrow \sigma(\hat{y})$$
for a certain $\hat{y} \in \widehat{M}$; moreover, $\sigma(\hat{x}_k) \rightarrow \sigma(\hat{x})$.
\end{proof}


\subsection{Proof of main theorem given a contracting holonomy sequence}
\label{sec:contracting_proof}

The following proposition illustrates the application of stability and propagation of holonomy.  Conformal flatness in the presence of predominantly contracting dynamics is relatively well known --- see, for example \cite[Props 4,5]{frances.ccvf}, \cite[Prop 3.6]{fm.champsconfs}.

\begin{proposition}
\label{prop:contracting_proof}
Let $(M^n,g)$ be a Lorentzian manifold with $n \geq 3$.  Let $\{h_k \}$ be a sequence in $\Conf(M,[g])$ with
holonomy sequence  $\{ p_k = d_k \tau_k\}$ in $A'P^+$-form at $x \in M$.
Let $D_k = \ln d_k$ and suppose that $\alpha(D_k) + \beta(D_k) \rightarrow  \infty$ and $\{ \tau_k \}$ is bounded.  Then an open neighborhood  $U \subset M$ of $x$ is conformally flat.
\end{proposition}

\begin{proof}
On $\lieg/\liep$, the action of $p_k$ factors through $d_k$, which acts by scalars $e^{- \alpha(D_k) + w \beta(D_k)}$ for $w=-1,0,1$.  As $\beta(D_k) \leq 0$, the exponents $- \alpha(D_k) + w \beta(D_k) \rightarrow - \infty$ for $w=-1,0$ or $1$.  By proposition \ref{prop:propagation_holonomy}, $\{ p_k \}$ is a holonomy sequence at all points of $U=\pi \circ \exp_{\hat{x}}(\mathcal{U})$, for $\mathcal{U}$ a neighborhood of $0$ in $\lieg_{-1}$ intersect the domain of $\exp_{\hat{x}}$.

The action of $p_k$ on the Weyl curvature module $\mathbb{V}_{Ric}$ also factors through $d_k$, which acts by the scalars $e^{2\alpha(D_k) + w \beta(D_k)}$ for $w = -2, -1, 0, 1,2$.  On the Cotton module $\mathbb{U}$, it acts by the scalars $e^{3\alpha(D_k) + w \beta(D_k)}$ for $w = -2,-1, 0, 1,2$ (see section \ref{sec:weyl_representation}).  These sequences of exponents diverge to $\infty$ for all values of $w$, and the approximately stable set in either module is $\{ 0 \}$.  Then by proposition \ref{prop:as_for_inv_sections}, the Weyl curvature vanishes on $U$, and the Cotton tensor vanishes on $U$ if $n=3$.  
\end{proof}

If we also assume that $(M,g)$ is real-analytic, as in theorem \ref{thm:main}, then we can further conclude that $(M,g)$ is everywhere conformally flat.
By proposition B, we may henceforth assume there is no contracting holonomy sequence, as in the hypotheses of proposition \ref{prop:contracting_proof} above, associated to any sequence $\{h_k\} \subset H$.

\section{Proof given a linear, balanced holonomy sequence}
\label{sec:balanced_proof}

We now prove a global result which relates to Alekseevky's examples \ref{ex:alekseevsky}, in which an essential conformal flow
fixes a null geodesic pointwise and has linear, balanced holonomy at these fixed points (see definition \ref{def:holonomy_trichotomy}).  The proposition below implies, in combination with Proposition B and Haefliger's Theorem, that such an essential flow is not possible on a closed, simply connected, real-analytic Lorentzian manifold.  

\begin{proposition}
\label{prop:balanced_proof}
Let $(M,g)$ be a compact, real-analytic Lorentzian manifold of dimension $n \geq 3$.  Let $\{ h_k \} \subset \Conf(M,g)$, and 
let $\{ p_k  \}$ be a holonomy sequence for $\{ h_k \}$ at $x$ that is balanced and linear.  Then $M$ is conformally flat or the universal cover of $M$ admits a codimension-one analytic foliation; if $n=3$, then $M$ is conformally flat.
\end{proposition}

The construction of the foliation, assuming $M$ is not conformally flat, is a miniature of the construction of a codimension-one, analytic foliation on an open dense set in section \ref{sec:foliation_construction}, where we are in the more difficult situation of mixed and bounded holonomy sequences.  We will freely use elementary properties of the grading of $\lieg = \so(2,n)$ and and of its root space decomposition, as
given in sections
\ref{sec:equivalence_principle} and \ref{sec:cartan_decomposition}.

\begin{proof}
Let $\{ p_k \}$ be a holonomy sequence as in the proposition with respect to $\hat{x}_k \rightarrow \hat{x}$.  
  We can assume $p_k = d_k \in A'$ for all $k$.  Write $D_k = \ln d_k$ as usual.

\emph{Step 1: curvature values.}
The action of $p_k$ on the Weyl curvature module $\mathbb{V}$ is by scalars $e^{2 \alpha(D_k) + w \beta(D_k)}$ for $w=-2, -1,0,1,2$.  Our assumptions imply these sequences diverge unless $w=2$, which corresponds to values of $\bar{\kappa}^\sharp(\hat{x})$ in $\mathbb{V}_{+2}$.  On the Cotton module $\mathbb{U}$, the scalars are $e^{3 \alpha(D_k) + w \beta(D_k)}$, which diverge for every $w$ when the holonomy sequence is balanced and linear.  The full Cartan curvature $\kappa(\hat{x})$ is also approximately stable for $\{ d_k \}$, so it lies in $\wedge^2 (\lieg/\liep)^* \otimes \lieg_0$, and can be identified with the Weyl curvature;   
moreover, if $n = 3$, then the Cartan curvature vanishes at $x$.

  \emph{Step 2: propagation of holonomy}.  
  The adjoint action of $p_k$ on $\lieg_{-1} \cong \BR^{1,n-1}$ is by scalars $e^{- \alpha(D_k) + w \beta(D_k)}$ for $w=-1,0,1$.  Our assumptions imply these scalar sequences are all bounded.
By proposition \ref{prop:propagation_holonomy}, there is a neighborhood $\mathcal{U}$ of $0$ in $\lieg_{-1}$ such that $\{ p_k \}$ is a holonomy sequence with respect to all $\hat{y} \in \widehat{U} = \exp_{\hat{x}}(\mathcal{U})$.  Denote $U = \pi(\widehat{U})$, a neighborhood of $x$.  By Step 1, all values of $\bar{\kappa}^\sharp$ on $\pi^{-1}(U)$ belong to $G_0.\mathbb{V}_{+2}$; moreover, for the Cartan curvature $\kappa$, raising an index gives $\kappa^\sharp(\widehat{U}) \subset \mathbb{V}_{+2}$.   Also from Step 1, if $n=3$, then $U$ is conformally flat; by analyticity, so is $M$.  Henceforth we assume $n \geq 4$.

\smallskip

\emph{Step 3: reduction of first-order frame bundle}.
By proposition \ref{prop:line_determination}, there is an algebraic map from $G_0.\mathbb{V}_B^+ \backslash \{ 0 \} \rightarrow {\bf P}(\mathcal{N} \backslash \{ 0 \} )$.  Choosing a different base point $x$ in $U$ or shrinking $U$ if necessary, we may assume (unless $M$ is conformally flat), that $\bar{\kappa}^\sharp$ does not vanish on $\pi^{-1}(U)$.   As $\mathbb{V}_{+2} \subset \mathbb{V}_B^+$, we obtain via $\bar{\kappa}^\sharp$ a $P$-equivariant real-analytic map $\hat{\eta}: \pi^{-1}(U) \rightarrow {\bf P}(\mathcal{N} \backslash \{ 0 \} )$.
Recall the basis $\{ E_1, \ldots, E_n \}$ of $\lieg_{-1}$ introduced in section \ref{sec:cartan_decomposition}.
Let $Q_0 < G_0$ be the stabilizer of $\BR E_1$, with Lie algebra $\lieq_0$.  Denote $\hat{\pi}$ the projection $\widehat{M} \rightarrow \widehat{M}/P^+$.  We define an analytic $Q_0$-reduction
$$ \mathcal{R}' = \hat{\pi}(\hat{\eta}^{-1}( \BR E_1 ))$$
The interpretation of $\mathcal{R}'$ is as the conformal normalized 1-frames at points $x \in U$ in which the strongly stable set of $\{ h_k \}$---the limits in $T_xM$ of sequences of tangent vectors tending to $0$ under $h_k$---is identified with $E_1^\perp$.  Thus $\mathcal{R}'$ determines a distribution of degenerate hyperplanes on $U$, as well as an isotropic line field, which we will denote $\mathcal{L}$.  Observe that, as $\{ p_k \}$ is a holonomy sequence with respect to all $\hat{y}$ in $\widehat{U}$, the projection $U' = \hat{\pi}(\widehat{U})$ is contained in $\mathcal{R}'$.  Thus $\hat{\pi}_* \omega^{-1}_{\hat{x}} (\lieg_{-1})$ is tangent to $\mathcal{R}'_{\hat{\pi}(\hat{x})}$ whenever $\{ p_k \}$ is a holonomy sequence at $\hat{x}$.  

\smallskip

\emph{Step 4: reduction of $\left. \widehat{M} \right|_U$}.  
Define 
$$ \mathcal{R} = \{ \hat{y}  \ : \ \hat{\pi}(\hat{y}) = y' \in \mathcal{R}' \ \mbox{and} \ \hat{\pi}_*(\omega^{-1}_{\hat{y}}(\lieg_{-1})) \subset T_{y'} \mathcal{R}' \}$$ 
Let $S_1 < P^+$ be the subgroup with Lie algebra $\lies_1 \subset \liep^+ \cong \BR^{1,n-1*}$ equal the annihilator of $E_1^\perp$.  The Lie algebra $\lies_1$ comprises all elements $\xi \in \liep^+$ with $[\xi,\lieg_{-1}] \subset \lieq_0$, as can be seen via the root space decomposition of section \ref{sec:cartan_decomposition}, in which $\lies_1$ equals the one-dimensional space $\lieg_{\alpha + \beta}$.  The subgroup $S_1$ equals the normalizer in $P^+$ of 
$\lieg_{-1} + \lieq_0 \subset \lieg/\liep^+$.   It follows that $\mathcal{R}$ is an analytic reduction of $\left. \widehat{M} \right|_U$ to $Q_0 \ltimes S_1$; it equals $\widehat{U}.(Q_0 \ltimes S_1)$.  We can interpret $\mathcal{R}$ as the conformal normalized 2-frames at points $x \in U$ in which $\mathcal{L} = \BR E_1$ and $\mathcal{L}$ is parallel (infinitesimally at $x$).

\emph{Step 5: unipotent reduction}.
Note that $Q_0 < G_0$ preserves $\mathbb{V}_{+2} \subset \mathbb{V}$.  If the Cartan curvature $\kappa(\hat{x})$ is identified with the Weyl curvature $\bar{\kappa}(\hat{x})$ in the sense of step 1 for $\hat{x} \in \mathcal{R}$, and $\bar{\kappa}^\sharp(\hat{x}) \in \mathbb{V}_{+2}$, then $S_1$ fixes $\kappa^\sharp(\hat{x})$ also in the $P$-representation corresponding to the full Cartan curvature.
It follows that $\kappa^\sharp(\mathcal{R}) \subseteq \mathbb{V}_{+2}$.

Recall that $\omega_{\hat{y}}^{-1}(\lieg_{-1}) \subset T_{\hat{y}} \widehat{U}$ for all $\hat{y} \in \widehat{U}$.  It follows that $\omega_{\hat{y}}^{-1}(\lieg_{-1}) \subset  T_{\hat{y}} \mathcal{R}$ for all $\hat{y} \in \mathcal{R}$ because $\Ad(Q_0 \ltimes S_1).\lieg_{-1} \subset \lieg_{-1} + \lieq_0$.  Recall that $\lieu_+$ is the unipotent subalgebra of $\lieg_0$ parametrized by $U_+$ in section \ref{sec:cartan_decomposition}.  As $\lieu_+ \subset \lieq_0$, the subalgebra $\lieg_{-1} + \lieu_+ \subset \omega_{\hat{y}}(T \mathcal{R})$ for all $\hat{y} \in \mathcal{R}$.  Denote the corresponding analytic distribution on $\mathcal{R}$ by $\widehat{\mathcal{D}}$.  

Because $\kappa^\sharp(\mathcal{R}) \in \mathbb{V}_{+2}$, the values 
$$ \kappa_{\hat{x}}(u,v) \in \lieu_+ \qquad \forall \ u,v \in \lieg_{-1}, \hat{x} \in \mathcal{R}$$
Recall the Cartan curvature 2-form from section \ref{sec:cartan_curvature}.  Now take $X = \omega^{-1}(u)$ and $Y = \omega^{-1}(v)$ and apply the usual formula for $d \omega$ to obtain 
\begin{equation}
\label{eqn:curvature_formula}
\kappa_{\hat{x}}(u,v) = \Omega_{\hat{x}}(X,Y) = X.\omega(Y) - Y. \omega(X) - \omega([ X,Y]) + [\omega(X), \omega(Y)] = - \omega([X,Y])
\end{equation}
 and to conclude that $\widehat{\mathcal{D}}$ is integrable in $\mathcal{R}$. 

The result is an analytic reduction $\mathcal{S}$ to the unipotent subgroup $e^{\lieu_+} < G_0$, on which $\omega$ has values in $\lieg_{-1} + \lieu_+$.  This reduction determines a metric in the conformal class on $U$.  The restriction of $\omega$ to $T \mathcal{S}$ is a principal connection with holonomy in $e^{\lieu_+}$.  It is the Levi-Civita connection of the metric---note that $\omega$ is already torsion-free by the regularity conditions in the solution of the equivalence problem (see \cite{sharpe} Ch V).  
We note that the same reduction and metric were previously obtained in \cite[Lemme 6.5]{frances.degenerescence}.
Frances' techniques are similar to ours, though the context of his paper is different and he arrives at this reduction by a different path.

\emph{Step 6: conclusion}.
Because $e^{\lieu_+}$ fixes $E_1$, the reduction $\mathcal{S}$ 
determines
an analytic, isotropic Killing field in $\mathcal{L}$ by $Z(y) = \pi_* \omega_{\hat{y}}^{-1}(E_1)$, for any $\hat{y} \in \mathcal{S}_{y}$.
The vector field $Z$ is parallel, which implies it is Killing.  It is thus an analytic, isotropic local conformal vector field on $U$.  By \cite{amores.killing}, Z extends to a conformal vector field after lifting to the universal cover $\widetilde{M}$ of $M$, which is everywhere isotropic by analyticity; we will denote this vector field also by $Z$.  As $Z$ is \emph{causal}, if it vanishes at some point in $\widetilde{M}$, then $\widetilde{M}$, and hence $M$, is conformally flat by \cite{frances.ccvf}.  

Now suppose $Z$ is nonvanishing.  It determines an analytic hyperplane distribution by $\mathcal{D} = Z^\perp$.  We will show that $Z^\perp$ is an integrable distribution on $U$, which suffices by analyticity to show integrability of $\mathcal{D}$ on $\widetilde{M}$.  For any $\hat{y} \in \mathcal{R}_y$, the subspace $Z^\perp(y) = \pi_* \omega^{-1}_{\hat{y}} (E_1^\perp)$.  As $\omega^{-1}(E_1^\perp) \subset \omega^{-1}(\lieg_{-1})$ is tangent to $\mathcal{R}$, this distribution in fact lifts to a distribution $\widetilde{\mathcal{D}}$ on $\mathcal{R}$.    

Now we again use that $\kappa^\sharp(\mathcal{R}) \subset \mathbb{V}_{+2}$.  First,
$$ \kappa_{\hat{y}}(u,v) = 0 \qquad \forall \ u,v \in E_1^\perp, \hat{y} \in \mathcal{R} $$
Now, again setting $X = \omega^{-1}(u)$ and $Y = \omega^{-1}(v)$ and applying the formula (\ref{eqn:curvature_formula}) gives
\begin{equation*}
\kappa_{\hat{y}}(u,v) = \Omega_{\hat{y}}(X,Y) = X.\omega(Y) - Y. \omega(X) - \omega([ X,Y]) + [\omega(X), \omega(Y)] = - \omega_{\hat{y}}([X,Y]).
\end{equation*}
It follows that $\widetilde{\mathcal{D}}$ is integrable in $\mathcal{R}$.  The leaves project to integral leaves of $Z^\perp$ in $U$.  Thus $\mathcal{D}$ is integrable, yielding a codimension-one, analytic foliation on $\widetilde{M}$.  
  \end{proof}

If $M$ as in the proposition above is also simply connected, then proposition B and Haefliger's Theorem C apply, yielding a contradiction to either of the two conclusions.  We may henceforth assume that there are no balanced, linear holonomy sequences associated to any sequences of conformal transformations of $M$.

\section{Proof given a fixed point}
\label{sec:fixed_points}

Throughout this section, $H$ will be, as in section \ref{sec:stratification_abelian}, a maximal connected, abelian subgroup of $\Conf(M,g)$, obtained
from a maximal commuting collection of conformal vector fields. The maximal compact subgroup of $H$ is denoted $L$, so $H \cong \BR^k \times L$ for some $k \in \BN$.  In this section, $(M,g)$ will be a closed, simply connected, analytic Lorentzian manifold of dimension at least 3.

Here we prove that $H$ must be compact, assuming that $H$ fixes a point of $M$.  In this case, there is a more refined version of the holonomy, which we call \emph{isotropy}.  It is a straightforward extension of the classical notion of isotropy.  Namely, choosing any $\hat{x} \in \pi^{-1}(x)$---which corresponds to choosing a two-jet of a conformal normalized frame at $x$---determines a representation $\Stab_H(x) \rightarrow P$ sending $h$ to the unique $\hat{h}$ satisfying $h.\hat{x}.\hat{h}^{-1} = \hat{x}$.  If $\{h_k \}$ is a sequence in $\Stab_H(x)$, then $\{ \hat{h}_k \}$ defined in this way in terms of a choice of $\hat{x}$ is a holonomy sequence for $\{ h_k \}$ at $x$.  

Given a one-parameter subgroup $\{ h^t \} \leq \Stab_H(x)$, we may choose $\hat{x}$ so that $\{ \hat{h}^t \} \leq G_0$, by \cite[Thm 1.2]{fm.champsconfs} and proposition B. (Since $H$ is abelian, we may actually arrange that the full isotropy subgroup is in $G_0$, as is proved in proposition \ref{prop:simult_linear} below.)  

Fixed points will fall into two cases, according to whether they contain hyperbolic or unipotent isotropy.  We may assume that no one-parameter subgroup of the isotropy is hyperbolic of contracting type by proposition \ref{prop:contracting_proof}, or of balanced type by proposition \ref{prop:balanced_proof}.   In the remaining mixed and bounded hyperbolic cases, we ``slide'' along an invariant, null geodesic to obtain another fixed point, with contracting hyperbolic isotropy, and apply proposition \ref{prop:contracting_proof}.  The contradiction in the unipotent case stems from the incompatibility of the linear dynamics with the Gromov stratification.

\subsection{Jordan decomposition of isotropy of a conformal flow}
\label{sec:isotropy_jordan_decomp}

It is another consequence of Gromov's stratification that isotropy groups, for a real-analytic rigid geometric structure on a closed, simply connected manifold, are algebraic \cite[3.5.B]{gromov.rgs}; in our case, this means they are algebraic in $P$.  It follows (see, eg, \cite[Thm 4.3.3]{morris.ratners.thms.book}) that isotropy groups are closed under Jordan decomposition---that is, if $\{ \phi^t \} \leq \Stab_H(x)$, and if $p^t = p_h^t p_u^t p_e^t$ is the Jordan decomposition of a corresponding isotropy flow into semisimple, unipotent, and elliptic one-parameter subgroups of $P$, then these are the isotropies of flows $\{ \phi^t_h\} , \{ \phi^t_u \}, \{ \phi^t_e \}$, respectively, each belonging to $\Stab_H(x)$, and multiplying together to give $\phi^t$, for all $t$.
%
%
A singularity of  $\{ \phi^t \} \leq \Conf(M,g)$ will be called hyperbolic, unipotent or elliptic if the isotropy at this singularity is of the corresponding type, and the flows $\{ \phi_h^t \}$, $\{ \phi_u^t \}$, and $\{ \phi_e^t \}$ will be called the hyperbolic, unipotent, and elliptic components, respectively, of $\{ \phi^t \}$ at $x$.

\begin{lemma}
\label{lem:hyperbolic_implies_no_unipotent_component}
Let $x \in M$.
Let $S \leq \Conf(M)$ be an abelian group stabilizing $x$, with isotropy $\widehat{S}_{x}$.  If $\widehat{S}_{x}$ contains a nontrivial one-parameter subgroup of hyperbolic elements, then it contains no nontrivial unipotents.  
\end{lemma}

\begin{proof}
Let $\widehat{S}_x$ be defined with respect to $\hat{x} \in \pi^{-1}(x)$, and let $\widehat{X} \in \hat{\lies}_x$ have Jordan decomposition with nontrivial hyperbolic component $\widehat{X}_h$.  Let $\{ \phi^t_h \} \leq \Stab_H(x)$ be as in the discussion above.
Up to conjugacy in $P$, we may assume $\widehat{X}_h$ belongs to the Cartan subalgebra $\liea$.  Let it have parameters $(a,b)$ according to (\ref{eqn:cartan_subalgebra}).
If $b = 0$, then proposition \ref{prop:contracting_proof} applies to $\{ \phi^t_h \}$, leading to conformal flatness and a contradiction of proposition B.  Then $b \neq 0$, and  the centralizer of $X_h$ in $\liep$ is contained in $\liea \oplus \liem \oplus \liep^+$ (as can be seen, for example, from the root space decomposition in section \ref{sec:cartan_decomposition}); in particular, $X_u \in \liep^+$.

Now let $u \in \widehat{S}_x$ be a nontrivial unipotent.  Since $\widehat{S}_x$ is algebraic, $u$ belongs to a one-parameter unipotent subgroup $\{ e^{tY} \}$ with $[X,Y]=0$.    Decompose $Y$ according to the grading $\liep = \lieg_0 + \lieg_1$ as $Y_0 + Y_1$.  The component $[X,Y]_0$ equals $[X_h + X_e, Y_0]$.  Vanishing of the latter implies $Y_0 \in \liea + \liem$,
which means since $Y$ is nilpotent that $Y_0 = 0$.  Thus, if $Y \in \lieg_1 = \liep^+$ were nonzero, the corresponding conformal flow would be nonlinearizable, and Theorem 1.2 of \cite{fm.champsconfs} would imply conformal flatness of $(M,g)$, a contradiction.
\end{proof}

The above lemma implies in particular for $\{ \phi^t \} \leq \Stab_H(x)$ that if $\{ \phi^t_h \}$ is nontrivial, then $\{ \phi^t_u \}$ is trivial.

\begin{proposition}
\label{prop:simult_linear}
Let $x \in M$.
Let $S \leq \Conf(M,g)$ be an abelian group stabilizing $x$, with isotropy $\widehat{S}_{x}$.  Then $S$ is linearizable at $x$---that is, $\widehat{S}_x$ is conjugate in $P$ into $G_0$.  
\end{proposition}

\begin{proof}
Linearizability of $\widehat{S}_x \leq P$ is equivalent to existence of an $\Ad \hat{S}_x$-invariant complement to $\liep$ in $\lieg$.  The complementary subspaces to $\liep$ in $\lieg$ form an affine space modeled on $\Hom(\lieg/\liep, \liep)$, and $P$ acts affinely on this space.  If $\widehat{S}_x$ is compact, such an invariant complement can be obtained by averaging. 

An element $X \in \liep$ is conjugate in $P$ into $G_0$ if and only if, when decomposed into $X_0 + X_1$ according to the grading $\lieg_0 \oplus \lieg_1$, the element $X_1 \in \lieg_1$ is in the image of $X_0$, for the representation by $\mbox{ad}$ of $\lieg_0$ on $\lieg_1 = \liep^+ \cong \BR^{1,n-1*}$; this condition is independent of the choice of decomposition $\liep \cong \lieg_0 \ltimes \liep^+$.  If $X_1$ as above equals $[X_0,Z]$ for $Z \in \liep^+$, conjugating $X$ by $e^{Z}$ yields $X_0$.

Suppose $\widehat{S}_x$ contains a nontrivial, connected, unipotent subgroup $\widehat{U}_x$.  The projection of $\widehat{\lieu}_x$ to $\lieg_0$ is abelian and comprises nilpotent elements.  Up to conjugation, it annihilates $E_1$, and is thus contained in $\lieu_+$ (where as before $\lieu_+$ is the unipotent subalgebra of $\lieg_0$ parametrized by $U_+$ in section \ref{sec:cartan_decomposition}).  The images in $\liep^+$ of elements of $\lieu_+$ are two-dimensional and degenerate, contained in the annilator of $E_1$.  Then, given two linearly independent elements $X, Y \in \widehat{\mathfrak{u}}_x$, with projections to $\lieg_0$ in $\lieu_+$, we can suppose, thanks again to \cite[Thm 1.2]{fm.champsconfs} and proposition B, that one of them, say, $X$, is in $\lieg_0$ and that $Y$ satisfies the linearizability criterion: if $Y = Y_0 + Y_1$ with $Y_1 \neq 0$, then $Y_1 \in \mbox{im}(Y_0)$.  As $X_0$ and $Y_0$ are linearly independent, the intersection of $\mbox{im}(X_0)$ with $\mbox{im}(Y_0)$ is a line $\ell$, the annihilator of $E^\perp_1$, which equals $\lieg_{\alpha + \beta}$.  If $Y_1 \notin \ell$, then it is not in the image of $X_0 + Y_0$, so $X + Y$ is not linearizable, which would again imply $(M,g)$ is conformally flat and lead to a contradiction.  We conclude that the $\liep^+$-components of all elements in $\widehat{\mathfrak{u}}_x$ belong to $\ell$.  Then a commuting product of conjugations $\Ad(e^{Z_i})$ with $Z_i \in \lieg_{\alpha}$ sends $\widehat{U}_x$ into $G_0$.  By lemma \ref{lem:hyperbolic_implies_no_unipotent_component}, $\widehat{S}_x$ contains no nontrivial, connected, hyperbolic subgroup.  Assuming $\widehat{U}_x$ to be a maximal connected, unipotent subgroup, $\widehat{S}_x$ is the commutative product of $\widehat{U}_x$ with a torus subgroup.  Averaging an $\Ad \widehat{U}_x$-invariant complement to $\liep$ in $\lieg$ over this compact subgroup yields an $\widehat{S}_x$-invariant complementary subspace.

Next suppose $\widehat{S}_x$ contains a nontrivial, connected, hyperbolic subgroup.  
By proposition \ref{prop:contracting_proof} and proposition B, it does not intersect the center of $G_0$.  Then it projects isomorphically to a connected, abelian group of hyperbolic elements of $G_0' \cong \mbox{O}(1,n-1)$, so it comprises a single one-parameter group $ \{ \phi^t_h \}$.  By \cite[Thm 1.2]{fm.champsconfs} and proposition B, this subgroup is linearizable.  Again, by lemma \ref{lem:hyperbolic_implies_no_unipotent_component}, $\widehat{S}_x$ is the commutative product of $\{ \phi^t_h \}$ with a torus subgroup.
As above, we average 
to obtain an $\widehat{S}_x$-invariant complementary subspace.
\end{proof}

\subsection{Hyperbolic fixed points}

The goal of this section is to prove the following: 

\begin{proposition}
\label{prop:no_hyperbolic_singularity}
Let $H$ fix a point $x_0 \in M$.  Then the isotropy at $x_0$ contains no nontrivial hyperbolic one-parameter subgroup.
\end{proposition}

We denote throughout this section by $X$ an infinitesimal generator of $\{ \phi^t \}$, which we assume admits a hyperbolic singularity $x_0 \in M$.
For a well chosen point $\hx_0 \in \pi^{-1}(x_0)$, the image of $X$ in the isotropy algebra $\hat{\lieh}_{x_0}$ belongs to the Cartan subalgebra $\liea$.  Let it correspond to $(a,b)$ in the parametrization (\ref{eqn:cartan_subalgebra}), so the isotropy of $\{ \phi^t \}$ is $\{ d^t \}$ with $\alpha(\ln d^t) = ta$ and $\beta(\ln d^t) = tb$ .  By proposition \ref{prop:contracting_proof}, $|a| \leq |b|$ (If $a < 0$, then we simply replace $X$ with $-X$).  By proposition \ref{prop:balanced_proof}, we can assume $|a| < |b|$.


Let, as above, $\mathcal{U} \subseteq \mathfrak{g}_{-1}$ be an open neighborhood of $0$ on which $\Phi = \pi \circ \exp_{\hat{x}_0} : \mathcal{U} \rightarrow M$ is a diffeomorphism onto its image.  In the basis  $\{ E_1, \ldots, E_n\}$,  $\left. \Ad(d^t) \right|_{\mathfrak{g}_{-1}}$ has the form
\begin{equation*}
h^t = 
\begin{pmatrix}
e^{t(b-a)} & & & & \\
 & e^{-ta} & & & \\
 & & \ddots & & \\
 & & & e^{-ta} & \\
 & & & & e^{-t(b+a)}
\end{pmatrix}
\end{equation*}
Note that $E_1$ and $E_n$ are lightlike and orthogonal to $E_2, \ldots, E_{n-1}$, which are spacelike.
We see from
$$ \phi^t \circ \Phi = \Phi \circ \Ad(d^t)$$
that the fixed set of $\{ \phi^t \}$ in $\Phi(\mathcal{U})$ is spacelike if $a=0$ and otherwise just the point $x_0$.  

The isotropy of the maximal compact factor $L < H$ at $\hat{x}_0$ has compact image in $\mbox{CO}(T_{x_0} M,g_{x_0})$ and commutes with $\{h^t \}$, so it fixes the lightlike vectors $D_0\Phi(E_1) = \pi_* \omega_{\hat{x}_0}^{-1}(E_1)$ and $D_0\Phi(E_n) = \pi_* \omega_{\hat{x}_0}^{-1}(E_n)$.
Let $\hat{\alpha}(s) = \exp_{\hat{x}_0} (s E_1)$ and $\alpha = \pi \circ \hat{\alpha}$.  The latter curve is a lightlike geodesic.  The calculation in $\SO(2,n)$ that
$$ d^t e^{s E_1} = e^{A_ts E_1} d^t \qquad A_t = e^{t(b-a)} $$
gives in $\widehat{M}$ by (\ref{eqn:isotropy_exp_curves})
\begin{eqnarray}
  \label{eqn.axn.on.alphahat}
  \phi^t. \hat{\alpha}(s) & = & \hat{\alpha}(A_ts). d^t \qquad A_t = e^{t(b-a)} 
\end{eqnarray}

  In particular,
$$ \phi^t .\alpha(s) = \alpha(e^{t(b-a)} s),$$
and $\alpha$ and $\hat{\alpha}$ are defined for all $s \in \BR$.  The curve $\alpha$ is moreover pointwise fixed by $L$.  Then $\alpha^+ = \{ \alpha(s) \ : \ s > 0 \}$ is an $H$-orbit.  Let $x_1$ be an accumulation point of $\{ \alpha(s_k) \}$ for a sequence $s_k \rightarrow \infty$.  It is in the closure $\overline{\alpha^+}$.

\begin{lemma} $\overline{\alpha^+} = \{ x_0 \} \cup \alpha^+ \cup \{ x_1 \}$ and $x_1$ is a fixed point for $H$.
\end{lemma}

\begin{proof} First, note that for the locally closed orbit $\alpha^+ = \{ \phi^t. \alpha(\epsilon) \ : t \in \BR \}$, the orbit map $t \mapsto \alpha(e^{t(b-a)})$ is a homeomorphism onto its image (see \cite[Lem. 2.1.15]{zimmer.etsg}).

Theorem \ref{thm:gromov_stratification} (3) gives a neighborhood $U \subseteq M$ of $x_1$ such that $U \cap \overline{\alpha^+}$ is connected and semianalytic, thus path connected.  In particular, for any sufficiently large $T \in \BR_{>0}$, there is an injective, continuous curve $\gamma : [0,1] \rightarrow U \cap \overline{\alpha^+}$ with $\gamma(0) = \alpha(T)$ and $\gamma(1) = x_1$.  Because $\alpha^+$ is open in $\overline{\alpha^+}$, the set $\gamma^{-1}(\alpha^+)$ is a collection of relatively open intervals of $[0,1]$, including one of the form $[0,t_0)$ with $t_0 > 0$.  Now $\gamma(0,t_0)$ is open, and $\alpha^{-1}(\gamma(0,t_0))$ is a nonempty open interval $I$, with one endpoint equal $T$.  The other endpoint must be $0$ or $\infty$, because $\gamma(t_0) \in \overline{\alpha(I)} \backslash \alpha^+$.  
Assuming $x_0 \neq x_1$, we may choose $U$ so that $x_0 \notin U$ and $I = [T,\infty)$.  If $x_0 = x_1$, then choosing $T$ large ensures that again $I = [T,\infty)$.  Thus $\mbox{im } \gamma = \alpha[T,\infty) \cup \{ x_1 \}$, and $x_1 = \lim_{s \rightarrow \infty} \alpha(s)$.  
It follows that $\overline{\alpha^+} = \{ x_0 \} \cup \alpha^+ \cup \{ x_1 \}$, as desired.

The sets $\{ x_0 \}$, $\alpha^+$ and $\overline{\alpha^+}$ are each $H$-invariant, so $\{ x_1 \}$ is also $H$-invariant.
\end{proof}

Now, in a neighorhood of $x_1$, the segment $\overline{\alpha^+}$ is a null geodesic segment.  In suitable coordinates, it is a line segment in the light cone of $x_1$.  We have in particular the following corollary:

\begin{corollary}
\label{cor:smooth_parametrization}
The path $\overline{\alpha^+}$ admits a smooth parametrization by a finite interval.
\end{corollary}

Consider the one-parameter subgroup of $\SO(2,n)$ generated by

\begin{equation*}
R =
\begin{pmatrix}
0 & - 1 & & & \\
1 & 0 & & & \\
 & & 0 & & \\
 & & & 0 & 1 \\
 & & & -1 & 0
\end{pmatrix}
\ \  
\mbox{with}
\ \
e^{\theta R} =
\begin{pmatrix}
  \cos \theta & - \sin \theta &  &  &  \\
  \sin \theta & \cos \theta &  & &  \\
  &  &  I_{n-2}  &  &  &  \\
  &  &   &  \cos \theta & \sin \theta \\
  &  &  &   - \sin \theta & \cos \theta
  \end{pmatrix}
\end{equation*}

Let $\hat{\beta}(\theta) = \exp(\hat{x}_0,\theta R)$ and $\beta = \pi \circ \hat{\beta}$. In $\Ein^{1,n-1}$, there is equality $e^{\theta R}.o = e^{s E_1}.o$ with $s = \tan \theta$,  $0 \leq \theta < \pi/2$.  Thus there is a smooth path $\ell_\theta$ in $P$ such that
$$ e^{\theta R} = e^{s E_1} . \ell_\theta \qquad s = \tan \theta, 0 \leq \theta < \pi/2$$

Differentiating with respect to $\theta$ and evaluating with $\omega_G$ gives, by formula (\ref{eqn:omega_along_curves})
$$ R = \Ad \ell_\theta^{-1} \left( \frac{ds}{d\theta}  E_1 \right) + \omega_P \left( \frac {d \ell_\theta}{d \theta} \right) $$
which in turn implies
\begin{equation}
  \label{eqn:reln.betahat.alphahat}
  \hat{\beta}(\theta) = \hat{\alpha}(s). \ell_\theta \qquad s = \tan \theta, \qquad 0 \leq \theta < \pi/2
  \end{equation}

\begin{lemma} There is $\hat{x}_1 \in \pi^{-1}(x_1)$ equal to $\lim_{\theta \rightarrow \pi/2} \hat{\beta}(\theta)$.
  \end{lemma}

\begin{proof}
  From (\ref{eqn:reln.betahat.alphahat}), we have $\beta(\theta) = \alpha(\tan \theta)$.
  Then both $\hat{\beta}$ and $\beta$ are defined for $0 \leq \theta < \pi/2$. 
  Defining $\beta(\pi/2) = x_1$ extends the domain of $\beta$ to $[0,\pi/2]$.  The resulting path is a monotone parametrization of $\overline{\alpha^+}$ by $[0,\pi/2]$, smooth on $[0,\pi/2)$.  By corollary \ref{cor:smooth_parametrization}, it is smooth on $[0,\pi/2]$.  We wish to extend $\hat{\beta}$, as well.


  Let $\tilde{\beta}$ be a path in $\widehat{M}$, smooth on $[0,\pi/2]$, originating at $\hat{x}_0$, such that $\pi \circ \tilde{\beta}$ equals the extended $\beta$ on $[0,\pi/2]$.  There is a smooth path $p_{\theta}$ in $P$ for which
  $$ \tilde{\beta}(\theta) = \hat{\beta}(\theta) . p_{\theta} \qquad 0 \leq \theta < \pi/2$$
  By equation (\ref{eqn:omega_along_curves}),
  $$ \omega(\tilde{\beta}'(\theta))  = \Ad p_\theta^{-1} (R)+ \omega_P( p_\theta') \qquad 0 \leq \theta < \pi/2$$
  which is equivalent to
  $$ R = \Ad p_\theta \left( \omega(\tilde{\beta}'(\theta)) - \omega_P( p_\theta') \right)$$
  We will compute the derivative with respect to $\theta$ of the above equation, with the aid of the following identity.
  For $g = g(\theta)$ a path in a Lie group $G$ with Maurer-Cartan form $\omega_G$,
 and for $X = X(\theta)$ a path in $\lieg$,
  $$ \left( g X g^{-1} \right)^\prime = g \left( [\omega_G (g'), X] + X^\prime \right) g^{-1}$$

  Using this identity gives
  $$ 0 = p_\theta \left( [\omega_P (p_\theta'), \omega(\tilde{\beta}')] + \DDth \omega(\tilde{\beta}') - \DDth \omega_P( p_\theta') \right) p_\theta^{-1} $$
which is equivalent to
  $$ \DDth \omega_P( p_\theta') =  [\omega_P (p_\theta'), \omega(\tilde{\beta}')] + \DDth \omega(\tilde{\beta}') $$

  This is a linear, first-order, inhomogeneous ODE on $\omega_P( p_\theta')$ in $\lieg$, for which the coefficient functions are smooth on $[0,\pi/2]$.  For any initial value---in this case, $\omega(\tilde{\beta}'(0)) - R$---the unique solution is defined and continuous on $[0,\pi/2]$ (see, eg, \cite[Sec 3.4]{teschl.ode.dyn.syst}).  The path $p_\theta$ is known to be smooth on $[0,\pi/2)$, and now it is also known to have bounded $\omega$-derivative.  It follows by completeness of $P$ that $\lim_{\theta \rightarrow \pi/2} p_\theta$ exists, and therefore, so does $\lim_{\theta \rightarrow \pi/2} \hat{\beta}(\theta)$.
    \end{proof}

\begin{lemma}
  The holonomy of $\phi^t$ at $x_1$ with respect to $\hat{x}_1$ is
\begin{equation*}
g^t = 
\begin{pmatrix}
e^{tb} & & & & \\
 & e^{ta} & & & \\
 & & I_n & & \\
 & & & e^{-ta} & \\
 & & & & e^{-tb}
\end{pmatrix}
\end{equation*}
    \end{lemma}

\begin{proof}
Combining (\ref{eqn.axn.on.alphahat}) and (\ref{eqn:reln.betahat.alphahat}) yields
  $$ \phi^t .\hat{\beta}(\theta) = \hat{\beta}(\theta_t).(\ell_{\theta_t}^{-1} d^t \ell_\theta) \qquad \theta_t = \tan^{-1}(A_t \tan \theta), \ A_t = e^{t(b-a)}$$
The desired holonomy with respect to $\hat{x}_1$ is the limit of $(\ell_{\theta_t}^{-1} d^t \ell_\theta)$ as $\theta \rightarrow \pi/2$.  It is possible to compute this limit directly, but easier to compute in $\SO(2,n)$, where the above curves develop to
$$d^t e^{\theta R} = e^{\theta R} . (e^{- \theta R} d^t e^{\theta R}) = e^{\theta_t R}.(\ell_{\theta_t}^{-1}d^t \ell_\theta) $$
The conjugate $e^{- \theta R} d^t e^{\theta R}$ is not in $P$ for arbitrary $\theta$, but it is for $\theta = \pi/2$, and equals $g^t$.  Then
$$ g^t = \lim_{\theta \rightarrow \pi/2} (\ell_{\theta_t}^{-1} d^t \ell_\theta) \qquad \mbox{so} \qquad \phi^t. \hat{x}_1 = \hat{x}_1 . g^t$$
as desired.
\end{proof}

The holonomy $\{ g^t \}$ is in contradiction with proposition \ref{prop:contracting_proof} and proposition B.  The proof of Proposition \ref{prop:no_hyperbolic_singularity} is complete.

\subsection{Unipotent fixed points}

The goal of this section is to prove the following:

\begin{proposition}
\label{prop:no_unipotent_singularity}
Suppose $H$ fixes a point $x_0 \in M$.  Then the isotropy of $H$ at $x_0$ contains no nontrivial unipotent elements.
\end{proposition}

\begin{proof}
Denote $\widehat{H}_{x_0}$ the isotropy of $H$ at $x_0$ with respect to $\hat{x}_0$, and suppose it contains nontrivial unipotent elements.  By lemma \ref{lem:hyperbolic_implies_no_unipotent_component}, there are no nontrivial one-parameter hyperbolic subgroups of $\widehat{H}_{x_0}$, and by proposition \ref{prop:simult_linear}, we may assume $\widehat{H}_{x_0} < G_0$. 
  
Denote $\widehat{U}_{x_0}$ a maximal connected, unipotent subgroup of $\widehat{H}_{x_0} < G_0$.  We may assume $\widehat{U}_{x_0}$ belongs to the maximal unipotent subgroup of $\SO(1,n-1) < G_0$ with 
 isotropic characteristic line $\BR E_1$.  
The full $1$-eigenspace, call it $\mathcal{E}'$, of $\widehat{U}_{x_0}$ is degenerate and codimension at least two.
  
For $L$ a maximal compact factor of $H$, the linear isotropy $\widehat{L}_{x_0}$ of $L$ at $x_0$ with respect to $\hat{x}_0$ is compact and commutes with $\widehat{U}_{x_0}$, so it fixes $E_1$ and another isotropic vector, which we may assume to be $E_n$.  It preserves $\mathcal{E}'$ and pointwise fixes the spacelike subspace $\mathcal{E} = (\mathcal{E}' \oplus \BR E_n)^\perp$, which is contained in the image of $\widehat{\lieu}_{x_0}$. The full isotropy $\widehat{H}_{x_0}$ has the form $\widehat{U}_{x_0} \times \widehat{L}_{x_0} < G_0$. 

The orbit of any $v \in \BR^{1,n-1} \backslash \{ 0 \}$ under the unipotent isotropy is bounded away from the origin, and thus so is the orbit of $v$ under the full isotropy $\widehat{H}_{x_0}$.  Let $y_0  = \pi \circ \exp_{\hat{x}_0}(v)$ for $v \neq 0$ in a sufficiently small neighborhood $\mathcal{V}$ of $0$ in $\lieg_{-1} \cong \BR^{1,n-1}$, mapping diffeomorphically to $V = \pi \circ \exp_{\hat{x}_0}(\mathcal{V}) \subset M$.  
The connected component of $y_0$ in $H.y_0 \cap V$ equals $\pi \circ \exp_{\hat{x}_0}(\widehat{H}_{x_0} .v \cap \mathcal{V})$ and is bounded away from $x_0$.  By theorem \ref{thm:gromov_stratification}, $\overline{H.y_0}$ is locally connected and contains $H.y_0$ as a relatively open subset; it follows that $x_0 \notin \overline{H.y_0}$ for any $y_0 \neq x_0$ in $V$.

Now consider $v \in (\mathcal{E} \oplus \BR E_n) \cap \mathcal{V}$, so that $y_0$ is fixed by $L$ but not by $H$.
By theorem \ref{thm:gromov_stratification}, there is $x_1 \in \overline{H.y_0}$ with compact $H$-orbit.  As $x_1$ is $L$-fixed, it is a fixed point for $H$.  By proposition \ref{prop:simult_linear}, $H$ is again linearizable at $x_1$.  By 
proposition \ref{prop:no_hyperbolic_singularity}, the isotropy at $x_1$ is again the product of a nontrivial unipotent subgroup of dimension $k = \dim H/L$ and a compact subgroup isomorphic to $L$. We can choose $\hat{x}_1 \in \pi^{-1}(x_1)$ so that the unipotent factor of $\widehat{H}_{x_1}$ has the same $1$-eigenspace $\mathcal{E}'$ as for $\widehat{H}_{x_0}$ above.

Any point $y_1 \in H.y_0$ is $L$-fixed but not $H$-fixed.  Assuming $y_1$ is sufficiently close to $x_1$, it can be written $\pi \circ \exp_{\hat{x}_1}(v)$ for $v \in \BR^{1,n-1} \backslash \{ 0 \}$.
But now, the same argument as above  with the linearization shows that the $H$-orbit of $y_1$ is bounded away from $x_1$, contradicting $x_1 \in \overline{H.y_1}$.
\end{proof}

\subsection{Conclusion}

Henceforth we will assume $L \neq 1$.  Here is one corollary of the absence of fixed points:

\begin{proposition}
\label{prop:1d_implies_closed}
Any $1$-dimensional $H$-orbit is closed.  
\end{proposition}

\begin{proof}
If $H.x$ is $1$-dimensional and not closed, then 
$L \leq \mbox{Stab}_H(y)$ for all $y \in H.x$.  By theorem \ref{thm:gromov_stratification} (2), there is $y \in \overline{H.x}$ with closed orbit.  One the one hand, $\mbox{Stab}_H(y)$ is cocompact in $H$; on the other hand, $L \leq \mbox{Stab}_H(y)$.  Therefore, $y$ is an $H$-fixed point, a contradiction.  
\end{proof}

\section{End of proof using approximately-stable hypersurfaces}
\label{sec:end_of_proof}

In this final section, $(M,g)$ is again a closed, simply connected, analytic Lorentzian manifold, and $H$ is a maximal, connected, abelian subgroup of $\Conf(M,g)$, obtained from a maximal commuting collection of conformal vector fields, which we assume noncompact.  The results of all previous sections apply.  
From section \ref{sec:fixed_points} just above, $H$ has no global fixed points; in particular, $\dim H \geq 2$.   

Under these assumptions, we deduce a detailed description of holonomy sequences at points of $\Omega_f$ (see proposition \ref{prop:freeness}) in section \ref{sec:analysis_holonomy}.  They are balanced, bounded, or mixed, as in definition \ref{def:holonomy_trichotomy}.  If they are balanced, then we show they are also linear, so proposition \ref{prop:balanced_proof} applies.  The bounded and mixed cases occupy the remainder of the paper.  These holonomy sequences need not be linear; they are moreover not ``stable'' in the sense of \cite{frances.degenerescence}, which makes their local behavior much more difficult to analyze.  We are nonetheless able to construct a codimension-one foliation by degenerate hypersurfaces on an open dense subset $\Omega_{\mathcal{F}} \subseteq \Omega_f$, in section \ref{sec:foliation_construction}.  The construction roughly follows the outline of section \ref{sec:balanced_proof}, but it is more delicate.  In section \ref{sec:extension_closed_isotropic} we show that the foliation extends over closed isotropic orbits on the boundary of orbits in $\Omega_{\mathcal{F}}$; on the latter, $H$ necessarily has linear, unipotent isotropy.  
In section \ref{sec:final_contradiction}, many such distinct closed, isotropic orbits are found, and the prevalence of isotropic orbits finally contradicts the local geometry of the linear, unipotent isotropy.

Some of the proofs in this section make extensive use of elementary properties of the grading of $\lieg = \so(2,n)$ and of its root space decomposition.  The reader not intimately familiar with the structure of this Lie algebra can verify these properties by referring to the descriptions in sections
\ref{sec:equivalence_principle} and \ref{sec:cartan_decomposition}.

\subsection{Analysis of holonomy sequences in the absence of a fixed point}
\label{sec:analysis_holonomy}

Let $x \in M$ with $H.x$ not closed.  By theorem \ref{thm:gromov_stratification}, there is a closed orbit $H.y$ in the closure $\overline{H.x}$.  By the results of the previous section, $y$ is not an $H$-fixed point.  Let $h_k \in H$ be such that $h_k.x \rightarrow H.y$.
As $L$ acts transitively on $H.y$, we may translate $\{ h_k \}$ by a sequence in $L$ to obtain $h_k.x \rightarrow y$ (without changing the notation for $\{ h_k \}$).  Let $\{ p_k \}$ be a holonomy sequence in $A'P^+$-form for $\{ h_k \}$. 

 Proposition \ref{prop:simult_linear} gives that the isotropy of $\mbox{Stab}_H(y)$ is linearizable.  
Replacing $\{ p_k \}$ with $\{  q p_k  \}$ for suitable $q \in P^+$ gives another holonomy sequence (again denoted $\{ p_k \}$) with $h_k .\hat{x}.p_k^{-1} \rightarrow \hat{y}$ such that the isotropy of $\Stab_H(y)$ with respect to $\hat{y}$ is in $G_0$.  Denote this isotropy by $\widehat{S}_{y}$. 
Note that the new $\{ p_k \}$ can also be written in $A'P^+$-form, which will be denoted $\{ d_k \tau_k \}$.  An \emph{ACL holonomy} sequence  will be one associated to $h_k.x_k \rightarrow y$ with the following properties:
\begin{itemize}
\item[{\bf A})] $\{ p_k \}$ in $A'P^+$-form
\item[{\bf C})] $H.y$ a closed orbit
\item[{\bf L})] linear isotropy in $G_0$ of $\Stab_H(y)$ with respect to $\hat{y}$
\end{itemize}
Note that any $x$ with nonclosed orbit admits an ACL pointwise holonomy sequence as a consequence of theorem \ref{thm:gromov_stratification} and proposition \ref{prop:simult_linear}.  
 
Now, $\mbox{Stab}_H(y)$ is cocompact in $H$, with positive codimension.  There is thus $Y \in \mathfrak{l}$ with $Y(y) \neq 0$.  If $x \in \Omega_f$ (see proposition \ref{prop:freeness}) or if $\{ p_k \}$ is a pointwise holonomy sequence, then necessarily $Y(x) \neq 0$.  Asuming $H.x$ is not closed, there is $X \in \lieh \backslash \mathfrak{l}$ vanishing at $y$ but not at $x$.  If $x \in \Omega_f$ or if $\{ p_k \}$ is a pointwise holonomy sequence, then $X(x)$ and $Y(x)$ must be linearly independent.  By lemma \ref{lem:hyperbolic_implies_no_unipotent_component}, we may assume $X$ has unipotent or hyperbolic isotropy with respect to $\hat{y}$; ensuring this may require adding to $X$ an element of $\mathfrak{stab}_\lieh(y) \cap \mathfrak{l}$, which does not affect independence of $X$ and $Y$ at $x$ under the assumptions on $x$.  We will call such a pair $(X,Y)$ of elements of $\lieh$ with
\begin{itemize}
\item $X \in \mathfrak{stab}_\lieh(y)$ with unipotent or hyperbolic isotropy with respect to $\hat{y}$;
\item $Y(y) \neq 0$;
\item $X,Y$ linearly independent at $x$
\end{itemize}
a \emph{leverage pair} for the holonomy sequence $\{ p_k \}$.  A leverage pair for the ACL holonomy sequence $\{ p_k \}$ exists whenever $x \in \Omega_f$ or $\{ p_k \}$ is a pointwise holonomy sequence.
As usual, we denote also by $X,Y$ the lifts to $\widehat{M}$.  

\begin{remark}
The following propositions \ref{prop:all_about_holonomy},  \ref{prop:e1perp_stable}, \ref{prop:distribution_in_open_dense}, and \ref{prop:propagation_mixed_bounded}, as well as lemma \ref{lemma:holonomy_second_order}, are stated for points $x \in \Omega_f$, since that is what is used in the sequel; however, they apply as well to pointwise holonomy sequences at any point with nonclosed orbit (with $\Omega_{\mathcal{F}}$ suitably redefined). 
\end{remark}

\begin{proposition}
\label{prop:all_about_holonomy}
Let $\mathbb{V} = \lieg / \liep$. Let $x \in \Omega_f$
and let $\{ p_k = d_k \tau_k \}$ be an ACL holonomy sequence at $x$.   
Then $\{p_k \}$ falls into one of the following three cases:
\begin{enumerate}
\item \emph{balanced}:  $\{ p_k \}$ is vertically equivalent to a linear holonomy sequence
\end{enumerate}
Otherwise, $\mathbb{V}^{AS}(p_k) = E_1^\perp$ and $\lieh(x)$ is in the approximately stable subset of $T_xM$, which is a degenerate hyperplane.  Furthermore, $\{ p_k \}$ is 
\begin{enumerate}[start=2]
\item \emph{mixed}: $H.y$ is one-dimensional and isotropic, and $\hat{X} \in \lieg_0$ is unipotent; or
\item \emph{bounded}: $X(x)$ is in the orthogonal of the approximately stable subset of $T_xM$.
\end{enumerate}
\end{proposition}

\begin{proof}
Let $\hat{x}_k \rightarrow \hat{x} \in \pi^{-1}(x)$ be such that $h_k.
\hat{x}_k.p_k^{-1} \rightarrow \hat{y}$.
Let $D_k = \ln d_k$ and $\xi_k = \ln \tau_k$.  Recall that $d_k \in A'$ implies $\beta(D_k) \leq 0$. 

Let $\widehat{X}_k$ and $\widehat{Y}_k$ equal $\omega_{\hat{x}_k}(X)$ and $\omega_{\hat{x}_k}(Y)$, respectively, and denote with superscripts their components on the grading $\lieg = \lieg_{-1} \oplus \lieg_0 \oplus \lieg_1$.  The components of $\Ad p_k(\widehat{X}_k)$ are
\begin{multline}
\label{eqn:holo_on_X_k}
\underbrace{\Ad(d_k). \widehat{X}_k^{(-1)}}_{\lieg_{-1}}
+ \underbrace{\Ad(d_k). \left( \widehat{X}_k^{(0)} + [\xi_k,\widehat{X}_k^{(-1)}] \right) }_{\lieg_0} \\
+ \underbrace{\Ad(d_k). \left (\widehat{X}_k^{(1)} + [\xi_k,\widehat{X}_k^{(0)}] + \frac{1}{2}[\xi_k,[\xi_k,\widehat{X}_k^{(-1)}]] \right) }_{\lieg_1}
\end{multline}
and similarly for $\Ad p_k(\widehat{Y}_k)$; moreover, by the $H$-invariance and $P$-equivariance of $\omega$, each component above converges to the corresponding component of $\omega_{\hat{y}}(X)$ or $\omega_{\hat{y}}(Y)$, respectively.  

The eigenvalues of $\Ad(d_k)$ on $\lieg_{-1}$ are $e^{-\alpha(D_k) \pm \beta(D_k)}$, each with multiplicity $1$, and $e^{-\alpha(D_k)}$ with multiplicity $n-2$.
Now $X(y) = 0$ means $\omega_{\hat{y}}(X)^{(-1)} = 0$.  As $X(x) \neq 0$, the vectors $X_k^{(-1)} \neq 0$.  Thus $\Ad(d_k).\widehat{X}_k^{(-1)} \rightarrow 0$ implies $(- \alpha + \beta)(D_k) \rightarrow - \infty$.
On the other hand, both $Y(x)$ and $Y(y)$ are nonzero, so at least one eigenvalue of $\Ad(p_k)$ on ${\lieg_{-1}}$ has a limit in $\BR_{>0} \cup \{+\infty\}$. Thus the largest, $(-\alpha-\beta)(D_k)$, is bounded below. Together these limits imply $\beta(D_k) \rightarrow -\infty$.
 
 Finally, $\alpha(D_k)$ cannot diverge to $- \infty$.  If it did, then, since $(\beta - \alpha)(D_k) \rightarrow - \infty$, both $\omega_{\hat{x}}(X)^{(-1)}$ and $\omega_{\hat{x}}(Y)^{(-1)}$ would belong to the 1-dimensional negative weight space of $\beta$ in $\lieg_{-1}$ (identified with $\BR E_1$ in the basis introduced in section \ref{sec:cartan_decomposition}), contradicting the fact that $X(x)$ and $Y(x)$ are linearly independent. 
 
We are left with the three possibilites of balanced, mixed, or bounded types for $\{ d_k \}$. The action of $p_k$ on $\mathbb{V}$ factors through $d_k$.  We identify $\mathbb{V}$ as a $G_0$-module with $\BR^{1,n-1}$.

In the mixed and bounded cases $(\alpha + \beta)(D_k) \rightarrow - \infty$, and it is clear that $\mathbb{V}^{AS}(p_k) = E_1^\perp$.  Every $Z \in \lieh$ is $H$-invariant, so $\omega_{\hat{x}}(Z) \in \lieg^{AS}(p_k)$.  Then the projections of $\omega_{\hat{x}}(Z)$ modulo $\liep$ are in $\mathbb{V}^{AS}(p_k)$, which means $Z(x)$ is in the approximately stable subset of $T_xM$.  

Now suppose $\{ p_k \}$ is of mixed type, and consider its action on $\mathbb{V}$.  For any bounded sequence $v_k \in \mathbb{V}$, the components of $(\Ad d_k).v_k$ on $E_1^\perp$ tend to $0$.  If $\lim \Ad(d_k). v_k$ exists, it belongs to $\BR E_n$.  It follows that $\lieh(y)$ is one-dimensional, and $Y(y) = q(\hat{y}, \lim \Ad(d_k).\hat{Y}_k^{(-1)})$ is isotropic (where $q : \widehat{M} \times \mathbb{V} \rightarrow TM$ as in section \ref{sec:approximately_stable}).  Now $Y(y)$ is fixed by the derivative action of $\Stab_H(y)$.  Hyperbolic elements of $\CO(1,n-1)$ have isotropic fixed vectors only if they are of balanced type.  We conclude by proposition \ref{prop:balanced_proof} that there are no hyperbolic elements of $\hat{S}_{y}$, so $\hat{X}$ is unipotent.

  Now suppose $\{ p_k \}$ is of bounded type, and consider again the action on $\mathbb{V}$.  For a convergent sequence $v_k \in \mathbb{V}$, if the limit $\Ad(d_k) .v_k$ exists, it is nonzero unless $\lim v_k \in \BR E_1$, the orthogonal of $\mathbb{V}^{AS}(p_k)$.  We conclude $X(x)$ is in the orthogonal of the approximately stable subset of $T_xM$.
   
 Now suppose $\{ p_k \}$ is of balanced type.   By similar considerations as above of the action of $\{ d_k \}$ on $\mathbb{V}$, this behavior forces $\omega_{\hat{x}}(X)^{(-1)} \in E_1^\perp$ and $\omega_{\hat{x}}(Y)^{(-1)}$ to be transverse to $E_1^\perp$.  After rescaling, we may write $\omega_{\hat{x}}(Y)^{(-1)} = E_n + F$ with $F \in E_1^\perp$, and $\widehat{Y}_k^{(-1)} = c_k E_n + F_k$ with $c_k \rightarrow 1$ and $F_k \in E_1^\perp$.  We will prove that in this case, $\{ \xi_k \}$ is bounded, with a rather detailed analysis of equation (\ref{eqn:holo_on_X_k}).
 
 To simplify this analysis, we first make a slight modification to $\{ \widehat{X}_k^{(-1)} \}$.  As shown above, it has the form $\epsilon_k E_n + G_k$ with $G_k \in E_1^\perp$ and $\epsilon_k \rightarrow 0$.  Applying $\Ad p_k$ to the sequence $\widehat{X}'_k = \omega_{\hat{x}_k}(X - \epsilon_k Y/c_k)$ gives a bounded sequence in $\lieg$, tending to $\omega_{\hat{x}}(X)$.  We make this replacement without change of notation---that is, we simply assume $\widehat{X}_k^{(-1)} \in E_1^\perp$.
 
 We consider the second line of equation (\ref{eqn:holo_on_X_k}), corresponding to the $\lieg_0$-components of $\Ad p_k(\widehat{X}_k)$ and $\Ad p_k(\widehat{Y}_k)$, and decompose it further into components on the root spaces in $\lieg_0$, which correspond to $- \beta, 0,$ and $\beta$.  As $\widehat{X}_k^{(0)}$ 
 and $\Ad(d_k).\left(\widehat{X}_k^{(0)} + \left[\xi_k,\widehat{X}_k^{(-1)} \right] \right)$ are bounded, the components $ [\xi_k,\widehat{X}_k^{(-1)}]_{- \beta}$ and $\left[\xi_k, \widehat{X}_k^{(-1)} \right]_0$ must be bounded.  The same goes for $\widehat{Y}_k$.  Using that $\widehat{X}_k \in E_1^\perp$, we begin with 
 $$ \left[\xi_k,\widehat{X}_k^{(-1)} \right]_0 = \left[(\xi_k)_{\alpha - \beta},(\widehat{X}_k^{(-1)})_{\beta - \alpha} \right] + \left[ (\xi_k)_{\alpha},(\widehat{X}_k^{(-1)})_{- \alpha} \right]$$
  The first and second terms on the right-hand side are linearly independent, so each must be bounded.  At least one of $\omega_{\hat{x}}(X)^{(-1)}_{\beta - \alpha}$ and $\omega_{\hat{x}}(X)^{(-1)}_{ - \alpha}$ is nonzero, so at least one of $(\xi_k)_{\alpha - \beta}$ and $(\xi_k)_{\alpha}$ is bounded, by the nondegeneracy of the brackets between opposite root spaces given by lemma \ref{lem:bracket_nondegeneracy}.
  
  Next consider
  $$ \left[ \xi_k,\widehat{Y}_k^{(-1)} \right]_0 = \left[ (\xi_k)_{\alpha - \beta},(\widehat{Y}_k^{(-1)})_{\beta - \alpha} \right]  + \left[ (\xi_k)_{\alpha},(\widehat{Y}_k^{(-1)})_{- \alpha} \right] + c_k \left[ (\xi_k)_{\alpha + \beta},E_n \right] $$
At least one of the first two terms is bounded, so each of the remaining two must be bounded; we conclude $(\xi_k)_{\alpha + \beta}$ is bounded.

Supposing $\omega_{\hat{x}}(X)^{(-1)}_{\beta - \alpha} = 0$, we know $\lim (\widehat{X}_k^{(-1)})_{- \alpha} \neq 0$ and $\{ (\xi_k)_{\alpha} \}$ is bounded.  Now
$$ \left[ \xi_k,\widehat{X}_k^{(-1)} \right]_{ - \beta} = \left[ (\xi_k)_{\alpha - \beta},(\widehat{X}_k^{(-1)})_{- \alpha} \right]$$
implies that $\{ (\xi_k)_{\alpha - \beta} \}$ is bounded because, for any nonzero $\xi \in \lieg_{\alpha - \beta},$ the map $\ad \xi : \lieg_{- \alpha} \rightarrow \lieg_{- \beta}$ is an isomorphism.  On the other hand, if $\omega_{\hat{x}}(X)^{(-1)}_{ - \alpha} = 0$, then we know $\{ (\xi_k)_{\alpha - \beta} \}$ is bounded, and
$$ \left[ \xi_k,\widehat{Y}_k^{(-1)} \right]_{- \beta} = \left[ (\xi_k)_{\alpha - \beta},(\widehat{Y}_k^{(-1)})_{ - \alpha} \right]  + c_k \left[ (\xi_k)_{\alpha},E_n\right],$$
so $\{ (\xi_k)_\alpha \}$ is bounded, because $\ad E_n : \lieg_\alpha \rightarrow \lieg_{- \beta}$ is an isomorphism.

Finally, in the balanced case, $\{ \xi_k \}$ is bounded, which means $\{ p_k = d_k \tau_k \}$ is vertically equivalent to the linear sequence $\{ d_k \}$.
\end{proof}

By proposition \ref{prop:proof_with_finite_volume}, there is a dense set of points of $\Omega_f$ satisfying condition (\ref{eqn:negation_finite_volume}) of section \ref{subsec:volume_condition}.  This condition gives holonomy sequences, which we may assume in $A'P^+$-form $\{ p_k = d_k \tau_k \}$, with $\alpha(D_k) \rightarrow  \infty$, as $\alpha$ measures the negative of the logarithm of the dilation of $\Ad(d_k)$ on $\lieg / \liep$. (Note that they are in general not pointwise holonomy sequences).  Unboundedness of $\{ \alpha(D_k) \}$ implies $\{ p_k \}$ is of balanced or mixed type; thanks to proposition \ref{prop:balanced_proof}, it must be mixed.  
Given such a holonomy sequence associated to $h_k .x_k \rightarrow y$, the limit orbit $H.y$ is $1$-dimensional, as in the proof of the mixed case (2) above.   The orbit $H.y$ is then closed by proposition \ref{prop:1d_implies_closed}.  The isotropy of $\Stab_H(y)$ is linearizable by proposition \ref{prop:simult_linear}, so changing $\hat{y}$ in the fiber of $y$ if necessary gives an ACL holonomy sequence vertically equivalent to $\{ p_k \}$.

\begin{corollary}
\label{cor:dense_mixed}
There is a dense set of points $S_{mix} \subset \Omega_f$ such that for all $x \in S_{mix}$, there is $\{ h_k \} \subset H$ with ACL holonomy sequence at $x$ of mixed type.
\end{corollary}

If $\{ p_k \}$ in proposition \ref{prop:all_about_holonomy} above is mixed or bounded, it may not be linear as in the balanced case.  We can arrange, however, that the nonlinear factors lie in a 1-dimensional subgroup of $P^+$.

\begin{lemma}
\label{lemma:holonomy_second_order}
Let $x \in \Omega_f$.  
Any ACL holonomy sequence at $x$ of mixed or bounded type is vertically equivalent to an ACL holonomy sequence $\{ p_k = d_k \tau_k \}$ for which $\{ \xi_k = \ln \tau_k \}$ is contained in the line $\BR^tE_1 \mathbb{I} \subset \liep^+ \cong \BR^{1,n-1*}$.  Moreover, for $D_k = \ln d_k$, the sequence $\{ e^{\beta(D_k)} \xi_k \}$ is bounded.
\end{lemma}

\begin{proof}
Let $(X,Y)$  be a leverage pair for $\{ p_k \}$.
The proof is based on boundedness of each grading component of (\ref{eqn:holo_on_X_k}), as in proposition \ref{prop:all_about_holonomy}.
As before, let $\widehat{X}_k$ and $\widehat{Y}_k$ equal $\omega_{\hat{x}_k}(X)$ and $\omega_{\hat{x}_k}(Y)$, respectively, and denote with superscripts their components on the grading $\lieg = \lieg_{-1} \oplus \lieg_0 \oplus \lieg_1$.  By proposition \ref{prop:all_about_holonomy}, both $\omega_{\hat{x}}(X)^{(-1)}, \omega_{\hat{x}}(Y)^{(-1)} \in E_1^\perp$.  Then we can write
$$ \widehat{Y}_k^{(-1)} = \delta_k E_n + F_k \qquad \widehat{X}_k^{(-1)} = \epsilon_k E_n + G_k \qquad \mbox{where } F_k, G_k \in E_1^\perp, \ \delta_k, \epsilon_k \rightarrow 0$$
Possibly after passing to a subsequence, one of the sequences $\delta_k / \epsilon_k$ or $\epsilon_k / \delta_k$ has a limit in $\BR$ as $k \rightarrow \infty$.  In the first case, set $\widehat{Z}_k = \omega_{\hat{x}_k}(Y - \delta_k X / \epsilon_k)$; otherwise, set $\widehat{Z}_k = \omega_{\hat{x}_k}(X - \epsilon_k Y / \delta_k ).$  Then, after passing to a subsequence, the limit $\Ad p_k (\widehat{Z}_k)$ exists, and $\widehat{Z}_k^{(-1)} \in E_1^\perp$; moreover, by linear independence of $X(x)$ and $Y(x)$, the limit of $\widehat{Z}_k^{(-1)}$ is a nonzero element of $E_1^\perp$.

Now we consider the second term of equation (\ref{eqn:holo_on_X_k}) for $\widehat{Z}_k$.  As before, the following sequence must be bounded:
$$\left[ \xi_k,\widehat{Z}_k^{(-1)} \right]_0 = \left[ (\xi_k)_{\alpha - \beta},(\widehat{Z}_k^{(-1)})_{\beta - \alpha} \right] + \left[ (\xi_k)_{\alpha},(\widehat{Z}_k^{(-1)})_{- \alpha} \right] $$
and one of the limits of $(\widehat{Z}_k)^{(-1)}_{- \alpha}$ or $(\widehat{Z}_k)^{(-1)}_{\beta - \alpha}$ is nonzero.  In the first case, $(\xi_k)_{\alpha}$ must be bounded by the nondegeneracy lemma \ref{lem:bracket_nondegeneracy}.  Then boundedness of
$$ \left[ \xi_k , \widehat{Z}_k^{(-1)} \right]_{- \beta} = \left[ (\xi_k)_{\alpha - \beta},(\widehat{Z}_k)^{(-1)}_{- \alpha} \right]$$
implies $(\xi_k)_{\alpha - \beta}$ is bounded because, for nonzero $\xi \in \lieg_{\alpha - \beta}$, the map $\ad \xi : \lieg_{- \alpha} \rightarrow \lieg_{- \beta}$ is an isomorphism.

Otherwise, $\lim \widehat{Z}_k^{(-1)} \in \BR E_1$ and $(\xi_k)_{\alpha - \beta}$ is bounded (again thanks to lemma \ref{lem:bracket_nondegeneracy}).  By linear independence of $X(x)$ and $Y(x)$, one of $\omega_{\hat{x}}(X)_{- \alpha}^{(-1)}$ or $\omega_{\hat{x}}(Y)_{- \alpha}^{(-1)}$ is nonzero.  Assume without loss of generality this holds for $X$.  Now
$$\left[ \xi_k,\widehat{X}_k^{(-1)} \right]_0 = \left[ (\xi_k)_{\alpha - \beta},(\widehat{X}_k^{(-1)})_{\beta - \alpha} \right] + \left[ (\xi_k)_{\alpha},(\widehat{X}_k^{(-1)})_{- \alpha}\right] + \left[ (\xi_k)_{\alpha + \beta},(\widehat{X}_k^{(-1)})_{- \alpha - \beta} \right]$$
The first term on the right-hand side is bounded.  The second and third terms are linearly independent, so each must be bounded.  We conclude with lemma \ref{lem:bracket_nondegeneracy} $(\xi_k)_\alpha$ is bounded.  Under the isomorphism $\lieg_1 = \liep^+ \cong \BR^{1,n-1*}$, the root space $\lieg_{\alpha + \beta} $ corresponds to $\BR^tE_1 \mathbb{I}$. 

Let $\tau_k'$ be the components of $\tau_k$ in the annihilator of $E_n$, a complementary hyperplane to $\BR^tE_1 \mathbb{I}$ in $\liep^+$.  Right-multiply $\hat{x}_k$ with the bounded sequence $\tau_k'^{-1}$ to obtain a new, vertically equivalent ACL holonomy sequence at $x$ with $\{ \xi_k \} \subset \BR^tE_1\mathbb{I}$.  

Next, we again use the nonvanishing of one of $(\widehat{X}_k)^{(-1)}_{- \alpha}$ or $(\widehat{Y}_k)^{(-1)}_{- \alpha}$, which we take without loss of generality to be the former.  The following sequence is bounded:
$$ \left( \Ad (p_k). \widehat{X}_k \right)_\beta = \Ad(d_k).\left( (\widehat{X}_k)^{(0)}_\beta + \left[ (\xi_k)_{\alpha + \beta},(\widehat{X}_k)^{(-1)}_{- \alpha}  \right] \right) $$
The first term on the right-hand side tends to 0, while the second equals $e^{\beta(D_k) } \left[ \xi_k, (\widehat{X}_k)^{(-1)}_{- \alpha} \right] $.  For nonzero $\xi \in \lieg_{\alpha + \beta}$, the map $\ad \xi : \lieg_{- \alpha} \rightarrow \lieg_\beta$ is an isomorphism, so $\{ e^{\beta(D_k) } \xi_k \}$ is bounded.
\end{proof}

The following important proposition is a corollary of the above lemma:

\begin{proposition}
\label{prop:e1perp_stable}
Let $x \in \Omega_f$, and let $\{ p_k \}$ be an ACL holonomy sequence at $x$ of mixed or bounded type.  Then $\{ p_k \}$ is vertically equivalent to an ACL holonomy sequence for which the $P$-module $\lieg$ with the adjoint representation satisfies
 $$ E_1^\perp \subset \lieg^{AS}(p_k) $$
 (under the $G_0$-equivariant identification of $\BR^{1,n-1} \cong \lieg_{-1} \subset \lieg$)
\end{proposition}

\begin{proof}
Replace $\{ p_k \}$ with a vertically equivalent sequence satisfying the conclusions of lemma \ref{lemma:holonomy_second_order}.
Let $v \in E_1^\perp$.  Because $\ad \lieg_{\alpha + \beta}(E_1^\perp) \subset \lieg_\beta$ and $\ad \lieg_{\alpha + \beta}(\lieg_\beta) = 0$, we simply verify that
$$ \Ad(p_k).v = \Ad(d_k). ( v + [\xi_k,v]) = \Ad(d_k).v + e^{\beta(D_k)} [\xi_k,v]$$
is bounded by our assumptions on $\{ p_k \}$.
\end{proof}

\subsection{Foliation in open, dense subset}
\label{sec:foliation_construction}

Assuming $(M,g)$ is simply connected in proposition \ref{prop:balanced_proof}, the balanced case cannot occur in
proposition \ref{prop:all_about_holonomy}, so for $x \in \Omega_f$ and any ACL holonomy sequence $\{ p_k \}$ at $x$, the approximately stable subset of $T_x M$ is a degenerate hyperplane containing $\lieh(x)$.  
Denote by $\mathcal{D}_x(p_k )$ the approximately stable subsets of $T_xM$---each a lightlike hyperplane---arising in this way.  In this section we show that these lightlike hyperplanes form an integrable distribution on an open, dense subset of $M$, following the broad outline of section \ref{sec:balanced_proof}.

\emph{Step 1: curvature values.}  First assume $n \geq 4$, and let $\mathbb{V} = \Sym^2(\lieg_0')$, where $\lieg_0' \cong \so(1,n-1)$, as in sections \ref{sec:weyl_representation} and \ref{sec:weyl_isotropic_lines}.  

  For any holonomy sequence of mixed or bounded type, $\{ p_k = d_k \tau_k \}$ with $D_k = \ln d_k$, the behavior $\beta(D_k) \rightarrow - \infty$ makes $\mathbb{V}^{AS}(p_k) \subset \mathbb{V}_0 \oplus \mathbb{V}_{+1} \oplus \mathbb{V}_{+2}$.  If $\{p_k \}$ is moreover of mixed type, then $\alpha(D_k) \rightarrow \infty$ makes $\mathbb{V}^{AS}(p_k) \subset \mathbb{V}_{+1} \oplus \mathbb{V}_{+2}$.  Then $\bar{\kappa}^\sharp(\hat{x}) \in \mathbb{V}^+_{Ric}$ by proposition \ref{prop:as_for_inv_sections}.  Recall that $\bar{\kappa}^\sharp$ is $P^+$-invariant, so the condition $\bar{\kappa}^\sharp(\hat{x}) \in G_0.\mathbb{V}^+_{Ric}$ is independent of the choice of $\hat{x} \in \pi^{-1}(x)$.   Because $S_{mix}$ is dense and $G_0.\mathbb{V}^+_{Ric}$ is Zariski closed by proposition \ref{prop:zariski_closed_subspaces}, we conclude that the image of $\bar{\kappa}^\sharp$ is contained in $G_0.\mathbb{V}^+_{Ric}$ on all of $M$.
  
Let the subspace ${\mathbb V}^+_B \subset \mathbb{V}^+_{Ric}$ be as in section \ref{sec:weyl_isotropic_lines}.
The set $(\bar{\kappa}^\sharp)^{-1}(G_0. {\mathbb V}^+_B)$ is analytic by proposition \ref{prop:zariski_closed_subspaces} and $P$-invariant, so it projects to an analytic subset of $M$.  If this set equals $M$, then let ${\mathbb V}^+_\Omega =  {\mathbb V}^+_B \backslash \{ 0 \}$; otherwise, let ${\mathbb V}^+_\Omega = {\mathbb V}^+_{Ric} \backslash {\mathbb V}^+_B$.  Then $\Omega_{W} = \pi \left( (\bar{\kappa}^\sharp)^{-1}(G_0.{\mathbb V}^+_\Omega) \right)$ is an open, dense subset of $M$ with analytic complement.

When $n=3$, we take $\mathbb{U}^\sim$ in place of $\mathbb{V}$;  $\mathbb{U}^{+ \sim}$ in place of $\mathbb{V}^+_{Ric}$; and $\mathbb{U}_{+2}^\sim$ in place of $\mathbb{V}_B^+$ in the argument above.  Proposition \ref{prop:zariski_closed_subspaces} and corollary \ref{cor:weyl_determines_lines} hold for these modules, as well.  Define $\Omega_{W}$ in the analogous way.
%
 
We can extend our knowledge of the Weyl curvature to the full Cartan curvature $\kappa$.  Recall that $\kappa$ has values in $\wedge^2 (\lieg/\liep)^* \otimes \liep$.  Via the isomorphism $\lieg/\liep \cong \lieg_{-1}$, the inner product $\mathbb{I}$ corresponds to a $P$-invariant inner product on $\lieg/\liep$.  As in section \ref{sec:weyl_representation}, using $\mathbb{I}$ to raise an index gives a $P$-equivariant isomorphism $\wedge^2 (\lieg/\liep)^* \otimes \liep \rightarrow (\liep/\liep^+)'\otimes \liep$, where the latter $P$-representation is via the adjoint with the modification that the dilation by $d$ gives an additional factor of $d^{-2}$; here $\liep/\liep^+ \subset \mbox{End}(\lieg/\liep)$ is isomorphic to $\lieg_0$, and $(\liep/\liep^+)'$ is the commutator subalgebra.  Denote by $\kappa^\sharp$ the composition of $\kappa$ with this isomorphism.  As a $G_0$-module, the target is isomorphic to $ \mbox{Sym}^2 \lieg_0' \oplus \lieg_0' \otimes \lieg_{-1}^*$, using that $\liep^+ \cong \lieg_{-1}^*$.  

\begin{lemma}
\label{lemma:cartan_curvature_positive}
Let $\{p_k\}$ be an ACL holonomy sequence of mixed type with respect to $\hat{x}_k \rightarrow \hat{x}$, for $x \in \Omega_f$, satisfying the conclusions of lemma \ref{lemma:holonomy_second_order}.  Then $\kappa^\sharp(\hat{x}) \in \mathbb{V}^+_{Ric} + \mathbb{U}^+$, where $\mathbb{V}^+_{Ric}$ and $\mathbb{U}^+ = \mathbb{U}_{+1} \oplus \mathbb{U}_{+2}$ are as in section \ref{sec:weyl_representation}.
\end{lemma}

\begin{proof}
Write $\{ p_k \} = \{ d_k \tau_k \}$ with $\xi_k = \ln \tau_k \in \BR ^tE_1 \mathbb{I}$ and $\{ e^{\beta(D_k)} \xi_k \}$ bounded.
  Write $\kappa^\sharp(\hat{x}_k) = \varpi_k + \nu_k$ with $\varpi_k \in \mathbb{V}$ and $\nu_k \in \mathbb{U}$.   The action of $\{ p_k \}$ on $\kappa^\sharp(\hat{x}_k)$ is
  \begin{equation}
\label{eqn:pk_on_kappa}
 p_k.\kappa^\sharp(\hat{x}_k) = e^{2 \alpha(D_k)} d_k. \left( \varpi_k + (\nu_k + \xi_k.\varpi_k) \right)
 \end{equation}
By proposition \ref{prop:as_for_inv_sections}, $\kappa^\sharp(\hat{x})$ is in the approximately stable set for this action.  From just above, $\bar{\kappa}^\sharp(\widehat{M}) \subset \mathbb{V}^+_{Ric}$.  In particular $\lim \varpi_k = \bar{\kappa}^\sharp(\hat{x}) \in \mathbb{V}^+_{Ric}$.  
Boundedness of each component of $e^{2 \alpha(D_k)} d_k. \varpi_k$ implies that each of the sequences $\{ e^{ - 2 \beta(D_k)} \varpi_k^{(-2)} \}$ and $\{ e^{ - \beta(D_k)} \varpi_k^{(-1)} \}$ tends to $0$, because $\alpha(D_k) \rightarrow \infty$.  Then 
$$ d_k \xi_k. \varpi^{(-1)}_k =  e^{\alpha(D_k)} \xi_k .\varpi_k^{(-1)} = \left( e^{(\alpha+\beta)(D_k)} \xi_k \right) . \left( e^{- \beta(D_k)} \varpi_k^{(-1)} \right) \rightarrow 0$$
In order that the $\mathbb{U}_0$-component of (\ref{eqn:pk_on_kappa}) remain bounded as $k \rightarrow \infty$, 
$$ d_k.(\nu_k + \xi_k.\varpi_k)^{(0)} = e^{\alpha(D_k)} \left( \nu_k^{(0)} + \xi_k. \varpi_k^{(-1)} \right) \rightarrow 0$$
which implies $\nu_k^{(0)} \rightarrow 0$.  Similar arguments with $d_k \xi_k.\varpi_k^{(-2)}$ and the $\mathbb{U}_{-1}$-component of (\ref{eqn:pk_on_kappa}) give $\nu_k^{(-1)} \rightarrow 0$.  Finally, consideration of the $\mathbb{U}_{-2}$-component of (\ref{eqn:pk_on_kappa}) implies $\nu_k^{(-2)} \rightarrow 0$.  Then $\kappa^\sharp(\hat{x})$ belongs to the claimed subspace $\mathbb{V}^+_{Ric} + \mathbb{U}^+$.   
\end{proof}

\medskip

\emph{Step 2: definition of distribution and propagation of holonomy.}

\begin{proposition}
\label{prop:distribution_in_open_dense}
For all ACL holonomy sequences $\{ p_k \}$
at $x \in \Omega_{\mathcal{F}} = \Omega_W \cap \Omega_f$, the approximately stable subspaces $\mathcal{D}_x( p_k ) \subset T_x M$ are equal.  These subspaces form an $H$-invariant analytic distribution of lightlike hyperplanes in $\Omega_{\mathcal{F}}$.  
\end{proposition}

\begin{proof}
We will show that the spaces $\mathcal{D}_x(p_k )$ are determined by the Weyl curvature.    
  For an ACL holonomy sequence $\{ p_k \}$, mixed or bounded with respect to $\hat{x}_k \rightarrow \hat{x} \in \pi^{-1}(\Omega_{\mathcal{F}})$ as in the hypotheses, proposition \ref{prop:all_about_holonomy} says $(\lieg / \liep)^{AS}(p_k) = E_1^\perp$.  The algebraic map given by corollary \ref{cor:weyl_determines_lines} has value $\BR E_1$ on $\mathbb{V}^+_{\Omega}$.  Composed with $\bar{\kappa}$, it gives a $P$-equivariant analytic map $\hat{\eta} : \pi^{-1}(\Omega_W) \rightarrow {\bf P}(\mathcal{N} \backslash \{ 0 \} )$ with the property that
the approximately stable space $(\lieg / \liep)^{AS}(p_k) = \hat{\eta}(\hat{x})^\perp$.   Now $\hat{\eta}$ descends to an analytic section $\eta : \Omega_{W} \rightarrow {\bf P}(\mathcal{N}(TM) \backslash \{ 0 \})$, with the property that $\mathcal{D}_x(p_k) = \eta(x)^\perp$ for all $x \in \Omega_{\mathcal{F}}$, and $\{ p_k \}$ any ACL holonomy sequence at $x$ as in the hypotheses.  We conclude that the approximately stable subspaces $\mathcal{D}_x(p_k)$, henceforth denoted simply $\mathcal{D}_x$, form an analytic distribution of lightlike hyperplanes in $\Omega_{\mathcal{F}}$.  Because $\{ \mathcal{D}_x \}$ are determined by the invariant tensor $W$, they form an $H$-invariant distribution in $\Omega_{\mathcal{F}}$.
\end{proof}

Propagation of holonomy for mixed and bounded types, in contrast to the balanced case, works only in approximately stable directions, as follows:
\begin{proposition}
\label{prop:propagation_mixed_bounded}
Let $\{ p_k \}$ be a holonomy sequence satisfying the hypotheses of proposition \ref{prop:e1perp_stable}.
Then given $v \in E_1^\perp$, and $t$ such that $\exp_{\hat{x}}(tv)$ is defined and projects to a point of $\Omega_{\mathcal{F}}$, the sequence $\{ p_k \}$ is also a holonomy sequence with respect to some $\hat{x}_k' \rightarrow \exp_{\hat{x}}(tv)$.
\end{proposition}
This is an immediate consequence of propositions \ref{prop:e1perp_stable} and \ref{prop:propagation_holonomy}.

\medskip

\emph{Step 3: reduction of first-order frame bundle.}

Let $Q_0 < G_0$ be the stabilizer of the line $\BR E_1$ (as in Step 3 of section \ref{sec:balanced_proof}), and again define
an $H$-invariant, analytic reduction $\mathcal{R}'$ of $\left. \widehat{M}/P^+ \right|_{\Omega_{\mathcal{F}}}$ to $Q_0$ by
$$ \mathcal{R}' = \hat{\pi}(\hat{\eta}^{-1} (\BR E_1))$$
By the definition of $\mathcal{D}$, this reduction also equals
$$\{ \hat{\pi}(\hat{x}) \ : \ \pi_* \omega^{-1}_{\hat{x}}(E_1^\perp) = \mathcal{D}_x \}$$
that is, $\mathcal{R}'$ comprises the (conformal normalized) 1-frames identifying $\mathcal{D}$ with $E_1^\perp$.

Let $\{p_k\}$ be a holonomy sequence with respect to $\hat{x}_k \rightarrow \hat{x} \in \pi^{-1}( \Omega_{\mathcal{F}})$ satisfying the hypotheses and conclusions of proposition \ref{prop:e1perp_stable}.  Then $\hat{\eta}(\hat{x}) = \BR E_1$ and $x' = \hat{\pi}(\hat{x})$ belongs to $\mathcal{R}'$.
Given $v \in E_1^\perp$, proposition \ref{prop:propagation_mixed_bounded} says $\{ p_k \}$ is a holonomy sequence at $\hat{z}_v(t) = \exp_{\hat{x}}(tv)$ for all $t$ in some interval $I$ containing a neighborhood of $0$.  That means $z'_v(t) = \hat{\pi}(\hat{z}_v(t)) \in \mathcal{R}'$,  so $\hat{\pi}_* \omega_{\hat{x}}^{-1}(v) \in T_{x'}\mathcal{R}'$.  Thus $\hat{\pi}_* \omega_{\hat{x}}^{-1}(E_1^\perp) \subset T_{x'}\mathcal{R}'$.

\medskip

\emph{Step 4: reduction of $\left. \widehat{M} \right|_{\Omega_{\mathcal{F}}}$}

Let $\mathfrak{q}_1$ be the annihilator of $E_1$ in $\liep^+ \cong \BR^{1,n-1*}$, with corresponding connected subgroup $Q_1 < P^+$, and let $Q = Q_0 \ltimes Q_1$.  For $\tau \in P^+$, the image $\Ad \tau(E_1^\perp + \lieq_0) \equiv E_1^\perp + \lieq_0 \ \mbox{mod } \liep^+$ if and only if $\tau \in Q_1$.  Define
$$ \mathcal{R} = \{ \hat{x} \ : \ \hat{\pi}(\hat{x}) \in \mathcal{R}', \ \mbox{and } \hat{\pi}_* \omega_{\hat{x}}^{-1} (E_1^\perp) \subset T_{\hat{\pi}(\hat{x})} \mathcal{R}' \}$$
As $\widehat{M}$ is a subbundle of the second-order frames of $M$, it can also be viewed as a bundle of first-order frames on $\widehat{M}/P^+$, the bundle of conformal normalized one-frames. 
With this point of view, we can define a $Q_0 \ltimes P^+$-equivariant map $\hat{\pi}^{-1}(\mathcal{R}') \rightarrow \Hom(E_1^\perp, \lieg_0/\lieq_0)$ such that $\mathcal{R}$ is precisely the inverse image of $0$.  The affine $Q_0 \ltimes P^+$-action on $\Hom(E_1^\perp, \lieg_0/\lieq_0)$ factors through $\mbox{Aff}(\BR)$.  The orbit of $0$ is one-dimensional, with stabilizer $Q$.

We will show that every $\pi$-fiber of $\hat{\pi}^{-1}(\mathcal{R}')$ contains a point mapping to $0 \in \Hom(E_1^\perp, \lieg_0/\lieq_0)$ and thus the image of $\hat{\pi}^{-1}(\mathcal{R}')$ is contained in the above one-dimensional orbit.
Let, as before, $\{p_k\}$ be a holonomy sequence with respect to $\hat{x}_k \rightarrow \hat{x} \in \pi^{-1}( \Omega_{\mathcal{F}})$ satisfying the hypotheses and conclusions of proposition \ref{prop:e1perp_stable}; recall that every $x \in \Omega_{\mathcal{F}}$ admits such a holonomy sequence.  From Step 3, $x' = \hat{\pi}(\hat{x}) \in \mathcal{R}'$ and $\hat{\pi}_* \omega_{\hat{x}}^{-1}(E_1^\perp) \subset T_{x'}\mathcal{R}'$.  Thus $\hat{x} \in \mathcal{R}$.  
It follows that 
$\mathcal{R}$ is a smooth $Q$-reduction of $ \widehat{M}$ over ${\Omega_{\mathcal{F}}}$.  
Note that $\mathcal{R}$ is $H$-invariant and analytic because $\mathcal{R}'$ and $\omega$ are.

The geometric interpretation of $\mathcal{R}$ is as the conformal normalized 2-frames at points $x \in \Omega_{\mathcal{F}}$ in which $\mathcal{D}$ is totally geodesic (infinitesimally at $x$).

\medskip

\emph{Step 5: integration.}

\begin{proposition}
\label{prop:more_about_holonomy}  Let $\hat{x} \in \mathcal{R}$.  
Given $v \in E_1^\perp$, and $t$ such that $\exp_{\hat{x}}(tv)$ is defined and projects to $\Omega_{\mathcal{F}}$, the point $\exp_{\hat{x}}(tv) \in \mathcal{R};$
moreover, $\omega_{\hat{x}}^{-1}(E_1^\perp) \subset T_{\hat{x}}\mathcal{R}$ and $E_1^\perp + \lieq \subset \omega(T_{\hat{x}} \mathcal{R})$.
\end{proposition}

\begin{proof}
Every $x \in \Omega_{\mathcal{F}}$ admits 
a holonomy sequence $\{ p_k \}$ satisfying the hypotheses and conclusions of proposition \ref{prop:e1perp_stable}.
If $\{ p_k \}$ is defined with respect to $\hat{x}_k \rightarrow \hat{x}$, then, as we saw in the previous step, $\hat{x} \in \mathcal{R}$.  The conclusion  $\omega_{\hat{x}}^{-1}(E_1^\perp) \subset T_{\hat{x}} \mathcal{R}$ is invariant by the right-$Q$-action.  All three conclusions for all $\hat{x} \in \mathcal{R}$ follow if we can prove them for all $\hat{x}$ which are the reference point for such a holonomy sequence.  

Proposition \ref{prop:propagation_mixed_bounded} gives, for $v \in E_1^\perp$, that $\{ p_k \}$ is a holonomy sequence with respect to $\hat{z}_v(t) = \exp_{\hat{x}}(tv)$ for all $ - \epsilon < t < \epsilon$; it follows, as above, that $\exp_{\hat{x}}(tv) \in \mathcal{R}$ for all such $t$.  Then $\omega^{-1}_{\hat{x}}(v) \in T_{\hat{x}} \mathcal{R}$, and since $v \in E_1^\perp$ was arbitrary, $\omega_{\hat{x}}^{-1}(E_1^\perp) \subset T_{\hat{x}} \mathcal{R}$.  Since $\mathcal{R}$ is a $Q$-reduction, $E_1^\perp + \lieq \subset \omega_{\hat{x}}(T_{\hat{x}} \mathcal{R})$. 
\end{proof}

By the above proposition, the distribution defined by $\widehat{\mathcal{D}} = \omega^{-1}(E_1^\perp + \mathfrak{q})$ restricted to $\mathcal{R}$ is tangent to $\mathcal{R}$. 
Finally, we can prove:

\begin{proposition}
\label{prop:distribution_integrability}
The distribution $\widehat{\mathcal{D}}$ restricted to $\mathcal{R}$ is integrable; in fact, the distribution $\widehat{\mathcal{D}}^\ell = \omega^{-1}(E_1^\perp + \lieu_+ + \BR^tE_1{\mathbb I})$ restricted to $\mathcal{R}$ is also integrable.  The integral leaves of either distribution project to integral leaves of $\mathcal{D}$.
\end{proposition}

Here $\lieu_+$ is the nilpotent subalgebra of $\lieg_0$ annihilating $E_1$ (as in section \ref{sec:balanced_proof}).

\begin{proof}
  Every $\hat{x} \in \mathcal{R}$ admits a holonomy sequence $\{ p_k \}$ 
  satisfying the hypotheses and conclusions of lemma \ref{lemma:holonomy_second_order}, as was seen in step 4.  Assuming that $\{ p_k \}$ is moreover of mixed type, lemma \ref{lemma:cartan_curvature_positive} implies $\kappa^\sharp(\hat{x}) \in \mathbb{V}^+_{Ric} + \mathbb{U}^+$.  (The latter subspace is $Q$-invariant.)
Corollary \ref{cor:dense_mixed} says points in $\Omega_{\mathcal{F}}$ with mixed holonomy sequences are dense.  Lemmas \ref{lemma:holonomy_second_order} and \ref{lemma:cartan_curvature_positive} apply to these holonomy sequences.  It follows that $\kappa^\sharp(\mathcal{R}) \subset \mathbb{V}^+_{Ric} + \mathbb{U}^+$.

Knowledge of the full Cartan curvature on $\mathcal{R}$ gives information on the brackets of vector fields on $\mathcal{R}$.  Given $u,v \in E_1^\perp + \lieq \subset \lieg$, denote $X = \omega^{-1}(u)$ and $Y = \omega^{-1}(v)$.  If one of $u,v \in \lieq$, then $[X, Y] = \omega^{-1}([u,v])$ so the result is in $\widehat{\mathcal{D}}$ because $E_1^\perp + \lieq$ is a subalgebra of $\lieg$.  Assume now that $u,v \in E_1^\perp$.  Then the formula (\ref{eqn:curvature_formula}) gives in this case
$$ \kappa_{\hat{x}}(u,v) =  - \omega_{\hat{x}}[X,Y] + [u,v] = - \omega_{\hat{x}}[X,Y]$$
because $\lieg_{-1}$ is abelian.  Elements of $\mathbb{V}^+_{Ric} + \mathbb{U}^+$ evaluated on $(E_1^\perp)^* \otimes E_1^\perp$ have values in $\lieu_+ + \BR^tE_1 \mathbb{I}$, which is contained in $\lieq$.  Thus $\widehat{D}$ is integrable.  Integrability of $\widehat{\mathcal{D}}^\ell$ follows from the observation that $E_1^\perp + \lieu_+ + \BR^tE_1{\mathbb I}$ is also a subalgebra of $\lieg$.

For any $\hat{x} \in \mathcal{R}$, the subspace $\omega_{\hat{x}}^{-1}(E_1^\perp)$ projects to $\mathcal{D}_x$, so an integral leaf of $\widehat{\mathcal{D}}$ or of $\widehat{\mathcal{D}}^\ell$ projects to an integral leaf of $\mathcal{D}$.
\end{proof}

Denote by $\mathcal{F}$ the resulting foliation by lightlike hypersurfaces in $\Omega_{\mathcal{F}}$.  
In the next section, we will extend leaves $\mathcal{L}$ of $\mathcal{F}$ over closed, isotropic orbits in the closure $\overline{\mathcal{L}}$ by extending $\mathcal{R}$ over such orbits.  The following proposition implies that whenever $\mathcal{R}$ can be extended, then leaves can be constructed through the boundary points.

\begin{proposition}
\label{prop:R_extension_leaf}
Let $\hat{y} \in \overline{\mathcal{R}}$.  There is a neighborhood $\mathcal{U}$ of $0$ in $E_1^\perp$ such that $\exp_{\hat{y}} (\mathcal{U})$ is contained in an integral leaf of $ \omega^{-1}(E_1^\perp + \lieu_+ + \BR^tE_1 \mathbb{I})$; the saturation by $Q$ is contained in an integral leaf of $ \omega^{-1}(E_1^\perp + \lieq)$.
\end{proposition}

\begin{proof}
Let $\hat{y}_k \rightarrow \hat{y}$ with $\hat{y}_k \in \mathcal{R}$.  Let $\mathcal{U} \subset E_1^\perp$ be a neighborhood of $0$ in the common domain of $\exp_{\hat{y}}$ and of $\exp_{\hat{y}_k}$ for all $k$.  Fix $k$ and consider the following analytic condition on $v \in \mathcal{U}$:
$$ \omega_{\exp_{\hat{y}_k} (v)} \left( D_v \exp_{\hat{y}_k} (E_1^\perp) \right) \subset E_1^\perp + \lieu^+ + \BR^tE_1 \mathbb{I}$$
This condition is satisfied by $v$ in a neighborhood of $0$ in $\mathcal{U}$, by proposition \ref{prop:more_about_holonomy} because $\pi(\hat{y}_k)$ is in the open set $\Omega_{\mathcal{F}}$, together with the integrability of $\widehat{D}^\ell$ in $\mathcal{R}$ given by proposition \ref{prop:distribution_integrability}.  Thus the above containment holds for all $v \in \mathcal{U}$.  

With $k$ and $v \in \mathcal{U}$ fixed, consider the path $\exp_{\hat{y}_k}(tv)$ for $0 \leq t \leq 1$.  The set $\Omega_{\mathcal{F}} = \Omega_W \cap \Omega_f$ is open and dense with analytic complement, because each of $\Omega_W$ and $\Omega_f$ are (see Step 1 and proposition \ref{prop:freeness}).  The projection $\pi \circ \exp_{\hat{y}_k} (tv)$ is in $\Omega_{\mathcal{F}}$ for $t < \epsilon$ for some $\epsilon > 0$, and therefore, it is in $\Omega_{\mathcal{F}}$ for all but finitely-many $t$ in $[0,1]$.  Then by proposition \ref{prop:more_about_holonomy}, $\exp_{\hat{y}_k}(tv) \in \mathcal{R}$ for all but finitely-many $t$.  We conclude that $\exp_{\hat{y}_k}(v)$ is in the closure $\overline{\mathcal{R}}$ for all $v \in \mathcal{U}$ and all $k$. 

Now fix $v \in \mathcal{U}$ and take the limit of the subspaces $\omega_{\exp_{\hat{y}_k} (v)} \left( D_v \exp_{\hat{y}_k} (E_1^\perp) \right)$ as $k \rightarrow \infty$.  By continuity of the expression in $\hat{y}$, we also have
$$ \omega_{\exp_{\hat{y}} (v)} \left( D_v \exp_{\hat{y}} (E_1^\perp) \right) \subset E_1^\perp + \lieu_+ + \BR^tE_1 \mathbb{I} \qquad \forall \ v \in \mathcal{U}$$
Thus the saturation of $\exp_{\hat{y}}(\mathcal{U})$ by the group $e^{\lieu_+} \cdot e^{\BR^tE_1 \mathbb{I}}$ is an integral submanifold for $\hat{\mathcal{D}}^\ell$, and similarly for $Q$ and $\hat{\mathcal{D}}$.
From the previous paragraph, both integral submanifolds are contained in $\overline{\mathcal{R}}$.
\end{proof}

\subsection{Extension of leaves to isotropic limit orbits}
\label{sec:extension_closed_isotropic}

In this section we verify that $\mathcal{R}$---and therefore $\mathcal{F}$---extends over closed, isotropic orbits; in particular, it extends over orbits of accumulation points for ACL holonomy sequences of mixed type.  Recall from the proof of proposition \ref{prop:all_about_holonomy} that if the isotropy at a point with closed, isotropic $H$-orbit contains a subgroup of hyperbolic elements, then the isotropy is of balanced type.  By proposition \ref{prop:balanced_proof}, we can assume that this isotropy contains no hyperblic elements.  Assuming $H$ noncompact, the isotropy along every closed, isotropic orbit contains unipotent elements.

\begin{proposition}
\label{prop:extension_of_R}
Let $y \in M$ have closed, isotropic $H$-orbit, and suppose that, with respect to $\hat{y} \in \pi^{-1}(y)$, the isotropy $\hat{S}_y$ contains a unipotent subgroup in $e^{\lieu_+}$.
Then $\mathcal{R}^e = \mathcal{R} \cup H.\hat{y}. Q \subseteq \overline{\mathcal{R}}$. 
\end{proposition}

\begin{proof}
Let $Y \in \mathfrak{l}$ with $Y(y) \neq 0$.  Let $X \in \lieh$ with $X(y) = 0$ and unipotent isotropy with respect to $\hat{y}$.  We will denote $\omega_{\hat{y}}(Y) = \hat{Y}_*$ and $\omega_{\hat{y} } (X) = \hat{X}_*$.

Let $z_k \in \Omega_{\mathcal{F}}$ converge to $y$.
By propositions \ref{prop:all_about_holonomy} and \ref{prop:distribution_in_open_dense}, $\mathcal{D}_{z_k}$ contains $Y(z_k)$ for all $k$.  Then for any sequence $v_k \in \mathcal{D}_{z_k}^\perp$, we have $v_k \in Y(z_k)^\perp$.   Since $Y(z_k) \rightarrow Y(y)$, with the limit isotropic and nonzero, $\mathcal{D}_{z_k}^\perp \rightarrow \BR Y(y)$ and thus $\mathcal{D}_{z_k} \rightarrow Y(y)^\perp$.  
The unipotent isotropy $\omega_{\hat{y}}(X) \in \lieu_+$ has unique isotropic fixed vector $E_1$ (up to a nonzero scalar).  Since $[X,Y]=0$, the vector $Y(y)$ must correspond in the frame $y' = \hat{\pi}(\hat{y})$ to a nonzero multiple of $E_1$.
It follows that $\mathcal{R}^{'e} = \mathcal{R}' \cup H.y'. Q_0$ is a continuous extension over  $\Omega_{\mathcal{F}} \cup H.y$.   

Now let $\hat{z}_k \in \mathcal{R}$ with $\hat{\pi}(\hat{z}_k) = z_k' \rightarrow y'$ in $\mathcal{R}^{'e}$.   Given $W_k, Z_k \in \lieh$, denote $\hat{W}_k = \omega_{\hat{z}_k} (W_k)$ and $\hat{Z}_k = \omega_{\hat{z}_k} (Z_k)$.  There exists $\tau_k  \in P^+$ such that $\hat{z}_k \tau^{-1}_k \rightarrow \hat{y}$, and by adjusting $\{ \hat{z}_k \}$ if necessary, we may assume $\ln \tau_k = \xi_k  \in \BR^tE_n \mathbb{I}$, a subspace of $\liep^+$ complementary to $\lieq_1$ (it is the root space $\lieg_{\alpha - \beta}$).  The goal is to show that $\lim \xi_k = 0$.

Recall that $\hat{Y}_k^{(-1)}$ tends to a nonzero element of $\BR E_1$ and $\hat{X}_k^{(-1)}$ is linearly independent from it and tends to $0$.  
The following limit exists:
$$( \hat{Y}_*)^{(0)}_\liea = \lim \left( \hat{Y}^{(0)}_k + \left[ \xi_k,\hat{Y}^{(-1)}_k \right] \right)_\liea $$
Let $A_k = (\hat{Y}_k)^{(0)}_\liea$ and $B_k = \left[ \xi_k,(\hat{Y}_k)^{(-1)}_{\beta - \alpha} \right] \in \liea$.
 Since $\hat{Y}_*$ commutes with $\hat{X}_*$, a nonzero element of $\lieu_+$, its $\liea$-component is in the kernel of $\beta$ --- that is, $\lim ( \beta(A_k) + \beta(B_k) ) = 0$.
Recall that $\lieh$ is everywhere tangent to $\mathcal{R}$ and has projection to $TM$ in $\mathcal{D}$, it follows that $\lieh(\hat{x}) \subset \hat{\mathcal{D}}$ for all $\hat{x} \in \mathcal{R}$.  Since $\hat{z}_k \in \mathcal{R}$, the components $\hat{Y}_k^{(1)} \in \lieq_1$ for all $k$, so $(\hat{Y}_k)^{(1)}_{\alpha - \beta} = 0$ for all $k$.  Then 
\begin{eqnarray*}
(\hat{Y}_*)^{(1)}_{\alpha - \beta} & = & \lim \left(\Ad \tau_k(\hat{Y}_k) \right)^{(1)}_{\alpha-\beta} \\
& = &  \lim \left( \left[ \xi_k,A_k \right]  + \left[ \frac{\xi_k}{2}, B_k \right] \right) \\
& = & \lim \left(  (\beta - \alpha)(A_k + \frac{1}{2} B_k) \right) \cdot  \xi_k \\
& = & \lim \left( - \alpha(A_k) + \beta(A_k + B_k) \right) \cdot \xi_k 
\end{eqnarray*}
since $\left[ \xi_k,\liem \right] = 0$ and $(\alpha + \beta) (B_k) \equiv 0$.
On the other hand, since $\hat{Y}_*$ commutes with $\hat{X}_*$, and, for $\xi \neq 0$ in $\lieg_{\alpha - \beta}$, the map $\ad \xi : \lieg_\beta \rightarrow \lieg_\alpha$ is an isomorphism, $\hat{Y}_*^{(1)}$ is also in $\lieq_1$.  Thus if $\xi_k$ does not converge to $0$, then, up to extraction, $\lim \alpha(A_k) = 0$ and $\alpha \left( (\hat{Y}_*)^{(0)}_\liea \right) = \lim \alpha(B_k)$, which exists and is moreover nonzero (recall $\lim (\hat{Y}_k)^{(-1)}_{\beta - \alpha} \neq 0$).  

In order to reach the desired conclusion that $\lim \xi_k = 0$, it remains to derive a contradiction from the condition $\alpha \left( (\hat{Y}_*)^{(0)}_\liea \right) \neq 0$. Since $\omega_{\hat{y}}(Y)$ has the form required by lemma \ref{lemma:no_dilation_at_limit}, $Y$ is timelike for a metric in the conformal class at a point arbitrarily close to $y$.  Then it is timelike on an open set with $y$ in its closure.  This contradicts propositions \ref{prop:all_about_holonomy} and \ref{prop:distribution_in_open_dense}, which say that inside the open, dense set $\Omega_f$, the values of $Y$ are always tangent to the approximately stable distribution, which is degenerate.
\end{proof}

\begin{lemma}
\label{lemma:no_dilation_at_limit}
Let $Y \in \lieh$ and $\hat{z} \in \widehat{M}$.  Suppose that $\omega_{\hat{z}}(Y) = E_1 + A + B + \eta$, where $A \in (\liea \cap \ker \beta) \backslash \{ 0 \}$; $B \in \liem + \lieu_+$; and $\eta \in \lieq_1$.  Let $\gamma(t) = \pi \circ \exp_{\hat{z}}(tE_n)$.  Then there is $\epsilon \neq 0$ such that $Y(\gamma(\epsilon))$ is timelike.  
\end{lemma}

\begin{proof}
We will compute the first-order Taylor approximation in $t$ for $\langle Y, Y \rangle_{\gamma(t)}$ with respect to a metric along $\gamma$ in the conformal class.
Let $\widehat{E}_n = \omega^{-1}(E_n)$, and 
write $\hat{\gamma}(t) = \exp_{\hat{z}}(t E_n) = \varphi^t_{\widehat{E}_n} .\hat{z}$.
Because $L_Y \omega(\widehat{E}_n) = 0$, formula (\ref{eqn:curvature_formula}) gives
$$ \Omega_{\hat{\gamma}(t)}(Y,\widehat{E}_n) = - \left. \widehat{E}_n.\omega(Y) \right|_{\hat{\gamma}(t)} + [\omega_{\hat{\gamma}(t)}(Y),E_n] $$
Then
$$ \left. \DDt \right|_0 \omega_{\hat{\gamma}(t)}(Y) = [ E_1 + A + B + \eta,E_n] - \kappa_{\hat{z}}(E_1,E_n) \equiv a E_n + V_B \ \mbox{mod } \liep$$
where $a \neq 0$ and $V_B \in E_1^\perp \cap E_n^\perp$.  We are interested in the components on $\lieg_{-1}$:
$$\omega^{(-1)}_{\hat{\gamma}(t)}(Y) = E_1 + t(a E_n + V_B + R(t)) \qquad \lim_{t \rightarrow 0} R(t) = 0$$
Denote $\mathcal{Q} = Q_{1,n-1}$ and $\langle \ , \ \rangle$ the corresponding inner product on $\lieg_{-1}$.  This inner product together with the path $\hat{\gamma}$ determine a metric along $\gamma$ in the conformal class with 
\begin{eqnarray*}
\langle Y, Y \rangle_{\gamma(t)}  & = & \mathcal{Q}(\omega^{(-1)}_{\hat{\gamma}(t)}(Y)) \\
& = & 2at + 2t \langle E_1, R(t) \rangle + t^2 \mathcal{Q}(a E_n + V_B + R(t))
\end{eqnarray*}
This number is negative for sufficiently small $t$ of the opposite sign from $a$.
\end{proof}

\begin{remark}
  \label{rmk:lifts_to_R}
The proof of proposition \ref{prop:extension_of_R} shows that for any sequence $z_k \in \Omega_{\mathcal{F}}$ converging to $y$, there are lifts $\hat{z}_k \in \mathcal{R}$ converging to $\hat{y}$.  Moreover, given a continuous path $\gamma$ in $\Omega_{\mathcal{F}}$ converging to $y$, there is a continuous lift $\hat{\gamma}$ to $\mathcal{R}$ converging to $\hat{y}$, when $\hat{y}$ satisfies the hypotheses of proposition \ref{prop:extension_of_R}.
  \end{remark}

We now show that all orbits in $\Omega_{\mathcal{F}}$ are degenerate.  

\begin{proposition}  
\label{prop:isotropic_in_leaf_local}
 Let $\hat{y} \in \overline{\mathcal{R}}$, and suppose $X \in \mathfrak{stab}_\lieh(y)$ has isotropy in $\lieu_+$ with respect to $\hat{y}$.  
 The vector field $X$ is isotropic in a neighborhood of $y$ in $\mathcal{L} = \pi \circ \exp_{\hat{y}}(\mathcal{U})$, for $\mathcal{U}$ a suitable neighborhood of $0$ in $E_1^\perp$. 
\end{proposition}

\begin{proof}

Let $\mathcal{V} \subset \lieg_{-1} \cong \BR^{1,n-1}$ be a neighborhood of $0$ in the domain of $\exp_{\hat{y}}$, and denote $\widehat{N} = \exp_{\hat{y}}(\mathcal{V})$.  The saturation $\widehat{N}.G_0 \subset \widehat{M}$ determines a Weyl connection $\nabla$ on $N = \pi (\widehat{N})$ (see definition 5.1.2 of \cite{cap.slovak.book.vol1}).  It is given by $\omega^{(0)}$ restricted to $\widehat{N}.G_0$, which is a $\lieg_0$-valued principal connection (see proposition 5.1.2 of \cite{cap.slovak.book.vol1}).  Let $\mathcal{U} \subset \mathcal{V}$ be as in proposition \ref{prop:R_extension_leaf}; now $\widehat{\mathcal{L}} = \exp_{\hat{y}}(\mathcal{U}) \subset \widehat{N}$ is contained in an integral leaf of $\omega^{-1}(E_1^\perp + \lieu_+ + \BR^tE_1 \mathbb{I})$.  Here the restriction of $\omega^{(0)}$ has values in $\lieu_+$; therefore, the restriction of $\nabla$ to $\mathcal{L} = \pi(\widehat{\mathcal{L}})$ has holonomy in $U_+$.  It follows that $\mathcal{L}$ is totally geodesic and carries a parallel, isotropic vector field $E = \pi_* \left. \omega \right|_{\widehat{\mathcal{L}}}^{-1}(E_1)$.  

For $v \in  \mathcal{V}$, let $\hat{\gamma}_v(t) = \exp_{\hat{y}}(tv)$ and note that $\hat{\gamma}_v(t) \in \widehat{N}$ for $0 \leq t \leq 1$.  Since $\omega^{(0)}$ vanishes on $\hat{\gamma}'_v$, the projection $\gamma_v = \pi \circ \hat{\gamma}_v$ is a $\nabla$-geodesic.  Therefore, under the isomorphism $\pi_* \circ \omega^{-1}_{\hat{y}} : \lieg_{-1} \rightarrow T_y M$, the exponential map of $\nabla$, restricted to a neighborhood of $0$, corresponds to $\left. \pi \circ \exp_{\hat{y}} \right|_{\mathcal{V}}.$

Denote $\hat{X} \in \hat{\lies}_y \cap \lieu_+$ the isotropy of $X$ with respect to $\hat{y}$. Because $\hat{X}$ preserves the complementary subspace $\lieg_{-1}$ to $\liep$ in $\lieg$, the action of $\hat{X}$ on $\mathcal{V} \subset \lieg_{-1}$ is equivalent via $\pi \circ \exp_{\hat{y}}$ to the action of $X$ on $N$.  As the former action is linear, 
$$ X(\gamma_v(t)) = (\pi \circ \exp_{\hat{y}})_{*tv} \hat{X}(tv) = (\pi \circ \exp_{\hat{y}})_{*tv}  (t \hat{X}.v)$$
is a Jacobi field along $\gamma_v$ for any $v \in \mathcal{V}$ (see, for example, \cite[Cor 5.2.5]{do.carmo.riem.geom}).

Now let $v \in \mathcal{U}$ and denote $J(t) = X(\gamma_v(t))$.  Of course, $J(0) = 0$.  As $\hat{X}.v \in \BR E_1$, the first derivative $J'(0) = c_0 E(y)$ for some $c_0 \in \BR$.  Denote by $R$ the curvature tensor of $\nabla$.  In the linear frames determined by $\omega_{\hat{\gamma}_v(t)}$, it corresponds to
$$ \left( d \omega^{(0)} + \frac{1}{2}[ \omega^{(0)}, \omega^{(0)}] \right)_{\hat{\gamma}_v(t)} = \Omega^{(0)}_{\hat{\gamma}_v(t)} -  [ \omega^{(-1)}, \omega^{(1)}]_{\hat{\gamma}_v(t)}$$
Since $\hat{\gamma}_v(t) \in \widehat{\mathcal{L}}$ for all $t$, the vector fields $E$ and $\gamma_v'$ along $\gamma_v$ lift to $\omega^{-1}(E_1)$ and $\hat{\gamma}_v'$, respectively.  Restricted to $\widehat{\mathcal{L}}$, the values of $\omega^{(1)}$ are in $\BR^tE_1 \mathbb{I}$, which annihilates $E_1$ and $v$.  The term
$$ \Omega^{(0)}_{\hat{\gamma}_v(t)} (\omega^{-1}(E_1), \hat{\gamma}_v') = \bar{\kappa}_{\hat{\gamma}_v(t)}(E_1,v)$$
where $\bar{\kappa} = \kappa^{(0)}$ corresponds to the Weyl curvature of $\omega$.
From $\bar{\kappa}^\sharp_{\hat{\gamma}_v(t)} \in \mathbb{V}^+_{Ric}$, we see that $\bar{\kappa}_{\hat{\gamma}_v(t)}(E_1,v) = 0$ for all $t$.  It follows that $R_{\gamma_v(t)}(E,\gamma_v') = 0$ for all $t$.  
Therefore $J(t) = tc_0 E$ is the unique solution of the Jacobi equation along $\gamma_v$ with the given initial conditions.  We conclude that $X$ is isotropic along $\gamma_v$.  Since $v$ was an arbitrary element of $\mathcal{U}$, the proposition follows.  
\end{proof}

\begin{corollary}
\label{cor:isotropic_in_leaf}
Let $x \in \Omega_{\mathcal{F}}$ and let $y \in \overline{H.x}$ have isotropic orbit, and let $\omega_{\hat{y}}(X) \in \lieu_+$ for $X \in \lieh$ and $\hat{y} \in \pi^{-1}(y)$.  Then $X$ is isotropic in the leaf $\mathcal{L}_x$.
\end{corollary}

\begin{proof}
By theorem \ref{thm:gromov_stratification}, there are paths connecting $y$ to points of $H.x$ in a neighborhood of $y$, which are contained in $H.x$ for all positive time.  Such a path $\gamma$ is contained in the leaf $\mathcal{L}_x$ of the foliation $\mathcal{F}$ for positive time.  

Let $\mathcal{R}^e$ be the extension of $\mathcal{R}$ over $H.y$ given by proposition \ref{prop:extension_of_R}.
There is a continuous lift $\hat{\gamma}$ of $\gamma$ to $\mathcal{R}^e$ beginning at $\hat{y}$ (see remark \ref{rmk:lifts_to_R}).  The leaves of the extended foliation of $\mathcal{R}^e$ tangent to $\widehat{\mathcal{D}}$ are exponentials of $E_1^\perp + \lieq$, by proposition \ref{prop:R_extension_leaf}.  There is a neighborhood $\mathcal{V}$ of 0 in $E_1^\perp  + \lieq$ such that $\exp_{\hat{\gamma}(t)}$ is defined on $\mathcal{V}$ for all $t$; as in the proof of proposition \ref{prop:R_extension_leaf}, $\exp_{\hat{\gamma}(t)}(\mathcal{V}) \subset \overline{\mathcal{R}}$ for all $t$.  As $t \rightarrow 0$, the plaques $\widehat{\mathcal{L}}_t = \exp_{\hat{\gamma}(t)}(\mathcal{V})$ converge to the plaque $\exp_{\hat{y}}(\mathcal{V})$.  On the other hand, $\widehat{\mathcal{L}}_t$ are in the same leaf for all $t > 0$, so $\widehat{\mathcal{L}}_t \cap \widehat{\mathcal{L}}_{t'}$ is open in each for every $t, t' > 0$.  The sizes of these intersections are bounded below, in the sense that $\exp_{\hat{y}}^{-1}(\widehat{\mathcal{L}}_t \cap \widehat{\mathcal{L}}_{t'}) \subset \lieg$ projected to $\mathcal{V}$ contains a ball around $0$ of radius bounded below independently of $t,t'$, provided both are sufficiently small.
Thus the intersection of $\widehat{\mathcal{L}}_t$ with $\widehat{\mathcal{L}}_0$ is also open in each and contains $\hat{y}$ for sufficiently small $t$.

Let $\mathcal{L}$ be given by proposition \ref{prop:isotropic_in_leaf_local}.
The projection $\pi(\widehat{\mathcal{L}}_t)$ is relatively open in $\mathcal{L}_x$ for all $t$, and $\pi(\widehat{\mathcal{L}}_0)$ is an open neighborhood of $y$ in $\mathcal{L}$, 
so $\mathcal{L}_x \cap \mathcal{L}$ is nonempty and open in each.  
By proposition \ref{prop:isotropic_in_leaf_local}, $X$ is isotropic in this intersection. The property extends along paths $\alpha$ in $\mathcal{L}_x$ by covering them with foliated charts and using the fact that the plaque of $\mathcal{L}_x$ containing $\alpha$ is $C^\omega$ inside each such chart.  Then $X$ is isotropic in all $\mathcal{L}_x$.
\end{proof}

\begin{remark}
  \label{rmk:mutually.open}
  A corollary of the above proof is that when $y \in \overline{H.x}$ with $x \in \Omega_{\mathcal{F}}$ and $y$ having closed isotropic orbit, then the leaf $\mathcal{L}_x\in \mathcal{F}$ and the plaque $\pi \circ \exp_{\hat{y}}(\mathcal{V})$ have mutually open intersection, where $\hat{y}$ and $\mathcal{V}$ are as above.  For such $y$, we will refer to the latter set as a \emph{plaque at $y$}, and denote it $\mathcal{L}_y$.

 Recall from proposition \ref{prop:simult_linear} that there is $\hat{y} \in \pi^{-1}(y)$ with respect to which the full isotropy $\hat{S}_y$ is contained in $G_0$; we may further assume that a maximal unipotent subgroup is contained in $e^{\lieu_+}$.  After conjugating by $g \in P \backslash Q$, none of these unipotents will be in $e^{\lieu_+}$.  Thus a plaque $\mathcal{L}_y$ is the projection of an integral submanifold of $\hat{\mathcal{D}}$.  Any two plaques at $y$ intersect in an open subset of each containing $y$---that is, they have the same germ.
  \end{remark}

\subsection{Too many closed isotropic orbits --- contradiction}
\label{sec:final_contradiction}

Note that isotropic orbits are necessarily 1-dimensional and by proposition \ref{prop:1d_implies_closed} necessarily closed.

Although all holonomy sequences of mixed type correspond to isotropic orbits in the limit, and there are infinitely-many such holonomy sequences by condition (\ref{eqn:negation_finite_volume}), it does not directly follow that there are infinitely-many isotropic orbits, because these holonomy sequences may not be pointwise.  Instead, we make use of the following functions: fix a basis $\{ X_1, \ldots, X_d \}$ of $\lieh$ and a metric $g$ in the given conformal class on $M$, and define the $C^\omega$ function
$$ \varphi(x) = \sum_i g_x(X_i, X_i) \qquad x \in M$$
Note that, because there are no $H$-orbits in $M$ on which the restriction of $g$ has Lorentzian signature, $\varphi(x) = 0$ if and only if $H.x$ is isotropic (and therefore 1-dimensional and closed).  Denote $\Sigma = \varphi^{-1}(0)$.
Next, define an $H$-invariant function
$$ \chi(x) = \mbox{min} \{ \varphi(y) \ : \ y \in \overline{H.x} \}$$ 
It follows from continuity of $\varphi$ that $\chi$ is upper semicontinuous.  There is $C > 0$ such that $\varphi \leq C$.  Then the sets $S_n = \chi^{-1}([1/n,C])$, $n \in \BN$ are closed.  If one of them has nonempty interior, then it contains an $H$-invariant nonempty open set, and
$\varphi^{-n/2} \mbox{vol}_g$, where $n = \dim M$,  restricts to a finite, $H$-invariant volume here.  We may assume by proposition \ref{prop:proof_with_finite_volume} that this is not the case.  Then $\cup_n S_n$ has dense complement by the Baire Category Theorem; this complement is $S_0 = \chi^{-1}(0)$.
Because $\Omega_{\mathcal{F}}$ is open and dense, there is a dense set 
 $$ I = \{ x \in \Omega_{\mathcal{F}} \ : \ \overline{H.x} \ \mbox{contains an isotropic orbit} \}$$
 


The compact, analytic set $\Sigma$ has finitely-many components.  Because analytic sets are stratified by regular, semianalytic submanifolds (see \cite{bierstone.milman}), 
points in each component can be connected by piecewise analytic paths.
Let $\alpha$ be a $C^\omega$ path in $\Sigma$ defined on $(- \epsilon, \epsilon)$.
We will construct a $C^\omega$ lift of $\alpha$ to $\overline{\mathcal{R}}$, along which the vector fields of $\lieh$ evaluate under $\omega$ to a special subspace of $E_1^\perp + \lieq$.  First, since $\lieh(\alpha(t))$ is an isotropic line for all $t$, there is a $C^\omega$ lift $\tilde{\alpha}$ of $\alpha$ to $M' = \widehat{M}/P^+$ such that $\lieh(\alpha(t))$ corresponds to $\BR E_1$ in the conformal frame given by $\tilde{\alpha}(t)$.  The path $\tilde{\alpha}$ has image in $\overline{\mathcal{R}'}$.  

Next let $\hat{\alpha}$ be any $C^\omega$ lift of $\tilde{\alpha}$ to $\widehat{M}$.
As the isotropy at each $\alpha(t)$ contains nontrivial linearizable, unipotent subgroups of dimension $k$, where $H \cong \BR^k \times L$ (see the discussion at the beginning of section \ref{sec:extension_closed_isotropic}), there is a $C^\omega$ path $u(t) \in  \lieu_+$ such that $0 \neq u(t) \in \omega_{\hat{\alpha}(t)}(\lieh)$ modulo $\liep^+$ for all $t$.  The adjoint action of $Q_0$ on $\lieu_+\backslash \{ 0 \}$ is transitive.  For a fixed nonzero $u \in \lieu_+$, there is a section of this action over the path $(u(t),u)$, yielding a $C^\omega$ path $q(t)$ in $Q_0 < Q$ such that, after replacing $\hat{\alpha}(t)$ with $\hat{\alpha}(t).q(t)$, we have $u \in \omega_{\hat{\alpha}(t)}(\lieh)$ modulo $\liep^+$ for all $t \in (- \epsilon,\epsilon)$.
For each $t$, there is a unique lift $\tilde{u}(t)$ of $u$ to $\omega_{\hat{\alpha}(t)}(\lieh)$, which is conjugate in $P$ to $u$, again by linearizability of the isotropy.  The component $\tilde{u}(t)^{(1)}$ is in the image of $\ad u$ for all $t$ (see the proof of proposition \ref{prop:simult_linear}).  Again a section of the action $ \liep^+ \times \lieu_+ \rightarrow \liep^+$ over $(u,\tilde{u}(t)^{(1)})$ gives a $C^\omega$ path $\xi(t) \in \liep^+$ such that $[\xi(t),\tilde{u}(t)] = \tilde{u}(t)^{(1)}$.  Using $\xi(t)$, we can replace $\hat{\alpha}$ with a $C^\omega$ path such that $\omega_{\hat{\alpha}(t)}(\lieh)$ contains $u$ for all $t$.  

Proposition \ref{prop:extension_of_R} implies $\hat{\alpha}(t) \in \overline{\mathcal{R}}$ for each $t$.  By construction, for all $Y \in \lieh$, the evaluation $\omega_{\hat{\alpha}(t)}(Y) = \hat{Y}_t$ is in $\BR E_1$ modulo $\liep$ and commutes with $u$.  Because $u$ is in the root space $\lieg_\beta$, the $\liea$-component of $\hat{Y}^{(0)}_t$ must be in the kernel of $\beta$.  By lemma \ref{lemma:no_dilation_at_limit} and the fact that $\lieh(x)$ is degenerate for all $x \in M$, the $\liea$-component of $\hat{Y}^{(0)}_t$ must be in the kernel of $\alpha$, as well, so it is trivial.  Therefore, for all $t \in (- \epsilon,\epsilon)$ and all $Y \in \lieh$,
\begin{equation}
  \label{eqn.alpha.subspace}
  \hat{Y}_t \in \BR E_1 + \lieu_+ + \mathfrak{c}_{\liem}(u) + \mathfrak{c}_{\lieq_1}(u)
  \end{equation}
where $\mathfrak{c}_{\liem}(u)$ and $\mathfrak{c}_{\lieq_1}(u)$ are the centralizers of $u$ in $\liem$ and $\lieq_1$, respectively.

Let $\mathcal{U}$ be a neighborhood of $0$ in $E_1^\perp + \lieq$ such that $\exp_{\hat{\alpha}(t)}(\mathcal{U})$ is contained in an integral leaf $\widehat{\mathcal{L}}_{\hat{\alpha}(t)}$ of $\widehat{\mathcal{D}}$ for each $t$, as guaranteed by proposition \ref{prop:R_extension_leaf}.  The projections $\pi(\widehat{\mathcal{L}}_{\hat{\alpha}(t)})$ are plaques $\mathcal{L}_{\alpha(t)}$.
Let $A_t = \omega(\hat{\alpha}'(t))$.  Suppose there is $t_0$ such that $A_{t_0}^{(-1)} \notin E_1^\perp$.  Let $X \in \mathfrak{stab}_\lieh(\alpha(t_0))$ have isotropy $u$ with respect to $\hat{\alpha}(t_0)$.
We will compute the first-order Taylor approximation in $t$ at $t_0$ of $\omega_{\hat{\alpha}(t)}(X) = \hat{X}_t$.
Because $L_X \omega(\hat{\alpha}') = 0$, formula (\ref{eqn:curvature_formula}) gives
$$ \Omega_{\hat{\alpha}(t)}(X,\hat{\alpha}') = - \hat{\alpha}'.\omega(X) + [\hat{X},\omega(\hat{\alpha}')]_{\hat{\alpha}(t)}$$
The left-hand side vanishes at $t = t_0$, so 
$$ \left. \DDt \right|_{t_0} \omega_{\hat{\alpha}(t)}(X) = [ u, A_{t_0}]  \equiv [u, A^{(-1)}_{t_0} ] \ \mbox{mod } \liep$$
Since $A^{(-1)}_{t_0}$ is transverse to $E_1^\perp$, the bracket on the right-hand side has the form $V_B + aE_1$ for a nonzero vector $V_B$ in the spacelike subspace $E_1^\perp \cap E_n^\perp$, and $a \in \BR$.
Now 
$$ \omega_{\hat{\alpha}(t)}^{(-1)}(X) = (t-t_0) (V_B + a E_1 + R(t)) \qquad \lim_{t \rightarrow {t_0}} R(t) = 0$$
which contradicts (\ref{eqn.alpha.subspace}).

Next suppose $A^{(-1)}_{t} \in E_1^\perp$ for all $t$, but $A^{(0)}_{t_0} \notin \lieq_0$.  Then
$$ \left. \DDt \right|_{t_0} \omega_{\hat{\alpha}(t)}(X) \equiv  [ u, A^{(-1)}_{t_0} +A^{(0)}_{t_0} ] \ \mbox{mod } \liep^+$$
For a nonzero element $v \in \lieg_{-\beta}$, the brackets $[v,u]$ and $[u,[v,u]]$ are not zero (see lemma \ref{lem:bracket_nondegeneracy}).  By assumption $A^{(0)}_{t_0} = V_B + V'$ with $0 \neq V_B \in \lieg_{-\beta}$ and $V' \in \lieq_0$.  Then  
$$\omega_{\hat{\alpha}(t)}^{(0)}(X) = (t-t_0) (W_B + W' + R(t)) \qquad \lim_{t \rightarrow {t_0}} R(t) = 0$$
where $W_B \in (\liea + \liem)\backslash \mathfrak{c}_{\liea + \liem}(u)$, and $W' \in \lieu_+$, again contradicting (\ref{eqn.alpha.subspace}).

Finally, suppose $A^{(-1)}_{t} \in E_1^\perp$ and $A^{(0)}_t \in \lieq_0$ for all $t$, but $A^{(1)}_{t_0} \notin \lieq_1$.
For $0 \neq v \in \lieg_{\alpha - \beta}$,
the brackets $[v,u]$ and $[u,[v,u]]$ are not zero.  Then
$$\omega_{\hat{\alpha}(t)}^{(1)}(X) = (t-t_0) (\xi_B + \xi' + R(t)) \qquad \lim_{t \rightarrow {t_0}} R(t) = 0$$
where $\xi_B \in \lieq_1 \backslash \mathfrak{c}_{\lieq_1}(u)$ and $\xi' \in \BR ^tE_1 \mathbb{I}$, again contradicting (\ref{eqn.alpha.subspace}).

We have shown that any $C^\omega$ path $\alpha$ in $\Sigma$ lifts to a $C^\omega$ path $\hat{\alpha}$ in $\overline{\mathcal{R}}$ and tangent to $E_1^\perp + \lieq$, and thus contained in an integral submanifold $\hat{\mathcal{L}}$ of $\hat{\mathcal{D}}$.   For all $t \in (- \epsilon, \epsilon)$,  the projection $\pi(\hat{\mathcal{L}})$ intersects the plaques at $\alpha(t)$ in an open neighborhood of $\alpha(t)$ in each.  By the compact, analytic structure of $\Sigma$ referenced above, we conclude that it can be covered by finitely many plaques, say $\mathcal{L}_1, \ldots, \mathcal{L}_k$.  By remark \ref{rmk:mutually.open}, the leaves of all points in $I$ have mutually open intersection with at least one $\mathcal{L}_i$.
As in corollary \ref{cor:isotropic_in_leaf}, to each $\mathcal{L}_i$ can be associated  a nonzero element $X_i \in \lieh$, in the isotropy of a point of $\Sigma$, which is isotropic in each leaf having mutually open intersection with $\mathcal{L}_i$.  By density of $I$, some $X_i$ will be isotropic in a nonempty open subset of $M$ and thus in all of $M$.  We arrive at a contradiction with \cite{frances.ccvf}.

\bibliographystyle{amsalpha}
\bibliography{karinsrefs}

\bigskip

\begin{tabular}{lll}
 Karin Melnick  & \quad\qquad & Vincent Pecastaing \\
Department of Mathematics & \quad\qquad & Laboratoire J.A. Dieudonn\'e \\
4176 Campus Drive & \qquad \qquad &  UMR CNRS 7351 \\
University of Maryland & \quad\qquad & Universit\'e C\^ote d'Azur, Parc Valrose \\
College Park, MD 20742 &\quad \qquad &   06108 Nice, Cedex 2 \\
USA &\quad \qquad &  France \\
\texttt{karin@math.umd.edu} &\quad \qquad & \texttt{vincent.pecastaing@unice.fr} 
\end{tabular}

\end{document}